\documentclass[10pt]{amsart}

\usepackage[utf8]{inputenc}
\usepackage[english]{babel}

\usepackage[charter]{mathdesign}
\usepackage[T1]{fontenc}
\usepackage{bbm}
\usepackage{esint}

\usepackage{mathtools}
\usepackage{etoolbox}
\usepackage[unicode,plainpages=false,hypertexnames=true]{hyperref}

\newcommand{\m}[1]{
\ifdefequal{#1}{1}
{\mathbbm{#1}}
{\mathbb{#1}}
}

\newcommand{\q}[1]{\mathcal{#1}}

\newcommand{\mr}[1]{\mathrm{#1}}

\newcommand{\ds}{\displaystyle}
\newcommand{\be}{\begin{gather}}
\newcommand{\ee}{\end{gather}}

\newcommand{\e}{\varepsilon}

\DeclareMathOperator{\Span}{\text{Span}}

\DeclareMathOperator{\Supp}{\text{Supp}}
\DeclareMathOperator{\diam}{\text{diam}}
\DeclareMathOperator{\Card}{\text{Card}}

\DeclareMathOperator{\Id}{\text{Id}}
\DeclareMathOperator{\Div}{\text{div}}
\DeclareMathOperator{\Tr}{\text{Tr}}

\DeclareMathOperator{\loc}{loc}

\renewcommand{\le}{\leqslant}
\renewcommand{\ge}{\geqslant}

\theoremstyle{plain}
\newtheorem*{thm*}{Theorem}
\newtheorem{thm}{Theorem}[section]

\newtheorem{prop}[thm]{Proposition}
\newtheorem{cor}[thm]{Corollary}
\newtheorem{lem}[thm]{Lemma}

\theoremstyle{definition}
\newtheorem{defi}[thm]{Definition}

\theoremstyle{remark}
\newtheorem{nb}{Remark}
\newtheorem{claim}[thm]{Claim}

\numberwithin{equation}{section}

\begin{document}

\title{Limiting motion for the parabolic Ginzburg-Landau equation with infinite energy data}

\date{\today}

\author{Delphine Côte}
\author{Raphaël Côte}

\subjclass[2010]{35K45, 35C15, 35Q56}

\keywords{Ginzburg-Landau, vortex, mean curvature flow, infinite energy}

\thanks{D.C. would like to thank Fabrice Bethuel for introducing her to this problem, and his constant support and encouragement. R.C. gratefully acknowledges support from the Agence Nationale de la Recherche under the contract MAToS ANR-14-CE25-0009-0.}

\maketitle

\begin{abstract}
We study a class of solutions to the parabolic Ginzburg-Landau equation in dimension 2 or higher, with ill-prepared infinite energy initial data. We show that, asymptotically, vorticity evolves according to motion by mean curvature in Brakke's weak formulation. Then, we prove that in the plane, point vortices do not move in the original time scale. These results extend the work of Bethuel, Orlandi and Smets \cite{Be1,Be2} for infinite energy data; they allow to consider the point vortices on a lattice (in dimension 2), or filament vortices of infinite length (in dimension 3).
\end{abstract}

\section{Introduction}

We consider the parabolic Ginzburg-Landau equation for complex functions $u_\e:\m{R}^d \times [0,+\infty) \rightarrow \m{C}$ 

\begin{equation}
\label{pgl} \tag{PGL$_\e$}
\begin{cases}
\displaystyle \frac{\partial u_\e}{\partial t} - \Delta u_\e =  \frac{1}{\e^2} u_\e(1-|u_\e|^2) \qquad  & \text{on } \m{R}^d \times (0,+\infty), \\
u_\e(x,0)  =  u_\e^0(x) & \text{for } x \in \m R^d,
\end{cases}
\end{equation}

and its associated energy 
\[ E_\e(w)=\int_{\m{R}^d} e_\e(w)(x) dx =\int_{\m{R}^d} \left( \frac{|\nabla w(x)|^2}{2} + V_\e(w(x)) \right) dx \quad \text{for} \quad w:\m{R}^d \rightarrow \m{C}, \]
where $V_\e$ denotes the non-convex double well potential and $e_\e$ is the energy density:
\[V_\e(w)=\frac{(1-|w|^2)^2}{4\e^2},  \quad \text{and} \quad e_\e(w) =  \frac{|\nabla w|^2}{2} + V_\e(w) .\]
It is a classical result that an initial data $u_0 \in L^\infty \cap \dot{H}^1$ yields a global in time solution $u(t) \in \q C([0,\infty), L^\infty \cap \dot{H}^1)$.

The Ginzburg-Landau equation (PGL$_1$) in the plane ($d=2$) admits vortex solutions of the form
\[\Psi(x,t)=\Psi(x)=U_\ell(r) \exp(i\ell \theta), \ \ell \in \m Z^*, \ U_\ell(0)=0,\ U_\ell(+\infty)=1\]
where $(r,\theta)$ corresponds to the polar coordinates in $\m{R}^2$ (by scaling, \eqref{pgl} admits stationary vortex solutions as well). Such functions $\Psi$ define complex planar vector fields whose zeros are called vortices (of order $\ell$, also called $\ell$-vortices). Vortices solutions arise naturally in Physics applications, and it is an important question to study the asymptotic analysis, as the parameter $\e$ goes to zero, of solutions to \eqref{pgl}.

We must stress out the fact that a single vortex does not belong to $\dot{H}^1(\m R^d)$ (for $d =2$, $|\nabla \Psi|(r,\theta) \sim d/r$, $\Psi \notin L^2(\m R^2)$ either). To overcome this problem, an easy way out is to consider configurations of multiple vortices
 where the sum of degrees of the vortices is equal to zero. In that case, the initial data belongs to the space of energy $ L^\infty \cap \dot{H}^1$, and we talk about well-prepared data, see for example Jerrard and Soner \cite{Je}, Lin \cite{Lin}, Sandier and Serfaty \cite{San}, and Spirn \cite{Spirn}.

One way to relax this condition was done in the seminal works by Bethuel, Orlandi and Smets \cite{Be1,Be2,Be4,Be5}: they consider \eqref{pgl} with $u_\e:\m{R}^d \times [0,+\infty) \rightarrow \m{C}$  and assume that the initial data $u_\e^0$ is in the energy space and verifies the bound
\begin{gather} 
\label{eq:energy_bound}
E_\e(u_\e^0) \le M_0 |\ln \e|
\end{gather}
where $M_0$ is a fixed positive constant. Observe that this condition encompasses large data, and almost gets rid of any well preparedness assumption. This only limitation can be seen as follows in dimension $d=2$ (where vortices are points): \eqref{eq:energy_bound} allows general sum of vortices, which are balanced by adding ``vortices at infinity'' (where the center of the vortices goes to spatial infinity as $\e \to 0$): in that case, for each $\e>0$, the initial data is of finite energy, but the limiting configuration can be any configuration of finitely many vortices.

\medskip

The main emphasis of \cite{Be1}, valid in any dimension, is placed on the asymptotic limits of the Radon measures $\mu_\e$ defined on $\m{R}^d \times [0,+\infty)$, and their time slices $\mu_\e^t$ defined on $\m{R}^d \times \{t\}$, by
\begin{gather}
\mu_\e(x,t)=\frac{e_\e(u_\e)(x,t)}{|\ln \e|}\ dxdt, \quad \text{and} \quad \mu_\e^t(x)=\frac{e_\e(u_\e)(x,t)}{|\ln \e|}\ dx, 
\end{gather}
so that $\mu_\e= \mu_\e^t dt$. 

The bound on the energy gives that, up to a subsequence $\e_m \rightarrow 0$, there exists a Radon measure $\mu_* = \mu_*^{t} dt $ defined on $\m{R}^d \times [0,+\infty)$ such that
\[ \mu_\e \rightharpoonup \mu_{*}, \quad \text{and} \quad \mu_\e^t \rightharpoonup \mu_{*}^{t} \]
as measures on $\m{R}^d \times [0,+\infty)$ and $\m{R}^d \times \{t\}$ for all $t \ge 0$, respectively (see \cite[Lemma 1]{Be1} and \cite{Ilm94}).  The purpose of \cite{Be1} is to describe the properties of the measures $\mu_{*}^t$: the main result is that asymptotically, the vorticity $\mu_*^t$ evolves according to motion by mean curvature in Brakke's weak formulation.

%\begin{thm} [\cite{Be1}]
%\label{ta1}
%Let $u_\e$ be a solution of \eqref{pgl} satisfying the initial condition $E_\e(u_\e^0) \le M_0 |\ln \e|$.
%There exist a subset $\Sigma_\mu$ in $\m{R}^d \times (0,+\infty)$, and a smooth real-valued function $\Phi_{*}$ defined on $\m{R}^d \times (0,+\infty)$ such that the following properties hold.

%\begin{enumerate}
%\item  $\Sigma_\mu$ is closed in $\m{R}^d \times (0,+\infty)$ and for any compact subset $\mathcal{K} \in \m{R}^d \times (0,+\infty) \setminus \Sigma_\mu$,
%\[|u_\e(x,t)| \rightarrow 1 \text{ uniformly on } \mathcal{K} \text{ as } \e \rightarrow 0. \]

%\item For any $t>0,\ \Sigma_\mu^t = \Sigma_\mu \cap \m{R}^d \times \{t\}$ verifies
%\[ \mathcal{H}^{d-2} (\Sigma_\mu^t) \le K M_0. \]
 
%\item The function $\Phi_{*}$ verifies the heat equation on $\m{R}^d \times (0,+\infty)$.

%\item For each $t>0$, the measure $\mu_{*}^t$ can be exactly decomposed as 
%\[ \mu_{*}^t = |\nabla \Phi_{*}|^2(.,t) \mathcal{H}^{n}\ +\ \Theta_{*}(x,t)\mathcal{H}^{d-2}\llcorner \Sigma_\mu^t\]
%where $\Theta_{*}(.,t)$ is a bounded function.

%\item There exists a positive function $\eta$ defined on $(0,+\infty)$ such that, for almost every $t>0$, the set $\Sigma_\mu^t$ is $(n-2)$-rectifiable and
%\[ \Theta_{*}(x,t)=\Theta_{d-2}(\mu_{*}^t,x)=\lim_{r \rightarrow 0} \frac{\mu_{*}^t(B(x,r))}{\omega_{d-2}r^{d-2}} \ge \eta(t),  \]
%for $\mathcal{H}^{d-2}$ a.e. $x \in \Sigma_\mu^t$.
 
%\end{enumerate}
%\end{thm} 

In dimension $d=2$, though, the vorticity $\mu^t_*$ is supported on a finite set of points (the vortices). One can actually compute that the energy of a $\ell$-vortex is roughly $\pi \ell^2 |\ln \e|$: the above bound \eqref{eq:energy_bound} implies that only a finite number of vortices can be created (at most $M_0/\pi$).
However the mean curvature flow for discrete points is trivial, they do not move. Therefore, in order to see the vortices evolve, one needs to consider a different regime, where time is adequately rescaled by a factor $|\ln \e|$. This is done by Bethuel, Orlandi and Smets in \cite{Be2,Be4,Be5}: they describe completely the asymptotics, and analyze precisely the dissipation times where collision or splitting of vortices occur. Again, the only assumption is the bound \eqref{eq:energy_bound} on the initial data $u_\e^0$ (and thus  $u_\e^0$ is in the energy space).

\bigskip

Our goal in this paper is to extend the results in \cite{Be1}, by relaxing the global energy bound \eqref{eq:energy_bound} to a local one. More precisely, we study families of solutions to \eqref{pgl} whose initial data $u_\e^0$ satisfy the following assumptions, for some constant $M_0 >0$: 
\begin{equation}
\label{in}  \tag{$H_1(M_0)$}
\begin{cases}
\forall \e>0, \quad u_\e^0 \in L^\infty(\m R^d), \\
\ds \forall \e>0, \ \forall x \in \m{R}^d, \quad \int_{B(x,1)} e_{\e}(u_{\e}^0)(y) dy \le M_0|\ln \e|.
\end{cases}
\end{equation}
Observe that \cite[Theorem 2]{DCot15} shows the existence of a unique solution $u_\e(t)$ to \eqref{pgl} with initial data $u_\e^0$ which is globally well defined for positive times. The crucial point is that \eqref{in} is a property which propagates along time, if one allows $M_0$ to depend on time. %(see Proposition \ref{p1p} below), so that $u_\e(t)$ also satisfies $(H_1(C(1+t)^d M_0))$. 
We also refer to \cite{BaOrWe09} where the Ginzburg-Landau functional is studied under a local energy bound in $|\ln \e|$ (and a $\Gamma$-convergence result is obtained).

\bigskip

Let us emphasize that the analysis here is done in the original time scale and not in the accelerated time scale (which is relevant for dimension $d=2$ only).

Our main results are the following. We define limiting energy $\mu_*^t$ and construct the vorticity set $\Sigma_\mu^t$ (and prove some regularity properties). Then we consider the concentrated energy $\nu_*^t$ on $\Sigma_\mu^t$ and show that it evolves under the mean curvature flow in a weak formulation. Finally we focus in dimension $d=2$, and show that in this case $\Sigma_\mu^t$ is made of a finite set of points which do not move.

We start with the description of the vorticity set and the decomposition of the asymptotic energy density.

\begin{thm}%{(similar to Theorem A, Part I, \cite{Be1}).} 
\label{thA}
Let $(u_\e)_{\e \in (0,1)}$ be a family of solutions of \eqref{pgl} such that their initial conditions $u_\e^0$ satisfy \eqref{in}.

Then there exist a subset $\Sigma_\mu$ in $\m{R}^d \times (0,+\infty)$, and a smooth real-valued function $\Phi_{*}$ defined on $\m{R}^d \times (0,+\infty)$ such that the following properties hold.

\begin{enumerate}
\label{tteo1}
\item  $\Sigma_\mu$ is closed in $\m{R}^d \times (0,+\infty)$ and for any compact subset $K \in \m{R}^d \times (0,+\infty) \setminus \Sigma_\mu$,
\[|u_\e(x,t)| \rightarrow 1 \text{ uniformly on } K \text{ as } \e \rightarrow 0. \]

\item For any $t>0$ and $x \in \m R^d$, $\Sigma_\mu^t = \Sigma_\mu \cap \m{R}^d \times \{t\}$ verifies
\[  \mathcal{H}^{d-2} (\Sigma_\mu^t \cap B(x,1)) \le C M_0. \]
 
\item The function $\Phi_{*}$ verifies the heat equation on $\m{R}^d \times (0,+\infty)$.

\item For each $t>0$, the measure $\mu_{*}^t$ can be exactly decomposed as 
\begin{equation}
\label{x60}
\mu_{*}^t = |\nabla \Phi_{*}|^2(.,t) \mathcal{H}^{d}\ +\ \Theta_{*}(x,t)\mathcal{H}^{d-2}\llcorner \Sigma_\mu^t 
\end{equation}
where $\ds \Theta_{*}(.,t) : =\lim_{r \rightarrow 0} \frac{\mu_{*}^t(B(x,r))}{\omega_{d-2}r^{d-2}}  $ is a bounded function.

\item There exists a positive function $\eta$ defined on $(0,+\infty)$ such that, for almost every $t>0$, the set $\Sigma_\mu^t$ is $(d-2)$-rectifiable and
\[ \Theta_{*}(x,t)%=\Theta_{d-2}(\mu_{*}^t,x)
 \ge \eta(t), \quad \text{for } \mathcal{H}^{d-2} \text{ a.e. } x \in \Sigma_\mu^t. \]
\end{enumerate}
\end{thm} 

In view of the decomposition (\ref{x60}), $\mu_*^t$ can be split into two parts: a diffuse part $|\nabla \Phi_{*}|^2$, and a concentrated part 
\begin{gather}
\label{def:nu}
\nu_*^t := \Theta_{*}(x,t)\mathcal{H}^{d-2} \llcorner \Sigma_\mu^t.
\end{gather}
By (3), the diffuse part is governed by the heat equation. Our next theorem focuses on the evolution of the concentrated part $\nu_*^t$.

\begin{thm}
\label{thB}
The family $(\nu_*^t)_{t>0}$ is a mean curvature flow in the sense of Brakke (see Section \ref{sec:mean_curvature} for definitions).
\end{thm}

%\bigskip

For our last result, we focus on dimension $d=2$,  where vortices are points. We show that these vortex points do not move in the original time scale.

\begin{thm}%{(similar to Theorem 3.1 of \cite{Be2}).} 
\label{theo3.1}
Let $d=2$. Then $\Sigma_\mu^t$ is a (countable) discrete set of $\m R^2$, which we can enumerate $\Sigma_\mu^t = \{ b_i(t) \mid i \in \m N \}$. Also, for all $x \in \m R^2$,
\[ \Card ( \Sigma_\mu^t \cap B (x,1)) \le C M_0, \]
and the points $b_i$ do not move, i.e.
\[ \forall t>0, \quad b_i(t)=b_i, \]
and $\ds\nu_*(t) = \sum_{i=1}^{+\infty} \sigma_i(t) \delta_{b_i}$, where the functions $\sigma_i(t)$ are non-increasing.
\end{thm}
%We emphasize that \eqref{in} allows us to consider configurations with a countable number of vortices (forming a discrete set), for example a vortex lattice.

Local bounds on the energy, that is, assumption \eqref{in}, make the set of admissible initial data more natural. We can consider general vortex configurations in dimension 2, without adding a vortex at infinity to balance it (so as make the total sum of the vortices' degrees equal to 0 for each $\e>0$); important physically relevant examples encompass  point vortices on a infinite lattice (dimension 2) or general vortex filament families (with possibly infinite length) in dimension 3.
 
Another  striking difference betwen the global bound \eqref{eq:energy_bound} and \eqref{in} is the following. In dimension $d \ge 3$, \eqref{eq:energy_bound} implies that after some finite time, the vorticity vanishes, that is $\Sigma_\mu^t \subset \m R^d \times [0,T]$ for some $T$ depending on $M_0$ (see \cite[Proposition 3]{Be1}). This is now longer the case under \eqref{in}, which we believe is a more physically accurate phenomenon.

\bigskip

The assumption that $u_\e^0 \in L^\infty(\m R^d)$  seems technical (because it comes without bounds in term of $\e$), but uneasy to get rid of: the main reason being the lack of a suitable local well posedness in the space of functions with uniformly locally finite energy. Indeed, the closest result in this respect (besides \cite{DCot15} in the $L^\infty$ setting), is the work by Ginibre and Velo \cite{GiVe97}, whose results do not apply to the Ginzburg-Landau nonlinearity, and and even then, their control of the solution at time 0 seems too weak.

\bigskip

These results are an extension of the works of Bethuel, Orlandi and Smets \cite{Be1,Be2}, and the proofs are strongly inspired by these: Theorems \ref{thA} and \ref{thB} by Theorems A and B in \cite{Be1} and Theorem \ref{theo3.1} by Theorem 3.1 in \cite{Be2}. Our main contribution will be to systematically improve their estimates, in order to solve the new problems raised by our dealing with infinite energy solutions of \eqref{pgl}; especially to make sense of a monotonicity property, which is at the heart of the proofs in \cite{Be1,Be2}. We will also need to derive pointwise estimates on $u_\e$ and $L^2$ space time estimates on $\partial_t u_\e$: in the finite energy setting, it appears as the flux of the energy, but this is no longer the case in our context. A leitmotiv of this paper is that, although many of the bounds in \cite{Be1,Be2} are global in time and/or space, their arguments are in fact local in nature, and so can be adapted under the hypothesis \eqref{in}.

In the proofs, we will focus on the differences brought by our change of context, and only sketch the arguments when they are similar to that of \cite{Be1,Be2}.

\bigskip

A natural question is now to focus on dimension $d=2$ and to study the dynamics of vortices in the accelerated time frame, as it is done in \cite{Be2,Be4,Be5}. We believe that the arguments in these works could be extended under the hypothesis \eqref{in}.
%but with an adequate change in the statements. i
However one has to make a meaningful sense of the limiting equation, (a pseudo gradient flow of the Kirchoff-Onsager functional involved), as it is not obviously well posed for a countable infinite number of points.
%, but for example the limiting equation involving the Kirchoff-Onsager flow is not well behaved for an infinite number of vortices. 
We leave these perspectives to subsequent research.

\bigskip

This paper is organized as follows. In Section 2, we study \eqref{pgl} and prove our main PDE tool, namely the clearing-out (stated in Theorem \ref{th1}). In Section 3, we define the limiting measure and the vorticity set $\Sigma_\mu$: we prove in particular regularity properties of $\Sigma_\mu^t$ and complete the proof of Theorem \ref{thA}. In Section 4, we show Theorem \ref{thB}, that is, the singular part $\nu_*^t$ follows the mean curvature flow in Brakke's weak formulation. Finally, in Section 5, we focus on dimension $d=2$ and prove Theorem \ref{theo3.1}.

\section[PDE Analysis of $($PGL$_\e)$]{PDE Analysis of \texorpdfstring{\eqref{pgl}}{PGL epsilon}}
\label{sec:cl-out}

\subsection{Statement of the main results on \texorpdfstring{\eqref{pgl}}{PGL epsilon}}

In this section, we work on \eqref{pgl}, that is with smooth solutions $u_\e$, where the parameter $\e$, although small, is \emph{positive}. We derive a number of properties on $u_\e$, which enter directly in the proof of the clearing- out Theorem \ref{th6} at the limit $\e \to 0$. Heuristically, the clearing-out means that if there is not enough energy in some region of space, then \emph{at a later time}, vortices can not be created in that region.

Let us first state the main results which will be proved in this section.

\subsubsection{Clearing-out and annihilation for vorticity}

The two main ingredients in the proof of the clearing-out are a clearing-out theorem for vorticity, as well as some precise pointwise energy bounds. Throughout this section, we suppose that $0<\e<1$. 
We define the vorticity set $\mathcal{V}_\e$ as
\[\mathcal{V}_\e=\left\{(x,t) \in \m{R}^d \times (0,+\infty) : \ |u_\e(x,t)|\le \frac{1}{2}\right\}.\]
Here is the precise statement.

\begin{thm}%{(similar to Theorem 1, Part I, \cite{Be1}).} 
\label{th1}
Let $0<\e<1$, $u_\e^0 \in L^{\infty}(\m{R}^d)$ and $u_\e$ be the associated solution of \eqref{pgl}. Let $\sigma>0$ be given. There exists $\eta_1=\eta_1(\sigma)>0$ depending only on the dimension $n$ and on $\sigma$ such that if
\begin{gather} \label{eq:en_conc}
\int_{\m{R}^d} e_\e(u_\e^0) \exp \left( -\frac{|x|^2}{4} \right) \le \eta_1 |\ln \e|,
\end{gather}
then \[ |u_\e(0,1)| \ge 1-\sigma. \]
\end{thm}

Notice that we only assume that $u^0_\e \in L^{\infty}$, whereas in \cite[Theorem 1]{Be1} the assumption was $E_\e(u_\e) <+\infty$; of course this latter bound will not be available for Theorem \ref{thA}. Observe that $L^\infty$ prevents us to use a density argument, and even more so as we are interested in non zero degree initial data. 
Also, the asumption \eqref{eq:en_conc} is not enough by itself to ensure existence and uniqueness of the solution to \eqref{pgl}, but $L^\infty$ is suitable (see \cite{DCot15}). 

Nonetheless the proof follows closely that of \cite[Part I]{Be1}, and we will only emphasizes the differences.

\bigskip

The proof of Theorem \ref{th1} requires a number of tools, in particular 
\begin{itemize}
\item the monotonicity formula, first derived by Struwe \cite{Str} in the case of the heat-flow for harmonic maps
\item a localizing property for the energy inspired by Lin and Rivière \cite{Lin2}
\item refined Jacobian estimates due to Jerrard and Soner \cite{Je2} 
\item techniques first developed for the stationary equation (for example \cite{BeBrHe94,BeBrOr01,BeBrOr05}).
\end{itemize}

Equation \eqref{pgl} has standard scaling properties. If $u_\e$ is a solution to \eqref{pgl}, then for $R>0$ the function $(x,t) \mapsto u_\e(Rx,R^2t)$ is a solution to $(PGL)_{R^{-1} \e}$, to which we may then apply Theorem \ref{th1}.
%
%For this purpose, we define, for $z_{*}=(x_*,t_*) \in \m{R}^d \times (0,+\infty)$ the scaled weighted energy, taken at time $t=t_*$,
%\begin{gather} \label{def:Ew}
% E_{w,\e}(u_\e,z_*,R) = E_{w}(z_*,R) := \frac{1}{R^{d-2}} \int_{\m{R}^d} e_\e(u_\e)(x,t_*) \exp \left(-\frac{|x-x_*|^2}{4R^2} \right) dx.
% \edn{gather}
As an immediate consequence of Theorem \ref{th1} and scaling, we have the following result.

\begin{prop}%{(similar to Proposition 1, Part I, \cite{Be1}).} 
\label{prop1}
Let $T>0,\ x_T \in \m{R}^d$, and set $z_T=(x_T,T)$. Let $u_\e^0 \in L^{\infty}(\m{R}^d)$ and $u_\e$ be the associated solution of \eqref{pgl}. Let $R> \sqrt{2\e}$. Assume moreover
\[ \frac{1}{R^{d-2}} \int_{\m{R}^d} e_\e(u_\e)(x,T) \exp \left(-\frac{|x-x_T|^2}{4R^2} \right) dx \le \eta_1(\sigma) |\ln \e|, \]
then
\[ |u_\e(x_T,T+R^2)| \ge 1-\sigma. \]
\end{prop}

The condition in Proposition \ref{prop1} involves an integral on the whole of $\m{R}^d$. In some situations, it will be convenient to integrate on finite domains. Here is how one should localize in space the conditions.

\begin{prop}%{(similar to Proposition 2, Part I, \cite{Be1}).} 
\label{prop2}
Let $u_\e$ be a solution of \eqref{pgl} satisfying the initial data \eqref{in}.
Let $\sigma >0$ be given. Let $ T>0, \ x_T \in \m{R}^d$, and $R \in [\sqrt{2\e},1]$.
There exists a positive continuous function $\lambda$ defined on $(0,+\infty)$ such that if
\[ \tilde{\eta}(x_T,T,R) \equiv \frac{1}{R^{d-2}|\ln \e|}  \int_{B(x_T,\lambda(T) R)} e_\e(u_\e)(x,T) dx \le \frac{\eta_1(\sigma)}{2}, \]
then \[ |u_\e(x,t)| \ge 1-\sigma \quad \text{ for } t\in [T+T_0,T+T_1] \text{ and } x \in B(x_T,\frac{R}{2}),\]
where $T_0=\max \left( 2\e,(\frac{2\tilde{\eta}}{\eta_1(\sigma)})^{\frac{2}{d-2}}R^2 \right)$ ($T_0=2\e$ in dimension $d=2$), and $T_1=R^2$.

Furthermore, $\lambda$ is non increasing on $(0,1]$, and non decreasing on $[1,+\infty)$, and there exists an absolute constant $C$ (not depending on $T$) such that
\[ \forall T >0, \forall \tau \in (T/2, 2T), \quad \lambda(\tau) \le C \lambda(T). \]
\end{prop}

\begin{nb} \label{nb:no_vanishing}
Recall that in dimension $d \ge 3$, a bound on the initial energy \emph{on the whole space} \eqref{eq:energy_bound} implies that in finite time, the vorticity vanishes (i.e $\Sigma_\mu^t \subset \m R^d \times [0,T]$). It is an easy consequence of the monotonicity formula combined with Theorem \ref{th1} ($\q E_{w,\e}(x,t,\sqrt{t}) \to 0$ uniformly in $x$).

In the case of uniform \emph{local} bound on the energy $H_1(M_0)$, this result does not persist, because  the monotonity formula does not imply the vanishing of $\q E_{w,\e}$ for large times. This is one striking difference with the finite energy case.
\end{nb}

\subsubsection{Improved pointwise energy bounds}

The following result reminds of a result of Chen and Struwe \cite{Che} developped in the context of the heat flow for harmonic maps.

\begin{thm}%{(similar to Theorem 2, Part I, \cite{Be1}).} 
\label{th2}
Let $u_\e$ be a solution of \eqref{pgl} whose initial data satisfies \eqref{in}.
Let $B(x_0,R)$ be a ball in $\m{R}^d$ and $T>0, \Delta T>0$ be given. Consider the cylinder
\[ \Lambda= B(x_0,R) \times [T, T+\Delta T]. \]
There exist two constants $0<\sigma \le \frac{1}{2}$ and $\beta>0$ depending only on $d$ such that the following holds. Assume that
\[ |u_\e| \ge 1-\sigma \text{ on } \Lambda. \]
Then 
\begin{equation}
\label{eq11}
e_\e(u_\e)(x,t) \le C(\Lambda) \int_{\Lambda} e_\e(u_\e) , 
\end{equation}
for any $\ds (x,t) \in \Lambda_{\frac{1}{2}} :=B(x_0,\frac{R}{2}) \times [T+\frac{\Delta T}{4}, T+\Delta T]$. Moreover,
\[ e_\e(u_\e)= |\nabla \Phi_\e|^2 +\kappa_\e \quad\text{ in } \Lambda_{\frac{1}{2}}, \]
where the functions $\Phi_\e$ and $\kappa_\e$ are defined on $\Lambda_{\frac{1}{2}}$, and verify
\begin{gather} 
\frac{\partial \Phi_\e}{\partial t} -\Delta \Phi_\e = 0 \quad \text{ in } \Lambda_{\frac{1}{2}}, \nonumber \\
\label{eq14}
\|\kappa_\e\|_{L^{\infty}(\Lambda_{\frac{1}{2}})} \le
C(\Lambda) M_0 \e^{\beta},\quad  \| \nabla \Phi_\e\|_{L^{\infty}(\Lambda_{\frac{1}{2}})} \le C(\Lambda) M_0 | \ln \e |. 
\end{gather}
On $\Lambda$ one can write $u_\e = \rho_\e e^{i \varphi_\e}$ where $\varphi_\e$ is smooth (and $\rho_\e = |u_\e|$), and we have the bound
\begin{gather}
\label{eq:th2:3}
\| \nabla \varphi_\e - \nabla \Phi_\e \|_{L^\infty(\Lambda_{\frac{1}{2}})} \le C(\Lambda) \e^\beta.
\end{gather}
\end{thm}

Combining Proposition \ref{prop2} and Theorem \ref{th2}, we obtain the following immediate consequence.

\begin{prop}%{(similar to Proposition 4, Part I, \cite{Be1}).} 
\label{prop4}
Let $u_\e$ be a solution of \eqref{pgl} satisfying whose initial data satisfies \eqref{in}.
There exist an absolute constant $\eta_2>0$ and a positive function $\lambda$ defined on $(0,+\infty)$ such that, if for $x \in \m{R}^d,\ t>0$ and $r \in[\sqrt{2\e},1]$, we have
\[ \int_{B(x,\lambda(t) r)} e_\e(u_\e) \le \eta_2 r^{d-2}|\ln \e|, \]
then
\[ e_\e(u_\e) =|\nabla \Phi_\e|^2 + \kappa_\e \]
in $\Lambda_{\frac{1}{4}}(x,t,r) \equiv B(x,\frac{r}{4}) \times [t+\frac{15}{16}r^2,t+r^2]$, where $\Phi_\e$ and $\kappa_\e$ are as in Theorem \ref{th2}. 

In particular,
\[ \mu_\e=\frac{e_\e(u_\e)}{|\ln \e|} \le C(t,r) \quad \text{ on } \Lambda_{\frac{1}{4}}(x,t,r). \]
\end{prop}

(The constant $\eta_2$ is actually defined as $\eta_2=\eta_1(\sigma)$ where $\sigma$ is the constant in Theorem \ref{th2} and $\eta_1$ is the function defined in Proposition \ref{prop2}).

\subsubsection{Identifying the sources of non compactness}

We identified in the previous arguments a possible source of non compactness, due to oscillations in the phase. But this analysis was carried out on the complement of the vorticity set. Now $u_\e$ is likely to vanish on $\mathcal{V}_\e$, which leads to a new contribution to the energy: however, this new contribution does not correspond to a source of non compactness, as it is stated in the following theorem.

\begin{thm}%{(similar to Theorem 3, Part I, \cite{Be1}).} 
\label{th3}
Let $u_\e$ be a solution of \eqref{pgl} whose initial data satisfies \eqref{in}.
Let $K \in \m{R}^d \times (0,+\infty)$ be any compact set. There exist a real-valued function $\Phi_\e$ and a complex-valued function $w_\e$, both defined on a neighborhood of $K$, such that

\begin{enumerate}
\item $u_\e= w_\e \exp(i\Phi_\e) $ on $K$,
\item $\Phi_\e$ verifies the heat equation on $K$,
\item $|\nabla \Phi_\e(x,t)| \le C(\mathcal{K}) \sqrt{M_0|\ln \e|}$ for all $(x,t)\in K$,
\item $\|\nabla w_\e\|_{L^p(\mathcal{K})} \le C(p, \mathcal{K})$, for any $\ds 1 \le p <\frac{d+1}{d}$.
\end{enumerate}
Here, $C(K)$ and $C(p,K)$ are constants depending only on $K$, and  $K$, $p$ respectively.

\end{thm}

The proof relies on the refined Jacobian estimates of \cite{Je}. 

We stress out the fact that Theorem \ref{th3} provides an exact splitting of the energy in two different modes, that is the topological mode (the energy related to $w_\e$), and the linear mode (the energy of $\Phi_\e$): in some sense, the lack of compactness is completely locked in $\Phi_\e$.

\medskip

The remainder of this section is to provide proofs for the results described above, which will be done in section \ref{sec:2.4}; we need some preliminary considerations before.

\subsection{Pointwise estimates}

In this section, we provide pointwise parabolic estimates for $u_\e$ solution of \eqref{pgl}, which rely ultimately on a supersolution argument, i.e a variant of the maximum principle.

\begin{prop}
\label{prop1.1}
Let $u_\e^0 \in L^{\infty}(\m{R}^d)$ and $u_\e$ be the associated solution of \eqref{pgl}. Then for all $t>0$, $u_\e(t)$, $\nabla u_\e(t)$ and $\partial_t u_\e(t)$ are in $L^\infty(\m{R}^d)$. More precisely, there exists a (universal) constant $K_0>0$ such that for all $t > \e^2$ and $x \in \m R^d$,
\begin{gather}
\label{zz28}
|u_\e(x,t)| \le 2,\quad |\nabla u_\e(x,t)| \le \frac{K_0}{\e},\quad \left| \partial_t u_\e(x,t) \right| \le \frac{K_0}{\e^2}.  
\end{gather}
Also, for all $t >0$ and  $x \in \m R^d$, $|u_\e(x,t)| \le \max(1, \| u_\e^0 \|_{L^\infty})$.
\end{prop}

\begin{nb}
We emphasize that, past the time layer $t \ge \e^2$, $\| u_\e(t) \|_{L^\infty}$ is bounded independently of $u_\e^0$.
\end{nb}

\begin{proof}
We make a change of variable, setting
\[ v(x,t)=u_\e(\e x, \e^2 t),\]

so that the function $v$ satisfies
\begin{equation}
\label{aa1}
\partial_t v -\Delta v = v(1-|v|^2) \quad \text{on } \m{R}^d \times [0,+\infty).
\end{equation}
We have to prove that, for $t \ge 1$ and $x \in \m{R}^d$, 
\[ |v(x,t)| \le 2, \quad |\nabla v(x,t)| \le K_0, \quad  \left| \partial_t v(x,t) \right| \le K_0, \]
and that for $t>0$ and $x \in \m R^d$, $| v(x,t) \le \max( \| u_\e^0 \|_{L^\infty},1)$.

Recall that $v \in \q C_b((0,+\infty), L^\infty(\m R^d))$, and that $\limsup_{t \to 0^+} \| v(t) \|_{L^\infty} \le \| u_\e^0 \|_{L^\infty}$ (see \cite{DCot15}). We begin with the $L^{\infty}$ estimates for $v$. Set
\[ \sigma(x,t)=|v(x,t)|^2-1.\]

Multiplying equation \eqref{aa1} by $U$, we are led to the equation for $\sigma$,
\begin{equation}
\label{aa2}
\frac{\partial \sigma}{\partial t}-\Delta \sigma +2|\nabla v|^2 +2\sigma(1+\sigma)=0. 
\end{equation}

Consider next the EDO
\begin{equation}
\label{aa4}
y'(t)+2y(t)(y(t)+1)=0, 
\end{equation}

and notice that (\ref{aa4}) possesses the explicit solution defined for $t>0$ by
\begin{equation}
\label{aa5}
y_{t_0}(t)=\frac{\exp(-(t-t_0)/2)}{1-\exp(-(t-t_0)/2)}, \quad \text{with} \quad t_0 = 2 \ln \left( 1- \frac{1}{\max(1,\| u_\e^0 \|_{L^\infty})^2} \right),
\end{equation}
so that $y_{t_0}(0) = \max(1,\| u_\e^0 \|_{L^\infty})^2-1$ and as consequence
\[ \sup_{x \in \m R^d} \sigma(x,0) \le y_{t_0}(0). \]
We claim that
\begin{equation}
\label{aa6}
\forall t >0, \ \forall x \in \m R^d, \quad \sigma(x,t) \le y_{t_0}(t). 
\end{equation}

Indeed, set $\tilde{\sigma}(x,t)=y_0(t)$. Then
\begin{equation}
\label{aa8}
\partial_t \tilde{\sigma} -\Delta \tilde{\sigma} +2\tilde{\sigma}(1+\tilde{\sigma})=0, 
\end{equation}

and therefore by (\ref{aa2}),
\begin{equation}
\label{aa9}
\partial_t (\tilde{\sigma}-\sigma)-\Delta (\tilde{\sigma}-\sigma) +2(\tilde{\sigma}-\sigma)(1+\tilde{\sigma}+\sigma)\ge 0. 
\end{equation}

Note that $1+\tilde{\sigma}+\sigma=|v|^2+\tilde{\sigma}\ge 0$ and $\tilde \sigma(0) - \sigma(0) >0$. The maximum principle implies that 
\[ \forall t >0, \ \forall x \in \m R^d, \quad \tilde{\sigma}(x,t)-\sigma(x,t) \ge 0, \]
which proves the claim (\ref{aa6}). Then observe that $t_0 < 0$ and that $y_{0}$ is decreasing on $(0,+\infty)$, so that
\[ \forall t >0, \ \forall x \in \m R^d, \quad \sigma(t,x) \le y_{t_0}(t) \le y_0(t). \]
Observe that the first bound give $|v(x,t)| \le \max(1,\| u_\e^0 \|_{L^\infty}$ for all $t>0$ and $x \in \m R^d$.
Al,so for $t \ge 1$ and $x \in \m R^d$, $|v(x,t)| \le \sqrt{1 + y_0(1)} \le 2$.

We next turn to the space and time derivatives. Since $|v(x,t)| \le \sqrt{1 + y_0(1/2)}$ for $t \ge 1/2$, there exists $K_1 \ge 1$ (independent of $\e$) such that
\begin{gather} \label{bd:v_nl}
\forall t \ge \frac{1}{2}, \ \forall x \in \m R^d, \quad |v(x,t)|^3 + |v(x,t)| \le K_1.
\end{gather}
Let $t \ge 1$. Now, differentiating in space the Duhamel formula between times $t-1/2 \ge 1/2$ and $t$ gives
\[ \nabla v(t) = (\nabla G)(1/2) * v(t-1/2) + \int_{t-1/2}^t (\nabla G)(t-s) * (v(s) (1-|v(s)|^2) ds, \]
where $\ds G(x,t) = \frac{1}{(4\pi)^{d/2}} e^{-x^2/4t}$ is the heat kernel. Recall that $\| \nabla_x G(t) \|_{L^1} \le C/\sqrt t$. Also, as $t-1/2 \ge 1/2$, there holds $\| v(t-1/2) \|_{L^\infty} \le 2$ and \eqref{bd:v_nl} for all $s \in [t-1/2,t]$: hence
\begin{align*} 
\| \nabla v(t) \|_{L^\infty} & \le \| \nabla G(1/2) \|_{L^1} \| v(t-1/2) \|_{L^\infty} \\
& \qquad + \int_0^t \|  \nabla G(t-s) \|_{L^1} \| v(s) (1-|v(s)|^2) ds \|_{L^\infty} ds \\
& \le \frac{C}{\sqrt 2} 2  + CK_1 \int_{t-1/2}^t \frac{ds}{\sqrt{t-s}} \le CK_1.
\end{align*}
Similarly, we can differentiate the Duhamel formula twice:
\[ \nabla^2 v(t) = (\nabla^2 G)(1/2) * v(t-1/2) + \int_{t-1/2}^t (\nabla G)(t-s) * \nabla ((v(s) (1-|v(s)|^2)) ds, \]
Using that $|\nabla( v(s,x) (1-|v(s,x)|^2)| \le C K_1 |\nabla v(s,x)|$ and $\| \nabla^2 G(t) \| \le C/t$, we can differentiate once 
\begin{align*}
\| \nabla^2 v(t) \|_{L^\infty} & \le \| \nabla^2 G(1/2) \|_{L^1} \| v(t-1/2) \|_{L^\infty} \\
& \qquad + \int_{t-1/2}^t \|  \nabla G(t-s) \|_{L^1} \| \nabla( v(s) (1-|v(s)|^2) ) ds \|_{L^\infty} ds \\
& \le C \sqrt 2 + \int_{t-1/2}^t \frac{CK_1}{\sqrt{t-s}} ds \le C K_1.
\end{align*}
Finally,
\[ |\partial_t v| = | \Delta v + v(1-|v|^2)| \le |\nabla^2 v | + K_1 \le C K_1. \qedhere \]
\end{proof}

%\begin{prop}%{(similar to Proposition 1.1, Part I, \cite{Be1}).} 
%\label{prop1.1}
%Let $u_\e$ be a solution of \eqref{pgl} satisfying the initial condition \eqref{in}. Then there exists a constant $K>0$ depending only n $n$  such that, for any $t\ge \e^2$ and $x \in \m{R}^d$, we have
%\[ |u_\e(x,t)| \le 3,\ |\nabla u_\e(x,t)| \le \frac{K}{\e},\ |\frac{\partial u_\e}{\partial t}(x,t)| \le %\frac{K}{\e^2}. \]
%\end{prop}

We have the following variant of Proposition \ref{prop1.1}.

\begin{prop}%{(similar to Proposition 1.2, Part I, \cite{Be1}).} 
\label{prop1.2}
Let $u_\e^0 \in L^{\infty}(\m{R}^d)$ and $u_\e$ be the associated solution of \eqref{pgl}. Assume that for some constants $C_0 \ge 1,\ C_1\ge 0$ and $C_2 \ge 0$, 
\[ \forall x \in \m R^d, \quad |u_\e^0(x)| \le C_0,\quad |\nabla u_\e^0(x)| \le \frac{C_1}{\e},\quad |\nabla^2 u_\e^0(x)| \le \frac{C_2}{\e^2}. \] 
Then for any $t>0$ and $x \in \m{R}^d$, we have
\[ |u_\e(x,t)| \le C_0,\quad |\nabla u_\e(x,t)| \le \frac{C}{\e},\quad \left| \partial_t u_\e (x,t) \right| \le \frac{C}{\e^2}, \]
where $C$ depends only on $C_0$, $C_1$ and $C_2$.
\end{prop}

Proposition \ref{prop1.2} provides an upper bound for $|u_\e|$. The next lemma provides a local lower bound on $|u_\e|$, when we know it is away from zero on some region.
Since we have to deal with parabolic problems, it is natural to consider parabolic cylinders of the type
\begin{gather} \label{def:cyl}
\Lambda_\alpha(x_0,T,R,\Delta T) = B(x_0,\alpha R) \times [T+(1-\alpha^2)\Delta T, T+\Delta T].
\end{gather}
Sometimes, it will be convenient to choose $\Delta T=R$ and write $\Lambda_\alpha(x_0,T,R)$. Finally if there is no ambiguity, we will simply write $\Lambda_\alpha$, and even $\Lambda$ if $\alpha=1$.

\begin{lem}[\cite{Be1}]%{(similar to Lemma 1.1, Part I, \cite{Be1}).}  
\label{lem1.1}
Let $u_\e^0 \in L^{\infty}(\m{R}^d)$ satisfying \eqref{in} and $u_\e$ be the associated solution of \eqref{pgl}. 
Let $x_0 \in \m{R}^d,\ R >0,\ T \ge 0$ and $\Delta T>0$ be given. Assume that
\[ |u_\e| \ge \frac{1}{2} \quad \text{ on } \Lambda(x_0,T,R,\Delta T), \]
then \[ 1-|u_\e| \le C(\alpha, \Lambda)\e^2\ (\| \nabla \phi_\e\|_{L^{\infty}(\Lambda)} + |\ln \e|) \quad \text{ on } \Lambda_\alpha, \]
where $\phi_\e$ is defined on $\Lambda$, up to a multiple of $2\pi$, by $u_\e=|u_\e| \exp(i\phi_\e)$. 
\end{lem}

\begin{proof}
We refer to \cite[Lemma 1.1, p. 52]{Be1}.
\end{proof}

\subsection{The monotonicity formula and some consequences}

In this section, we provide various tools which will be required in the proof of Theorem \ref{th1}.

\subsubsection{The monotonicity formula}

%We first need the following lemma, which will enable us to do integrations by parts.
%
%\begin{lem} {\bf Integration by parts.}
%\label{ibp}
%Let $f \in L^1(\m{R}^d,\m{R})$ and $\vec{g} \in L^1(\m{R}^d,\m{R}^d)$ such that $\nabla f.\vec{g} \in L^1(\m{R}^d)$ and $f \Div(\vec{g}) \in L^1(\m{R}^d)$. 
%Then, if $f |\vec{g}| \in L^1(\m{R}^d)$, we have
%\begin{equation}
%\label{IbP} \tag{IBP}
%\int_{\m{R}^d} \nabla f.\vec{g}  = \int_{\m{R}^d} f \Div(\vec{g}). 
%\end{equation} 
%\end{lem}
%
%\begin{proof}
%We first suppose that $f$ and $\vec{g}$ are $C^1$ functions.
%We consider a sequence $R_m \rightarrow +\infty$ to be determined later.
%
%Since $f$ and $\vec{g}$ are regular functions, we have
%\begin{equation}
%\label{ip} 
%\int_{B(0,R_n)} \nabla f.\vec{g}  = \int_{B(0,R_n)} f \Div(\vec{g}) + \int_{\partial B(0,R_n)} f\ \vec{g}.\vec{n}. 
%\end{equation} 
%
%Moreover, \[ \int_{B(0,R_n)} \nabla f.\vec{g} \rightarrow \int_{\m{R}^d} \nabla f.\vec{g} \text{ and } \int_{B(0,R_n)} f \Div(\vec{g}) \rightarrow \int_{\m{R}^d} f \Div(\vec{g}).\]
%
%We write  \[\int_{\m{R}^d} f\ \vec{g}.\vec{n} = \int_0^{\infty} dr \left( \int_{\partial B(0,r)} f\ \vec{g}.\vec{n} \right).\]
%Since $\int_{\m{R}^d} f\ \vec{g}.\vec{n} <\infty$, there exists a sequence $R_m \rightarrow +\infty$ such that $\int_{\partial B(0,R_m)} f\ \vec{g}.\vec{n} \rightarrow 0$.
%
%With this choice of $R_m$, the conclusion follows directly fom (\ref{ip}).
%By density, we also get the result for $f,\vec{g} \in L^1$.
%\end{proof}

%\bigskip

For $(x_{*},t_{*}) \in \m{R}^d \times [0,+\infty)$ we set
\[ z_{*} = (x_{*},t_{*}). \]
For $t_* > 0$ and $0< R \le \sqrt{t_{*}}$ we defined the weighted energy, scaled and time shifted, by
\begin{gather} \label{def:Ew}
E_{w,\e}(u_\e,z_*,R)  = E_w(z_*,R) :=  \frac{1}{R^{d-2}} \int_{\m{R}^d} e_{\e}(u_{\e})(x,t_{*}-R^2) \exp \left( -\frac{|x-x_{*}|^2}{4R^2} \right) dx.
\end{gather}
Also it will be convenient to use the multiplier
\begin{gather} \label{def:Xi}
\Xi(u_\e, z_*)(x,t) = \frac{1}{4|t-t_*|} \left[ (x-x_*).\nabla u_\e(x,t) + 2(t-t_*) \partial_t u_\e(x,t) \right]^2.
\end{gather}
We stress out that in the integral defining  $E_w$, we introduced a time shift $\delta t=-R^2$.
The following monotonicity formula was first derived by Struwe \cite{Str}, and used in his study of the heat flow for harmonic maps. 

\begin{prop}%{(similar to Proposition 2.1, Part I, \cite{Be1}).} 
Let $u_\e^0 \in L^{\infty}(\m{R}^d)$ satisfying \eqref{in} and denote $ u_\e$ the associated solution of \eqref{pgl}. We have, for $0 < r < \sqrt{t_*}$,
\begin{align}
\frac{d E_w}{dR}(z_*,r) & = \frac{1}{r^{d-1}} \int_{\m{R}^d} \frac{1}{2r^2}\ ((x-x_*) \cdot \nabla u_\e(x,t_*-r^2)-2r^2 \partial_t u_\e(x,t_*-r^2))^2 \exp \left( -\frac{|x-x_*|^2}{4r^2} \right) dx  \nonumber \\
& \quad + \frac{1}{r^{d-1}} \int_{\m{R}^d} 2V_\e(u_\e)(x,t_*-r^2) \exp \left(-\frac{|x-x_*|^2}{4r^2} \right) dx  \label{prop2.1} \\
& = \frac{(4\pi)^{d/2}}{r} \int_{\m{R}^{d+1}} 2|t-t_*| \Xi(z_*)(x,t) G(x-x_*,t-t_*) dx \delta_{t_*-r^2}(t) \nonumber \\
& \quad + (4\pi)^{d/2} r \int_{\m{R}^{d+1}} 2V_\e(u_\e)(x,t)G(x-x_*,t-t_*) dx \delta_{t_*-r^2}(t), \nonumber
\end{align} 
where $G(x,t)$ denotes the heat kernel
\begin{equation}
\label{eel4}
G(t,x) = \begin{cases}
\displaystyle \frac{1}{(4\pi t)^{\frac{d}{2}}} \exp \left( -\frac{|x|^2}{4t} \right) & \text{ for } t>0, \\
0 & \text{ for } t\le 0. 
\end{cases}
\end{equation}

In particular, 
\begin{equation}
\label{MF}
 \frac{d E_w}{dR}(z_*,r)\ge 0.
\end{equation}
As a consequence,  $R \mapsto E_w(z_*,R)$ can be extended to a non-decreasing, continuous function of $R$ on $[0,\sqrt{t_*}]$, with $E_w(z^*,0)=0$.
\end{prop}

\begin{proof}
For $0 < R < \sqrt{t_*}$, the map $\ds (R,x) \mapsto e_{\e}(u_{\e})(x,t_{*}-R^2) \exp \left( -\frac{|x-x_*|^2}{4R^2} \right)$ is smooth (due to parabolic regularization), and satisfies domination bounds due to \eqref{zz28} and the gaussian weight. Therefore $R \mapsto E_w(z_*,R)$ is smooth on $(0, \sqrt{t_*})$ and we can perform the same computations as in Proposition 2.1 in \cite{Be1}; integrations by parts are allowed for the same reasons. This proves formula \eqref{prop2.1}, and the monotonicity property follows immediately. It remains to study the continuity at the endpoints.

For the limit $R \to 0$, the bounds \eqref{zz28} show that $|e_\e(u_\e)(x,t_*-R^2)| \le C(t_*)/\e^2$ uniformly for $R \le t_*/2$, so that in that range
\begin{align*}
E_w(z_*,R) & = R^2 \int e_\e(u_\e)(x,t_*-R^2) \frac{1}{R^d} \exp \left( -\frac{|x-x_*|^2}{4R^2} \right) dx \\
& \le R^2 C(t_*)/\e^2 (4\pi)^{d/2} \to 0 \quad \text{as } R \to 0.
\end{align*}
For the limit $R \to \sqrt{t_*}$, let us recall that for any $p <+\infty$, $u_{\e}(\cdot, t) \to u_\e^0$ strongly in $L^p_{\mr{loc}}(\m R^d)$ (see \cite[Theorem 2]{DCot15}), and as $|u_\e(x,t) | \le \max (\| u_\e^0 \|_{L^\infty}, 1)$, we infer that
\[ \int V(u_\e)(x,t_*-R^2) \exp \left( -\frac{|x-x_{*}|^2}{4R^2} \right) dx \to \int V(u_\e^0)(x) \exp \left( -\frac{|x-x_*|^2}{4t_*} \right) dx. \]
%Furthermore, $u_\e(t) \weak u_\e^0$ in $\q S'(\m R)$ (due to the $L^\infty$ bound \eqref{zz28} and for example, the strong $L^p_{\mr{loc}}$ convergence). 
For the derivative term, we use the Duhamel formula:
\[ \nabla u_\e(t) = G(t) * \nabla u_\e^0 + \int_0^t \nabla G(t-s) * (u_\e (1-|u_\e|^2))(s) ds = G(t) * \nabla u_\e^0 + D(x,t). \]
The Duhamel term $D(x,t)$ is harmless, indeed
\begin{align*}
\left| D(x,t) \right| & \le \int_0^t \| \nabla G(t-s) \|_{L^1} \| u_\e (1-|u_\e|^2) \|_{L^\infty} ds \\
& \le C \max (1, \| u_\e^0 \|_{L^\infty})^3 \int_0^t \frac{ds}{\sqrt{t-s}} \le  C \max (1, \| u_\e^0 \|_{L^\infty})^3 \sqrt t.
\end{align*}
Therefore
\[ \int |D(x,t_*-R^2)|^2 \exp \left(- \frac{|x-x^*|}{4R^2} \right) dx \le C (t_* -R^2) \to 0 \quad \text{as } R \to \sqrt{t_*}. \]
The linear term requires to recall Claim 13 of \cite{DCot15}. Due to assumption $H_1(M_0)$, for any $\alpha >0$, $\nabla u_\e^0 \in L^2(e^{- \alpha |x|^2} dx)$. From \cite[Claim 13]{DCot15}, we infer that for $\beta >2 \alpha$,
\[ \| \tau_h \nabla u_\e^0 - \nabla u_\e^0 \|_{L^2(e^{- \beta |x|^2} dx)} \to 0 \quad \text{as } h \to 0 \]
(where $\tau_h \phi(x) = \phi(x-h)$), and satisfies
\[  \| \tau_h \nabla u_\e^0 - \nabla u_\e^0 \|_{L^2(e^{- \beta |x|^2} dx)}  \le C e^{C(\beta) |h|^2}. \]
As a consequence, we can apply Lebesgue's dominated convergence theorem and conclude that
\[ \| G(t) * \nabla u_\e^0 - \nabla u_\e^0 \|_{L^2(e^{- \beta |x|^2} dx)} \to 0 \quad \text{as } t \to 0. \]
Choose $\beta = 1/(8t_*)$ and $\alpha = 1/(17t_*)$. Then it follows, using Cauchy-Schwarz inequality, that
\begin{align*}
\MoveEqLeft \left| \int (G(t_* - R^2) * \nabla u_\e^0)(x)|^2 \exp \left( -\frac{|x-x_*|^2}{4R^2} \right) - \int |\nabla u_\e^0(x)|^2 \exp \left( -\frac{|x-x_*|^2}{4t_*} \right) dx \right| \\
& \le \int  |G(t_* - R^2) * \nabla u_\e^0)(x) - \nabla u_\e^0(x)| \\
& \qquad \qquad \times ( |G(t_* - R^2) * \nabla u_\e^0)(x) + \nabla u_\e^0(x)|) \exp \left( -\frac{|x-x_*|^2}{4R^2} \right) dx  \\
& \qquad \qquad + \int |\nabla u_\e^0(x)|^2 \left( \exp \left( -\frac{|x-x_*|^2}{4R^2} \right) -  \exp \left( -\frac{|x-x_*|^2}{4t_*} \right) \right) dx  \\
& \le \| G(t_* - R^2) * \nabla u_\e^0)(x) - \nabla u_\e^0(x) \|_{L^2(e^{- \beta |x|^2} dx)} \\
& \qquad \qquad \times  \| G(t_* - R^2) * \nabla u_\e^0)(x) + \nabla u_\e^0(x) \|_{L^2(e^{- \beta |x|^2} dx)} + o(1)\\
&  \to 0 \quad \text{as } R \to \sqrt{t_*}.
\end{align*}
(The $o(1)$ on the second last line comes from $\nabla u_\e^0 \in L^2(e^{- \beta |x|^2} dx)$ and Lebesgue's dominated convergence theorem.) Hence, summing up, we proved that $R \mapsto E_w(z_*,R)$ is (left-)continuous  at $R=t_*$.
\end{proof}

\subsubsection{Bounds on the energy}

\begin{lem} \label{H1:H0}
There exists a constant $C$ (not depending on the dimension) such that the following holds. Let $(v_\e)_{\e >0}$ be a family of functions satisfying $H_1(M_0)$, then for any $R >0$, 
\[ \forall \e >0, \ \forall y \in \m R^d, \quad \int e_\e(v_\e)(x) \exp \left( - \frac{|x-y|^2}{R^2} \right) dx \le C(d) (1+R)^d M_0 |\ln \e|. \]
Reciprocally, if $(w_\e)_{\e >0}$ is a family of functions such that for some $R >0$,
\[ \forall \e>0, \ \forall y \in \m R^d, \quad \int e_\e(w_\e)(x) \exp \left( - \frac{|x-y|^2}{R^2} \right) dx \le M_0 |\ln \e|, \]
then $(w_\e)_{\e >0}$ satisfies $H_1(C (1 + R^{-1})^d M_0)$.
\end{lem}

\begin{proof}
By translation invariance, we can assume $y=0$. We consider the case $R \ge 1$ (the case $R \le 1$ is dealt with $R=1$). For $k$ in $\m Z^d$, denote $Q_k$ the cube in $\m R^d$, of length R and centered at $Rk \in \m R^d$. Then 
\[ \forall x \in Q_k, \quad |x| \ge R|k|-R\sqrt d. \]
Also there exists a constant $C(d)$  such that any cube $Q_k$ is covered by $C(d) R^d$ balls of radius 1.
Therefore,
\begin{align*}
\int e_\e(v_\e)(x) e^{ - \frac{|x|^2}{R^2}} dx & \le \sum_{k \in \m Z^d} \exp \left( - \frac{(R|k| - R\sqrt{d})^2}{R^2} \right) \int_{Q_k} e_\e(v_\e)(x) dx \\
& \le C(d) R^d M_0 |\ln \e| \sum_{k \in \m Z^d} e^{-(|k| - \sqrt{d})^2}.
\end{align*}
The series is clearly convergent, which gives the first result claimed.

For the second, we clearly have
\[ \int e_\e(w_\e)(x) \exp \left( - \frac{|x-y|^2}{R^2} \right) dx \ge \frac{1}{e} \int_{B(y,R)} e_\e(w_\e)(x) dx. \]
This means that the energy on balls of radius $R$ is at most $e M_0 |\ln \e|$. If $R \ge 1$, then this is enough. If $R \le 1$, then any ball of radius 1 can be covered by at most $C(d)/R^d$ balls of radius $R$, so that for all $y \in \m R^d$,
\[ \int_{B(y,R)} e_\e(w_\e)(x) dx \le \frac{C(d)}{R^d} e M_0 |\ln \e|. \qedhere \]
\end{proof}

The first consequence of the monotonicity formula is that $(H_1)$ is a condition which propagates in time in the following way.

\begin{prop}
\label{p1p}
Let  $u_\e$ be a solution of \eqref{pgl} satisfying the initial condition \eqref{in}. Then for any $T>0$, $(x,t) \rightarrow u_\e(x,T+t)$ is still a solution of \eqref{pgl}, whose initial condition satisfies $H_1(C(d)(1+T)M_0)$. More precisely there holds
\begin{gather} \label{zz25}
\forall \e >0, \forall y \in \m R^d, \forall R >0, \quad \int_{B(y,R)} e_\e(u_\e)(x,t) dx \le C(d) (1+R)^d (1+t) M_0 |\ln \e|.
\end{gather}
\end{prop}

\begin{proof}
Let $R=1$, then applying the monotonicity formula at the point $\ds \left(y,t+\frac{1}{4} \right)$ between $\ds \frac{1}{2}$ and $\ds \sqrt{t+\frac{1}{4}}$, we get
\begin{align*}
2^{d-2} \int_{\m{R}^d} e_{\e}(u_{\e})(x,t) \exp \left( - |x-y|^2 \right) dx  & \le  \frac{2^{d-2}}{(4t +1)^{\frac{d-2}{2}}} \int_{\m{R}^d} e_{\e}(u^0_{\e}) \exp \left( -\frac{|x-y|^2}{4t+1} \right) dx. \\
\end{align*}
Hence,
\begin{align*}
 \int_{\m{R}^d} e_{\e}(u_{\e})(x,t) \exp \left( - |x-y|^2 \right) dx & \le \left( 1 + 4t) \right)^{-\frac{d-2}{2}}  
C(d) \left( 2+ 4t  \right)^{d/2} M_0 |\ln \e| \\
& \le C(d) (1+t) M_0 |\ln \e|.
\end{align*}
Using Lemma \ref{H1:H0} we have the result for $R = 1$ and so for $R \le 1$. Finally, for $R \ge 1$, we cover $B(y,R)$ by balls of radius 1, which can be done with at most $C(d) R^d$ balls.
\end{proof}

As an immediate consequence, we infer an upper bound on the energy on compact sets. The bound  \eqref{zz26} below will be  very useful in order to prove Theorem \ref{th3} in the same way as in \cite{Be1}.

\begin{cor} \label{Prop1}
Let $u_\e$ be a solution of \eqref{pgl} satisfying the initial data \eqref{in}. Then for any compact $K \subset \m{R}^d \times [0,+\infty)$, we have
\begin{equation}
\label{zz26}
\int_{K} e_{\e}(u_{\e})(x,t) dxdt  \le C(K) M_0|\ln \e|.
\end{equation}
\end{cor}

\begin{proof}
Being compact, $K$ is bounded: $K \subset B(0,R) \times [0,T]$. Therefore,
\begin{align*}
\MoveEqLeft \int_K e_{\e}(u_{\e})(x,t) dx dt \le \int_{B(0,R) \times [0,T]} e_{\e}(u_{\e})(x,t) dxdt  \\
& \le \int_0^T C(d) (1+R)^d (1+t) M_0 |\ln \e| dt \le C(d) (1+R)^d (1+T) T M_0 |\ln \e|.
\end{align*}
Whence \eqref{zz26}.
\end{proof}

\subsubsection{Space-time estimates}

One crucial lack in relaxing ($H_0$) to ($H_1$) is that the energy no longer provides a bound on $\| \partial_t u_\e \|_{L^2_{x,t}} \le E_\e(u_\e)$. To remedy this, we make use of the $\Xi$ multiplier.

\begin{lem}%{(similar to Lemma 2.2, Part I, \cite{Be1}).} 
Let $u_\e^0 \in L^{\infty}(\m{R}^d)$ and $u \equiv u_\e$ be the associated solution of \eqref{pgl}.

For any $z_*=(x_*,t_*) \in \m{R}^d \times [0,+\infty)$, the following equality holds, for $R_*=\sqrt{t_*}$.
\begin{multline}
\label{lem2.2}
\int_{\m{R}^d \times [0,t_*]} (V_\e(u_\e) + \Xi(u_\e,z_*))(x,t)\  G(x-x_*,t-t_*)\ dx dt \\
= \frac{1}{(4\pi)^{d/2} t_*^{\frac{d-2}{2}}}\int_{\m{R}^d \times \{0\}} e_\e(u)(x,0) \exp \left(-\frac{|x-x_*|^2}{4t_*} \right)\ dx =E_w(z_*,R_*),
\end{multline}
where $\Xi$ is defined in \eqref{def:Xi}.
\end{lem}
\begin{proof}

Integrating equality (\ref{prop2.1}) from $0$ to $R_{*}$ (recall that $E_{w}(z^*,0) =0$), we obtain
\begin{align}
\MoveEqLeft (4\pi)^{d/2} E_w(z_*,R_*) = \int_{0}^{R_*} 2rdr\ \int_{\m{R}^d \times \{t_*-r^2\}}V_\e(u(x,t))G(x-x_*,t-t_*)\ dx \label{equa2.18} \\
& \quad+ \int_{0}^{R_*} 2rdr\ \int_{\m{R}^d \times \{t_*-r^2\}} \frac{1}{4r^2} \left[ (x-x_*) \cdot \nabla u_\e -2r^2 \partial_t u \right]^2 G(x-x_*,t-t_*)\ dx. \nonumber
\end{align}

Expressing the integral on the right-hand side of (\ref{equa2.18}) in the variable $t=t_*-r^2$ (so that $dt=-2rdr$) yields
\begin{align*}
\MoveEqLeft  (4\pi)^{d/2} E_w(z_*,R_*) = -\int_{t_*}^{0} dt\ \int_{\m{R}^d \times \{t\}}V_\e(u(x,t))G(x-x_*,t-t_*)\ dx  \\
&\quad - \int_{t_*}^{0} 2rdr\ \int_{\m{R}^d \times \{t\}} \frac{1}{4|t-t_*|} ((x-x_*).\nabla u -2r^2 \partial_t u)^2 G(x-x_*,t-t_*)\ dx. \qedhere
\end{align*}
\end{proof}

\begin{prop}
For any compact $K \subset \m R^d \times [0,+\infty)$, there exist a constant $C(K)$ such that 
\begin{gather} \label{bd:dtu_L2}
 \int_K  |\partial_t u_\e(x,t)|^2 dxdt \le C(K) M_0 |\ln \e|
 \end{gather}
\end{prop}

\begin{proof}
It suffices to prove the bound on the compacts $K_T = B(0,\sqrt T) \times [0,T]$ for all $T \ge 1$. Let $t_* = 2T$ and $x_* =0$, then for $(x,t) \in B(0,\sqrt T) \times [0,T]$, we have $t-t_* \ge T \ge |x-x_*|^2$  so that
\[  G(x-x_*,t-t_*) \ge \frac{1}{(4\pi (t-t_*))^{d/2}} \exp \left( - \frac{|x-x_*|^2}{4(t-t_*)^2} \right) \ge  \frac{e^{-1/4}}{(4\pi T)^{d/2}} \]
and
\begin{align*} 
|\partial_t u_\e(x,t)|^2 & \le \frac{1}{2T} \frac{1}{4|t-t_*|} (2(t-t_*) \partial_t u_\e(x,t))^2 \\
& \le \frac{1}{T} \Xi(u_\e,z_*)(x,t) + \frac{2}{t-t_*} ((x-x_*) . \nabla u_\e(x,t))^2 \\
&  \le  \frac{1}{T} \Xi(u_\e,z_*)(x,t) + |\nabla u_\e(x,t)|^2.
\end{align*}
Therefore, using $V_\e(u_\e) \ge 0$, \eqref{lem2.2} and \eqref{zz26}, we get
\begin{align*}
\MoveEqLeft
\int_0^T \int_{B(0,\sqrt T)} |\partial_t u_\e(x,t)|^2 dxdt  \le C(T) \int_{\m R^d \times [0,t_*]}  \Xi(u_\e,z_*)(x,t)  G(x-x_*,t-t_*) dxdt \\
& \qquad + \int_{B(0,\sqrt{t_*}) \times [0,t_*]} |\nabla u_\e(x,t)|^2 dxdt  \\
& \le C(T) \int_{\m R^d \times [0,t_*]} (V_\e(u_\e) + \Xi(u_\e,z_*))(x,t)  G(x-x_*,t-t_*) dxdt \\
& \qquad + \int_{B(0,\sqrt{t_*}) \times [0,t_*]} e_\e(u_\e)(x,t) dxdt \\
& \le C(T) E_w(z_*,\sqrt{t_*}) + C(t_*) M_0 |\ln \e| \le C(T) M_0 |\ln \e|. \qedhere
\end{align*}

\end{proof}

\subsubsection{Localizing the energy}

In some of the proofs of the main results, it will be convenient to work on bounded domains for fixed time slices. But since the integral in the definition of $E_{w}$ is computed on the whole space, we will have to use two kinds of localization methods.

The first one results from the monotonicity formula.

\begin{prop}%{(similar to Proposition 2.3, Part I, \cite{Be1}).} 
\label{prop2.3}
Let $u_\e$ be a solution of \eqref{pgl} satisfying the initial bound \eqref{in}.

Let $T>0$ $x_T \in \m{R}^d$ and $r, \lambda >0$. There holds
\begin{multline}
\int_{\m{R}^d} e_\e(u_\e) (x,T) \exp \left( -\frac{|x-x_T|^2}{4r^2} \right) dx
 \le  \int_{B(x_T,\lambda r)} e_\e(u_\e)(x,T) dx  \\
  + C(d) \left( \frac{\sqrt{2}r}{\sqrt{T+2r^2}} \right)^{d-2} \exp \left(-\frac{\lambda^2}{8} \right) (1+T+2r^2)^d M_0|\ln \e|.  
\label{energy_local}
\end{multline}
\end{prop}

\begin{proof}
We split the integral between $B(x_T, \lambda r)$ and its complement. Now for $x$ such that $|x-x_T| \ge \lambda r$, we have
\[ \exp \left( -\frac{|x-x_T|^2}{8r^2} \right) \le \exp \left( -\frac{\lambda^2}{8} \right) \exp \left( -\frac{|x-x_T|^2}{8r^2} \right).\]
Therefore
\begin{align*}
\MoveEqLeft \int_{\m{R}^d} e_\e(u_\e) (x,T) \exp \left( -\frac{|x-x_T|^2}{4r^2} \right) dx \\
& \le  \int_{B(x_T,\lambda r)} e_\e(u_\e)(x,T) dx + e^{-\lambda^2/8} \int_{\m{R}^d} e_\e(u_\e) (x,T) \exp \left( -\frac{|x-x_T|^2}{8r^2} \right) dx.
\end{align*}
Now, we apply the monotonicity formula to the second integral term of the right hand side, at the point $(x_T,T+2r^2)$, and between $\sqrt 2 r$ and $\sqrt{T+2r^2}$. We get
\begin{align*}
\MoveEqLeft \frac{1}{(\sqrt{2}r)^{d-2}} \int_{\m{R}^d} e_\e(u_\e)(x,T) \exp \left( -\frac{|x-x_T |^2}{8r^2} \right)dx \\
& \le \frac{1}{(T+2r^2)^{\frac{d-2}{2}}} \int_{\m{R}^d} e_\e(u_\e^0)(x) \exp \left( -\frac{|x-x_T|^2}{4(T+2r^2)} \right) dx \\
& \le C(d)\frac{(1+T +2r^2)^d}{(T+2r^2)^{\frac{d-2}{2}}} M_0|\ln \e|,
\end{align*}
and the conclusion follows.
\end{proof}

The second localization method, inspired by Lin and Rivi\`ere \cite{Lin2} is based on a Pohozaev type inequality.

\begin{prop}%{(similar to Proposition 2.4, Part I, \cite{Be1}).} 
Let $u_\e^0 \in L^{\infty}(\m{R}^d)$, $u_\e$ be the associated solution of \eqref{pgl} satisfying the initial bound \eqref{in}, and $\Xi$ as in \eqref{def:Xi}.

Let $0<t<T$. The following inequality holds, for any $x_T\in \m{R}^d$:
\begin{align*}
\MoveEqLeft \int_{\m{R}^d} e_\e(u_\e)(x,t) \frac{|x-x_T|^2}{4(T-t)} \exp \left(-\frac{|x-x_T|^2}{4(T-t)} \right) dx \\  
& \le \frac{d}{2} \int_{\m{R}^d} e_\e(u_\e)(x,t) \exp \left(-\frac{|x-x_T|^2}{4(t-t)} \right) dx \\
& \qquad + \int_{\m{R}^d} (V_\e(u_\e)+3\Xi(u_\e,z_T))\ \exp \left(-\frac{|x-x_T|^2}{4(T-t)} \right)\ dx.
\end{align*}

As a consequence, 
\begin{multline*}
\int_{\m{R}^d \times \{t\}} e_\e(u_\e) \exp(-\frac{|x-x_T|^2}{4(T-t)})\ dx  \le \int_{B(x_T,r_T) \times \{T\}} e_\e(u_\e) \exp(-\frac{|x-x_T|^2}{4(T-t)})\ dx \\
 + \frac{2}{d} \int_{\m{R}^d \times \{t\}} (V_\e(u)+3\Xi(u,z_T))\ \exp(-\frac{|x-x_T|^2}{4(T-t)})\ dx
\end{multline*}
where $r_T=2\sqrt{d(T-t)}.$
\end{prop}

\subsection{End of the proof of clearing-out} \label{sec:2.4}
%[Proof of Theorems \ref*{th1}, \ref*{th2} and \ref*{th3}]{Proof of Theorems \ref{th1}, \ref{th2} and \ref{th3}}

In this paragraph, we complete the proofs of Theorems \ref{th1}, \ref{th2} and \ref{th3}.

\begin{proof}[Outline of the proof of Theorem \ref{th1}]
The proof follows word for word that of Theorem 1 in \cite[Sections 3, p. 67-99]{Be1}. Indeed, either calculations are made on bounded domains $K \times [T_0,T_1]$ (where $K$ is a compact of $\m R^d$ and $T_1>T_0 >0$), on which we have the same kind of bounds
\begin{itemize}
\item for the energy: for any $t \in [T_0,T_1]$:  $\ds \int_K |e_\e(u_\e)|(x,t) dx \le C(K,T_2) M_0|\ln \e|$ (obtained in \eqref{zz25}),
\item for the kinetic energy, in $L^2(dxdt)$: $\ds \int_{K \times [0,T_1]} |\partial_t u_\e|(x,t) dx \le C(K,T) M_0|\ln \e|$ (obtained in \eqref{bd:dtu_L2})
\item pointwise: $|u_\e| + \e | \nabla u_\e| + \e^2 |\partial_t u_\e|(x,t) \le C$ for $(x,t) \in K \times [\e^2,T_1]$ (obtained in \eqref{zz28} and useful as soon as $\e$ is so small that $T_1 >\e^2$),
\end{itemize}
or we multiply the energy by a weight of the form $\ds \exp \left( -\frac{|x|^2}{R} \right)$, for which we have the monotonicity formula and bounds \eqref{energy_local}.
%of (3.4) and (3.5) in \cite{Be1} which still stand in our case, given that $u_\e$ satisfies \eqref{in}. 

For the convenience of the reader, we remind the main steps of the argument. 

Let $\ds r_\e =1- \sqrt\frac{\sigma}{2K} \e$ and $t_\e = 1-r_\e^2 \ge 0$. As $|\partial_t u_\e| \le K/\e^2$ (in view of \eqref{zz28}),
\begin{gather} \label{eq:th1_3}
|1-|u_\e(0,1)|| \le  |1-|u_\e(0,t_\e)|| + \frac{K}{\e^2} (1-t_\e) \le |1-|u_\e(0,t_\e)|| + \frac{\sigma}{2}.
\end{gather}
Now, as $|\nabla_x u_\e| \le K/\e$ (again \eqref{zz28}), one easily deduces (see \cite[Lemma 3.3 p.458]{BeBrOr01}) that
\begin{gather} \label{eq:th1_4}
|1-|u_\e(0,t_\e)|| \le C \left( \frac{1}{\e^d} \int_{B(0,\e)} ( 1-|u_\e(x,t_\e)|)^2 dx \right)^{\frac{1}{d+2}}.
\end{gather}
This bound can be related to the weighted energy as follows:
\begin{align} \label{eq:th1_5}
\MoveEqLeft \frac{1}{\e^d} \int\limits_{B(0,\e)} ( 1-|u_\e(x,t_\e)|)^2 dx \le \frac{C}{r_\e^{d-2}} \int\limits_{B(0,\e)} V_\e(u_\e)(x,t_\e) \exp \left( - \frac{|x|^2}{r_\e^2} \right) dx \le E_{w,\e}((0,1),r_\e).
\end{align}
Let $\delta \in (0,1/2)$ to be fixed at the end of the proof. By considering the variations of the weighted energy on time intervals of the form $[1-\delta^{2k}, 1-\delta^{2(k+1)}]$ for $0 \le k \le |\ln \e|/4$ (so that $1-\delta^{2k} \le 1- r_\e^2$ and $E_{w,\e}((0,1),r_\e) \le E_{w,\e}((0,1),\delta^k)$ by monotonicity), at least one of these intervals (say for $k_0$) shows a decay less than $8\eta |\ln \delta|$ (we gained a $|\ln \e|$ factor). One can actually perform a time shift and rescaling (as $k \le  |\ln \e|/4$), so that we can furthermore assume $k_0=0$: 
\begin{gather} \label{eq:th1_1}
E_{w,\e}(u,(0,1),1) - E_{w,\e}(u,(0,1),\delta) \le 8\eta |\ln \delta|, \quad E_{w,\e}((0,1),r_\e) \le E_{w,\e}((0,1),1).
\end{gather}
Gathering \eqref{eq:th1_3}, \eqref{eq:th1_4}, and \eqref{eq:th1_5},
\[ |1-|u_\e(0,1)|| \le \frac{\sigma}{2} + C E_{w,\e}((0,1),1)^{\frac{1}{d+2}}. \]
The crux of the argument is therefore to bound $E_{w,\e}((0,1),1)$ solely in terms of $\eta$ and more precisely, \begin{gather} \label{eq:th1_2}
E_{w,\e}((0,1),1) \le C(\delta) \sqrt{\eta}.
\end{gather}
%Let $\delta \in (0,1/2)$ to be fixed at the end of the proof. By considering the variations of the weighted energy on time interval of the form $[1-\delta^{2k}, 1-\delta^{2(k+1)}]$ for $1 \le k \le |\ln \e|/4$, at least one of these intervals shows a decay less than $8\eta |\ln \delta|$. As $k \le  |\ln \e|/4$ one can perform a time shift and rescaling so that we can furthermore assume

To prove \eqref{eq:th1_2}, the starting point is the observation that
\[ |u_\e|^2 |\nabla u_\e|^2 = |u_\e \wedge \nabla u_\e|^2 + |u_\e|^2 |\nabla |u_\e||^2, \]
(which can be derived easily writing $u_\e = |u_\e| e^{i \varphi_\e}$). As $|\nabla u_\e| \le C/\e$ from \eqref{zz28},
\[ (1- |u_\e|^2) |\nabla u_\e|^2 \le \frac{1}{2}  |\nabla u_\e|^2 + C V_\e(u_\e), \]
and we get the pointwise bound
\[ e_\e(u_\e) \le 2 |u_\e|^2  |\nabla |u_\e||^2 + C V_\e(u_\e) + |u_\e \wedge \nabla u_\e|^2. \]
It is not so hard to obtain improved bounds on $|u_\e|$ and its derivative. Indeed, write the parabolic equation for $1- |u_\e|^2$: using \eqref{lem2.2} (in particular to treat the $\partial_t u_\e$ terms) and an averaging in time argument, one can infer that the set of times $t$ such that
\[ \int_{B(0,1)} \left( |\nabla |u_\e||^2 +  V_\e(u_\e) \right)(x,t) dx \le C(\delta) \sqrt \eta (E_w((0,1),1) +1) \]
is of large relative measure in $[1-4\delta^2,1-\delta^2]$ (for some explicit $C(\delta)$).

It follows that the crucial term to estimate is $u_\e \wedge \nabla u_\e$. For this term, one has the following Hodge-de Rham decomposition
\begin{gather*}
u_\e \wedge \nabla u_\e  = d\phi(t) + d^* \psi(t) + \xi(t) \quad \text{on } B(0,3/2) \times \{ t \}
\end{gather*}
where $d$ is the exterior derivative on $\m R^d$ ($d^*$ is its adjoint).

It is constructed as follows: define the Jacobian $Ju_\e := d( u_\e \wedge du_\e)$, and let $\psi(t)$ be such that $-\Delta \psi(t) = \chi Ju_\e$ where $\chi$ is a cut-off function on $B(0,2)$. Observe that $d^* d\psi(t)$ is closed on $B(0,3/2) \times \{ t \}$; invoking the Poincaré lemma, there exists $\xi(t)$ such that $d \xi(t) = d^* d\psi (t)$, $d^*\xi(t) =0$. Then $\phi(t)$ is obtained by invoking once again the Poincaré lemma. 

Direct elliptic estimates show that $\| \xi(t) \|_{L^2(B(0,3/2))} \le C \| \nabla \psi(t) \|_{L^2(B(0,2))}$. 

Noticing that $u_\e \wedge \partial_t u_\e = - d^*(u_\e \wedge du_\e)$, we have the elliptic equation for $\phi(t)$: for any small $\delta>0$,
\[ \left( - \Delta + \frac{x}{2\delta^2} \right) \phi(t) = u_\e \wedge \left( \frac{x}{2\delta^2} \cdot \nabla u_\e - \partial_t u_\e \right) - (d^* \psi(t) + \xi(t) ) \cdot \frac{x}{2\delta^2}. \]
One can obtain weighted elliptic estimates for the operator $\ds - \Delta + \frac{x}{2\delta^2} = - \exp \left( \frac{|x|^2}{4 \delta^2} \right) \nabla \cdot \left( \exp \left(- \frac{ |x|^2}{4 \delta^2} \right) \nabla \right)$ and from there obtain the bound
\begin{align*}
\int | \nabla \phi(t)|^2  \exp \left(- \frac{ |x|^2}{4 \delta^2} \right) dx \le C \delta^d E_w((0,1), \delta) +
C(\delta) \left( R(t) + \sqrt{R(t) E_w((0,1), \delta) } \right),
\end{align*}
where $R(t)$ only involves $\xi(t)$  (already bounded), $\Xi(u_\e)$ and $V_\e(u_\e)$ (taken care of by \eqref{lem2.2}) and $\psi(t)$:
\begin{align*}
R(t) & = \int_{\m R^d} ( \Xi(u_\e,(0,1))(x,t) + V_\e(u_\e))\exp \left(- \frac{ |x|^2}{4 \delta^2} \right)  dx \\ 
& \qquad + \int_{B(0,3/2)}  |\nabla \psi(x,t)|^2 + |\xi(x,t)|^2) \exp \left(- \frac{ |x|^2}{4 \delta^2} \right) dx.
\end{align*}
The remaining task is to estimate $\psi(t)$, which is the most involved. A first ingredient is a refined estimate on the Jacobian due to Jerrard and Soner \cite{Je2}: there exist $\beta,C >0$  such that for any smooth $w$ and test function $\varphi$,
\begin{align} 
\left| \int_{\m R^d} \langle Jw, \varphi \rangle dx \right| & \le \frac{C}{|\ln \e|} \| \varphi \|_{L^\infty} \int_{\Supp \varphi} e_\e(w) dx \nonumber \\
& \qquad + C \e^\beta \| \varphi \|_{W^{1,\infty}} \left( 1+ \int_{\Supp \varphi} e_\e(w) dx \right) (1+ \q H^d(\Supp \varphi)^2). \label{eq:JS_Jac}
\end{align}
(Observe the $|\ln \e|$ gain). A second ingredient is to compare $\psi$ with the solution $\tilde \psi$ to the analoguous \emph{heat} equation
\[ \partial_t \tilde \psi - \Delta \tilde \psi(t) = \chi Ju_\e. \]
One can get bounds on $\tilde \psi$ by the use of the monotonicity formula on $u_\e$, and then relate to $\psi$ by treating $\partial_t \tilde \psi$ as a perturbation. After an averaging in time argument, one can choose a right time slice $t \in [1-4\delta^2,1-\delta^2]$ such that, using the bounds from \eqref{lem2.2} and \eqref{bd:dtu_L2},
\[ \frac{1}{(1-t)^{d/2}} \int_{\m R^d} ( \Xi(u_\e,(0,1))(x,t) + V_\e(u_\e))(x,t) \exp \left(- \frac{ |x|^2}{4(1-t)} \right) dx \le C \frac{|\ln (1-t)|}{t^2} \eta \ln \e, \]
(a bound suitable for $\psi(t)$ and $\xi(t)$) and 
\begin{align*}
\int_{B(0,2)} |\nabla \psi(t)|^2 dx & \le C(\delta) \e^{1/6} E_w((0,1), \delta) \\
& \qquad + C(\delta) (E_w((0,1), \delta)+1) \int_{\m R^d} V_\e(u_\e)(t) \exp \left(- \frac{ |x|^2}{4 \delta^2} \right) dx.
\end{align*}
Combining all the above (and the monotonicity formula), one can choose $\delta>0$ small enough (independent of $\eta$ or $u_\e$) so that
\[ E_w((0,1),\delta) \le \frac{1}{2} E_w((0,1),1) + C \sqrt \eta. \]
Then using the first part of \eqref{eq:th1_1}, this proves \eqref{eq:th1_2} and the proof is complete.
%
%We recall that an important ingredient in finding the phase $\Phi_\e$ is a suitable Hodge-de Rham decomposition.
\end{proof}

From Theorem \ref{th1}, Proposition \ref{prop1} follows immediately.
We now provide a proof of Proposition \ref{prop2} which also is a consequence of Theorem \ref{th1}.

\begin{proof}[Proof of Proposition \ref{prop2}:]

Let $x_0$ be any given point in $\ds B \left( x_T,\frac{R}{2} \right)$.  The crux of the ar\-gu\-ment is the following claim.

\begin{claim}
We can find $\lambda(T)>0$ such that, for every $\sqrt{T_0}<r<\sqrt{T_1}$,
\[ \frac{1}{r^{d-2}} \int_{\m{R}^d} e_\e(u_\e)(x,T) \exp \left( -\frac{|x-x_0|^2}{4r^2} \right) dx \le \eta_1 |\ln \e|, \]
provided that $\ds \tilde{\eta} \le \frac{\eta_1}{2}$ (recall that $\lambda$ enters in the definition of $\tilde{\eta}$).
\end{claim}

 %{\bf Proof of the claim:}
To prove the claim, we use \eqref{energy_local}. Let  $\lambda>0$ and $\sqrt{T_0}<r<\sqrt{T_1}=R$, we have
%
%
%invoke Proposition \ref{prop2.3} and proceed as follows. Let $\lambda>0$ and $\sqrt{T_0}<r<\sqrt{T_1}=R$, we have by splitting between $B(0,\lambda R)$ and its complement
\begin{align}
 \MoveEqLeft 
 \frac{1}{r^{d-2}} \int_{\m{R}^d} e_\e(u_\e)(x,T) \exp \left( -\frac{|x-x_0|^2}{4r^2} \right) dx \le  \frac{1}{r^{d-2}} \int\limits_{B(x_0,\lambda r)} e_\e(u_\e) (x,T) dx \nonumber \\
 & \quad + C(d) \left( \frac{\sqrt{2} r}{\sqrt{T+2r^2}} \right)^{d-2} \exp \left(-\frac{\lambda^2}{8} \right) (1+T+2r^2)^d M_0 |\ln \e| \label{zz17}. 
\end{align}
As $r \le1$, it follows that
\[ \left( \frac{\sqrt{2} r}{\sqrt{T+2r^2}} \right)^{d-2}(1+T+2r^2)^d  \le C(d) \frac{(1+T)^d}{T^{\frac{d-2}{2}}}. \]
We first choose a function $\lambda_0$ continuous such that for $T >0$, $\lambda_0(T) \ge 1$ is so large that
\begin{equation}
\label{zz18}
C(d) \frac{(1+T)^d}{T^{\frac{d-2}{2}}} \exp \left(-\frac{\lambda_0^2(T)}{8} \right) M_0 \le \frac{\eta_1(\sigma)}{2}.
\end{equation}
Observe this can be done with furthermore requiring that there exists an absolute constant $c>0$ such that
\[ \forall T >0, \quad c \sqrt{\ln (2+1/T)} \le \lambda_0(T) \le \frac{1}{c} \sqrt{\ln (2+1/T)}. \]
Then it follows that
\[ \forall T >0, \forall \tau \in (T/2,2T), \quad \lambda_0(\tau)  \le \lambda_0(T). \]
Finally define
\[ \lambda(T) = \begin{cases}
\sup_{t \in [T,1]} 2\lambda_0(t) & \text{ if } T \le 1, \\
\sup_{t \in [1,T]} 2 \lambda_0(t)& \text{ if } T \ge 1.
\end{cases} \]
Then the function $\lambda$ is positive, continuous and satisfies the conditions of Proposition \ref{prop2}.

Furthermore, since $x_0$ belongs to $B(x,R/2)$ and $r<R$, it follows that
\[ B(x_0,\lambda_0(T) r) \subset B(x_T,\lambda(T) R). \]

Therefore, 
\begin{align*}
\frac{1}{r^{d-2}} \int_{B(x_0,\lambda_0 R)} e_\e(u_\e)(x,T) dx & \le  \left( \frac{R}{r} \right)^{d-2} \frac{1}{R^{d-2}} \int_{B(x_T,\lambda(T) R) } e_\e(u_\e)(x,T) dx \\
& \le \left( \frac{R}{r} \right)^{d-2} \tilde{\eta} |\ln \e| \le \left( \frac{R}{\sqrt{T_0}} \right)^{d-2} \tilde{\eta} |\ln \e|. 
\end{align*}

Choosing $\ds T_0=\left(\frac{2}{\eta_1} \right)^{\frac{2}{d-2}} \tilde{\eta}^{\frac{2}{d-2}} R^2$, we obtain
\begin{equation}
\label{zz16}
\frac{1}{r^{d-2}} \int_{B(x_0,\lambda_0(T) R)} e_\e(u_\e)(x,T) dx \le \frac{\eta_1}{2} |\ln \e|.
\end{equation}
We finally combine (\ref{zz17}), (\ref{zz18}) and (\ref{zz16}), and the claim is proved. The conclusion then follows from Proposition \ref{prop1}.
\end{proof}

The proofs of both Theorems \ref{th2} and \ref{th3} are now exactly the same as in \cite{Be1} (see sections 5.2 to 5.4) since computations are made on compact domains $\Omega \subset \m R^d \times (0,+\infty)$, where the bounds \eqref{zz25}, \eqref{bd:dtu_L2} and \eqref{zz28} hold.  As in the proof of Theorem \ref{th1} (and explained at its beginning), these are the only bounds which are made use of, along with the monotonicity formula and the identity \eqref{lem2.2}.

For the convenience of the reader, we remind a few elements of the proofs.

\begin{proof}[Outline of the proof of Theorem \ref{th2}]
One writes $u_\e = |u_\e| e^{i \varphi_\e}$, and starts with the equation for the phase $\varphi_\e$. As $|u_\e| \ge 1/2$, it is a uniformly parabolic equation. After performing a spatial cut-off in order to get rid of boundary conditions on $\Lambda$, one splits $\varphi_\e = \varphi_{\e,0} + \varphi_{\e,1}$, where $\varphi_{\e,0}$ is a solution to the \emph{linear} heat equation  (which turns out to be the desired phase $\Phi_\e$) and $\varphi_{\e,1}$ is a solution to a nonlinear equation \emph{with 0 initial data}. 

$\varphi_{\e,0}$ admits improved bounds due to parabolic regularity
\[ \| \nabla \varphi_{\e,0} \|_{L^2 L^{2^*}(\Lambda_{3/4})}^2 + \| \nabla \varphi_{\e,0} \|_{L^{\infty}(\Lambda_{3/4})}^2 \le C(\Lambda) \int_\Lambda e_\e(u_\e). \]
$\varphi_{\e,1}$ is shown to be essentially a perturbation (via a fixed point argument), and so $\varphi_\e$ enjoys integrability bounds in $L^2L^q(\Lambda_{3/4})$ for some $q >2$ (depending on $d$).

Plugging this information in the equation for $|u_\e|$, one infers that
\begin{gather}
\label{eq:th2_1}
e_\e(u_\e) \le \frac{1}{2} |\nabla \varphi_{\e,0}|^2 + \kappa_\e, \quad \text{where} \quad \int_{\Lambda_{3/4}} \kappa_\e \le C(\Lambda) M_0 \e^\alpha,
\end{gather}
for some $\alpha$ depending only on $d$.

An extra ingredient is provided by a result by Chen and Struwe \cite{Che} where the bound \eqref{eq11} is proved under an extra assumption of small energy. Together with \eqref{eq:th2_1} and a scaling argument, one can relax the small energy assumption, and prove \eqref{eq11}.

Finally, inserting \eqref{eq11} in the equation for $|u_\e|$, Lemma \ref{lem1.1} allows to improve the estimates to 
\[ |\nabla |u_\e|| + V_\e(u_\e) \le C(\lambda) \e^\alpha, \]
and from there, obtain $L^\infty$ bounds on $\varphi_{\e,1}$: this yields the bounds \eqref{eq14} and \eqref{eq:th2:3}, and completes the proof of Theorem \ref{th2}.
\end{proof}

\begin{proof}[Outline of the proof of Theorem \ref{th3}]
It suffices to prove the estimates on any cylinder of the form $\Lambda = B \times [T_0,T_1]$ where $B$ is a ball and $T_1 > T_0 > 0$. Up to increasing slightly the cylinder (and due to \eqref{zz26}), one can assume that 
\begin{gather} \label{eq:th3_en}
\int_{\Lambda} e_\e (u_\e) dxdt +\int_{\partial \Lambda} e_\e (u_\e) d\sigma(x,t) \le M_1 \ln \e.
\end{gather}
Denote now $\delta$ and $\delta^*$ the (\emph{space time}) exterior derivatives on $\m R^{d+1}$. The main step is a Hodge-de Rham type decomposition with estimates:
\begin{gather} \label{eq:th3_HdR}
u_\e \wedge \delta u_\e = \delta \Phi + \delta^* \Psi + \zeta,
\end{gather}
where a key feature is to prove that in addition to the expected bound
\[  \| \nabla \Phi \|_{L^2(\Lambda)} +  \| \nabla \Psi \|_{L^2(\Lambda)} \le C(\Lambda) | \ln \e |, \]
one also has for any $\ds p \in \left[1, 1+ \frac{1}{d}\right)$
\begin{gather} \label{eq:th3_psi} \| \nabla \Psi \|_{L^p(\Lambda)} \le C(p,\Lambda) M_2, \quad \text{and} \quad \| \zeta \|_{L^p(\Lambda)} \le C(p,\Lambda) \sqrt{\e}.
\end{gather}
In other words, $\Phi$ completely accounts for the lack of compactness.

The argument for \eqref{eq:th3_HdR} is as follows: one starts with the usual Hodge-de Rham decomposition on $\partial \Lambda$ (which is simply connected: here we use that $(d+1) -1 \ge 2$) so that
\[ u_\e \wedge \delta u_\e = d_{\partial \Lambda} \Phi_{\partial \Lambda} + d_{\partial \Lambda}^* \Psi_{\partial \Lambda}. \]
Then on $\partial \Lambda$, $-\Delta_{\partial \Lambda} \Psi_{\partial \Lambda} = 2 J_{\partial \Lambda} u_\e$: the Jerrard and Soner estimate on the Jacobian \eqref{eq:JS_Jac} gives the $1/\ln \e$ gain 
\[ \| J_{\partial \Lambda} u_\e \|_{C^{0,\alpha}(\partial \Lambda)'} \le C(\Lambda), \]
from where a similar gain is derived on $\Psi_{\partial \Lambda}$.

Let $\Phi_0$ be the harmonic extension of $\Phi_{\partial \Lambda}$ on $\Lambda$: we gauge away by considering on $\Lambda$
\[ v_\e := u_\e e^{- i \Phi_0}. \]
One now considers on $\Lambda$ the Hodge-de Rham decomposition $v_\e \wedge \delta v_\e = \delta \Phi_1 + \delta^* \Psi$, so that
\[ u_\e \wedge \delta u_\e = \delta \Phi_1 + \delta^* \Psi + \delta \Phi_0 + (1-|u_\e|^2) \delta \Phi_0. \]
Define $\zeta := (1-|u_\e|^2) \delta \Phi_0$: it satisfies \eqref{eq:th3_psi} using the bound on the energy \eqref{eq:th3_en}, and $\zeta|_{\partial \Lambda}$ as well.

For $\Psi$, one has again  $-\Delta \Psi  = 2 Jw_\e$, now with boundary conditions involving $\Psi_{\partial \Lambda}$ and $\zeta|_{\partial \Lambda}$. A Jacobian estimate similar to \eqref{eq:JS_Jac} (and proven in \cite[Proposition II.1]{BeOr02}) thus yields the bound \eqref{eq:th3_psi} on $\Psi$. Finally defining $\Phi := \Phi_0 + \Phi_1$ completes the decomposition.

\bigskip

With \eqref{eq:th3_HdR} and \eqref{eq:th3_psi} at hand, we now define $\Phi_\e$ by writing $\Phi = \Phi_\e + \Phi_2$, where 
\[ \partial_t \Phi_2 - \Delta \Phi_2 = \partial_t \Phi - \Delta \Phi \quad \text{on } \Lambda, \]
with the boundary condition $\Phi_2 = 0$ on $B \times \{ T_0 \} \cup \partial B \times (T_0,T_1)$. In particular, $\Phi_\e $ solves the (homogeneous) heat equation on $\Lambda$. 

From parabolic estimates,
\[ \| \nabla \Phi_\e \|_{L^\infty(\Lambda)} \le C \| \nabla_{x,t} \Phi \|_{L^2(B \times \{ T_0 \} \cup \partial B \times (T_0,T_1))} \le C(\Lambda) \ln \e, \]
which is the first estimate (iii). It remains to bound $w_\e := u_\e e^{- i \Phi_\e}$ in $\dot W^{1,p}(\Lambda)$.  
First we separate between $A_\Lambda = \{ ((x,t) \in \Lambda \mid |1-|w_\e(x,t)| | \le \e^{1/4} \}$ and $B_\Lambda = \Lambda \setminus A_\Lambda$. On $B_\Lambda$, the rough estimate $|\nabla w_\e(x,t)| \le C/\e$ inherited from \eqref{zz28} yields
\begin{align*}
\| \nabla w_\e \|_{L^p(B_\Lambda)}^p & \le C(\Lambda) \frac{1}{\e^p} \int_{\Lambda} \frac{(1-|w_\e(x,t)|^2)^2}{\e^{1/2}} dxdt \\
& \le C(\Lambda) \e^{2-p-1/2} \int_{\Lambda} \frac{(1-|u_\e(x,t)|^2)^2}{\e^2} dxdt \le C(\Lambda) \e^{3/2-p} \ln \e \to 0
\end{align*}
(observe that $|u_\e| = |w_\e|$ and $\ds p < \frac{d+1}{d} \le \frac{3}{2}$). We now work on $A_\Lambda$.
Then (with the same computation as in the proof of Theorem \ref{th1}), we have
\begin{align*}
|\nabla w_\e|^2 \le 2 |w_\e|^2 |\nabla w_\e|^2 = 2  |w_\e|^2 |\nabla |w_\e||^2 + 2| w_\e \wedge \nabla w_\e |^2 \le 2 |\nabla (|u_\e|^2)|^2 + 2| w_\e \wedge \nabla w_\e |^2.
\end{align*}
It follows from \eqref{eq:th3_psi} (writing $w_\e$ in terms of $u_\e$) that $\| w_\e \wedge \nabla w_\e \|_{L^p(\Lambda)} \le C(\Lambda)$. Then one actually bounds $\|  \nabla (|u_\e|^2) \|_{L^p(\Lambda)}$: for this, one writes the (uniformly) parabolic equation for $\rho = |u_\e|^2$, from which we get
\[ \int |\nabla \rho|^2 \chi \le \int (1-\rho) | \nabla u_\e|^2 \chi + \int \left( |\nabla \rho| |\nabla \chi|  + \partial_t \rho| \chi  \right) |1-\rho|, \]
(where $\chi$ is a suitable non negative cut-off function, $\chi =1$ on $\Lambda$). To bound the right-hand side, we split again between $A_\Lambda$ and $B_\Lambda$: $B_\Lambda$ has small measure (arguing as before) and on $A_\Lambda$ one uses smallness of $|1-\rho|$. Gathering all gets a final bound
\[ \| |\nabla (|u_\e|^2) \|_{L^p(\Lambda)}^p \le C(\Lambda) \e^{1/4-p/8} \ln \e \to 0. \]
This gives estimate (iv).

\end{proof}
%\bigskip

%The same remark also applies for the proof of Theorem \ref{th2}: the only difference with \cite{Be1}, which does not raise any difficulty, stands in the fact that we are working in two dimensions, where Sobolev is a little bit different. A complete proof is given in annex.

\section[Limiting measure and concentration set]{Description of the limiting measure and concentration set}

Let $u_\e$ be a solution of \eqref{pgl} satisfying the initial data \eqref{in}. 

%Our goal in this Section is to complete the proof of Theorem \ref{thA}.

Our goal in this section is to study the asymptotic limit, as $\e \rightarrow 0$, of the Radon measures $\mu_\e$ defined on $\m{R}^d \times [0,+\infty)$ by
\[ \mu_\e(x,t) = \frac{e_\e(u_\e)(x,t)}{|\ln \e|}\ dxdt. \]
To that purpose, we will study their time slices $\mu_\e^t$ defined on $\m{R}^d \times \{t\}$ by
\[ \mu_\e^t(x) = \frac{e_\e(u_\e)(x,t)}{|\ln \e|}\ dx. \]
When $\e \to 0$, these measures converge to $\mu_*$ and $\mu_*^t$ respectively (up to a subsequence), as it is shown in the next paragraph.

Then we will study the evolution of $\mu_*^t$, and show that it follows Brakke's weak formulation of the mean curvature flow. For this, we will of course heavily rely on the properties of $\mu_\e^{t}$ obtained in the previous section.

\subsection{Absolute continuity with respect to time of the limiting measure}

According to inequality (\ref{zz25}) , we have for all $R>0$ and $T>0$,
\begin{equation}
\label{equaH0}
\int_{B(0,R) \times [0,T]} d\mu_{\e}(x,t) \le C(d) T (1+T) (1+R)^d M_0.
\end{equation}

The bound \eqref{equaH0} yields a limiting measure via a diagonal extraction argument. This can also be done simultaneously for the time sliced measures: more precisely, following the proof in Brakke \cite{Br} word for word, we have the following.

\begin{thm}
\label{te1}
There exist a sequence $\e_m \rightarrow 0$, a Radon measure $\mu_*$ defined on $\m{R}^d \times [0,+\infty)$,  bounded on compact sets and, for each $t \ge 0$, a Radon measure $\mu_{*}^t$ on $\m{R}^d \times \{t\}$ such that:
\begin{gather}
\mu_{\e_m} \rightharpoonup \mu_{*} \quad \text{as } m \rightarrow \infty, \quad \text{and for all } t \ge 0, \quad \mu_{\e_m}^t \rightharpoonup \mu_{*}^t \quad \text{as } m \rightarrow \infty. 
\end{gather}
Moreover, the $(\mu_*^s)_s$ enjoy the bound
\begin{equation}
\label{Eq1}
\forall R>0, \forall t>0,\quad \mu_{*}^t(B(0,R)) \le C(d) (1+t) (1+R)^d M_0.
\end{equation}
and
\[ \mu_* = \mu_*^t dt. \]
\end{thm}

For the proof of Theorem \ref{te1}, we will need a few classical identites for the evolution of $\mu_\e^t$.

Before we dive into the study of the singular measure $\mu_*^t$,  let us state two useful following identities.

\begin{lem}%{(similar to Lemma 2.1, \cite{Be2}).}
\label{lll2.1}
Let $u_\e^0 \in L^{\infty}(\m{R}^d)$ and $u_\e$ be the associated solution of \eqref{pgl}. Then, for all $\chi \in \q D(\m{R}^d)$ and for all $t \ge 0$, we have
\begin{equation}
\label{equa2.1}
\frac{d}{dt} \int_{\m{R}^d} \chi(x) d\mu_\e^t = - \int_{\m{R}^d \times \{t\}} \chi(x) \frac{|\partial_t u_\e|^2}{|\ln \e|} \ dx + \int_{\m{R}^d \times \{t\}} \nabla \chi(x) \frac{-\partial_t u_\e.\nabla u_\e}{|\ln \e|} \ dx. 
\end{equation} 
\end{lem}

We usually choose $\chi \ge 0$, and this choice makes the first term of the right-hand side non positive. To handle the second term, we provide another identity which involves the stress-energy tensor.

\begin{lem}%{(similar to Lemma 2.2, \cite{Be2}).}
\label{ll2.2}
Let $\vec{X} \in \q D(\m{R}^d,\m{R}^d)$. Then for all $t \ge 0$,
\begin{equation}
\label{x30}
\frac{1}{|\ln \e|} \int_{\m{R}^d \times \{t\}} \left( e_\e(u_\e)\delta_{ij} - \partial_i u_\e \partial_j u_\e \right) \partial_j X_i = -\int_{\m{R}^d \times \{t\}}\vec{X}. \frac{-\partial_t u_\e.\nabla u_\e}{|\ln \e|} \ dx. 
\end{equation}
(Here we use Einstein's convention of implicit sommation over repeated indices.)
\end{lem}

The proof of Lemma \ref{ll2.2} is given in \cite{Be2}, and involves the stress-energy matrix $A_\e$ given by\begin{gather}
\label{def:stress-energy}
 A_\e = A_\e(u_\e) :=  e_\e(u_\e) \Id - \nabla u_\e \otimes \nabla u_\e = T(u_\e)+V_\e(u_\e) \Id,
 \end{gather}
where the matrix $T(u)$ and the potential $V_\e$ are given by
\begin{equation}
\label{equa2.4}
T(u)= \frac{1}{2} |\nabla u|^2 \Id - \nabla u \otimes \nabla u, \quad V_\e(u)=\frac{(1-|u|^2)^2}{4 \e^2}.
\end{equation}

%In dimension two, the product $T_{ij}\frac{\partial X_i}{\partial x_j}$ has a particularly simple expression using complex notation. Denote
%\[ X=X_1+iX_2 \quad \text{ and } \quad \omega=|u_{x_1}|^2-|u_{x_2}|^2-2iu_{x_1}u_{x_2}.\]
%
%Then we have
%\[ \int_{\m{R}^2} T_{ij}(u)\frac{\partial X_i}{\partial x_j} = \mathop{\mathrm{Re}} \left(-\int_{\m{R}^2} \omega \frac{\partial X}{\partial \bar{z}}\right).\]

Combining  Lemma \ref{lll2.1} and  Lemma \ref{ll2.2}  with the choice $\vec{X}=\nabla \chi$, we get rid of the time derivative of the right hand side of (\ref{equa2.1}). More precisely

\begin{lem}%{(similar to Lemma 2.3, \cite{Be2}).}
\label{ll2.3}
\[ \frac{\partial}{\partial t} \int_{\m{R}^d} \chi(x) d\mu_{\e}^{t} = -\int_{\m{R}^d \times \{t\}} \chi(x) \frac{|\partial_t u_{\e}|^2}{|\ln \e|}\ dx +  \int_{\m{R}^d \times \{t\}} \frac{D^2 \chi \nabla u_\e.\nabla u_\e - \Delta \chi e_{\e}(u_\e)}{|\ln \e|} \ dx. \]
\end{lem}

\begin{proof}[Proof of Theorem \ref{te1}]
It boils down to the following claim.

\begin{claim}
\label{pr}
Let $T>0$. We consider a sequence $\e_m \rightarrow 0$. Then there exists a subsequence $\e_{\sigma(m)}$ such that for every $s \in [0,T]$, 
\[ \mu_{\e_{\sigma(m)}}^s \rightharpoonup \mu_{*}^s \quad \text{as } m \rightarrow \infty, \text{ in the sense of measures.} \]
\end{claim}

Indeed let us proof Theorem \ref{te1} assuming this claim holds. With $T_n =n$, the claim let us dispose of sequences $\e_{n,m} \to 0$ as $m \to +\infty$ such that $\mu^{s}_{\e_{n,m}}$ enjoys the desired convergence on $s \in [0,T_n]$. By an argument of diagonal extraction, we deduce that there exist a sequence $\e_m \rightarrow 0$ and, for each $s \ge 0$, a measure $\mu_{*}^s$ on $\m{R}^d \times \{s\}$ such that for every $s \ge 0$, $\mu_{\e_m}^s \rightharpoonup \mu_{*}^s$ as $m \rightarrow \infty$.

We now prove the claim. We use the following lemma, which is an easy variant of Helly's selection principle.

\begin{lem}%{(similar to Lemma 5.4, \cite{Be2}).}
\label{le2}
Let $I$ be an at most countable set, and let $(f_i^m)_{m \in \m{N},i \in I}$ be a collection of real-valued functions defined on some interval $(a,b)$. Assume that for each $i \in I$, the family $(f_i^m)_{m \in \m{N}}$ is equibounded and satisfies the following semi-decreasing property

\begin{gather}
 \text{For all } \delta >0, \text{ there exist } \tau>0 \text{ and } m_i \in \m{N} \text{ such that, if } s_1,s_2 \in (a,b)
\nonumber \\
\label{e1}
\text{ and } s_2-\tau \le s_1 \le s_2, \text{ then for all } m \le m_i, \quad  f_m^i(s_2) \le f_m^i(s_1)+ \delta.
\end{gather}

Then there exist a subsequence $\sigma (m)$ and a family $(f^i)_{i \in I}$ of real-valued functions on $(a,b)$ such that for all $s \in (a,b)$ and $i \in I$,
\[ f_{\sigma(m)}^i(s) \rightarrow f^i(s). \]
\end{lem}

Let $(\chi_i)_{i \in I}$ be a countable family of compactly supported non-negative smooth functions on $\m{R}^d$; assume that for all $i \in I$, $0 \le \chi_i \le 1$, and that $\Span (\chi_i)_{i \in I}$ is dense in $\q C_c^0(\m{R}^d)$.
Let $m_0$ be such that if $m \ge m_0$, then $\ds \e_m \le \frac{1}{2}$.
We define for $m \in \m N$, $i \in I$ the function $f_m^i$ defined on $[0,T]$ by
\[ f_m^i(s)= \int_{\m{R}^d} \chi_i d\mu_{\e_m}^s. \] 

\emph{Step 1.} We first show that $(f_i^m)_{m \ge m_0}$ satisfies (\ref{e1}). 

Let $i \in I$. Recalling  Lemma \ref{ll2.3}, we have  
\begin{multline*}
\frac{d}{ds} \int_{\m{R}^d} \chi_i d\mu_{\e_m}^s \\
 = - \int_{\m{R}^d \times \{s\}} \chi_i(x) \frac{|\partial_t u_{\e_m}|^2}{|\ln \e_m|} \ dx + \frac{1}{|\ln \e_m|} \int_{\m{R}^d \times \{s\}} (D^2 \chi_i \nabla u_{\e_m}.\nabla u_{\e_m}-\Delta \chi_i \ e_{\e_m}(u_{\e_m}))\ dx. 
 \end{multline*}
Therefore 
\[ \frac{d}{ds} \int_{\m{R}^d} \chi_i d\mu_{\e_m}^s  \le \frac{1}{|\ln \e|} \int_{\m{R}^d \times \{s\}} (D^2 \chi_i \nabla u_{\e_m}.\nabla u_{\e_m}-\Delta \chi_i \ e_{\e_m}(u_{\e_m}))\ dx. \]
Let $R>0$ such that $\Supp(\chi_i) \subset B(0,R)$. We have, for $s \in [0,T]$
\begin{align}
\frac{d}{ds} f_m^i(s) & = \frac{d}{ds} \int_{\m{R}^d} \chi_i d\mu_{\e_m}^s \le  3 \|\chi_i\|_{\q C^2} \int_{B(0,R) \times \{s\}} \frac {\ e_{\e_m}(u_{\e_m})}{|\ln \e_m|}\ dx \nonumber  \\
&\le  3 \|\chi_i\|_{\q C^2} M_0 C(T,R), \label{eq:fmis}
\end{align}
by Proposition \ref{Prop1}.

Let $\delta>0$. We set $\ds \tau=\frac{\delta}{3 \|\chi_i\|_{C^2}M_0 C(T,R)}$.
 
If $s_1,s_2 \in [0,T]$ and $s_2-\tau \le s_1 \le s_2$, the inequality \eqref{eq:fmis} leads to
\[ f_m^i(s_2)-f_m^i(s_1) \le 3 \|\chi_i\|_{\q C^2} (s_2-s_1)M_0 C(T,R) \le (s_2-s_1)\frac{\delta}{\tau} \le \delta. \]
Therefore
\[ \forall m \le m_0, \quad f_m^i(s_2) \le f_m^i(s_1)+ \delta. \]
Now, we prove that $\forall i \in I$, the family $(f_m^i)_{m \ge m_0}$ is equibounded.

Let $i \in I$. Let $R>0$ such that $\Supp(\chi_i) \subset B(0,R)$ and $t \in [0,T]$. We have
\[ \int_{\m{R}^d} \chi_i d\mu_{\e_m}^t \le \int_{B(0,R)} d\mu_{\e_m}^t \le M_0 C(T,R). \]
Therefore, 
\[ \forall m \ge m_0, \quad  \|f_m^i\|_{L^\infty} \le C(T,R), \]

i.e. the family $(f_m^i)_{m \ge m_0}$ is equibounded.
According to Lemma \ref{le2}, there exist a subsequence $\sigma(m)$ and for all $i\in I$, a Radon measure $f^i$ such that for all $i \in I$, and $s \in (a,b)$
\[f_{\sigma(m)}^i(s) \rightarrow f^i(s). \]
Hence, for all $s \in [0,T]$, $\mu_{\e_{\sigma(m)}}^s (\chi_i)$ converges as $m \rightarrow +\infty$.

\bigskip

\emph{Step 2.}

Let $s_0 \in [0,T]$ be arbitrary but fixed, and $\chi \in C_c^0(\m{R}^d)$.
Let us show that $(\mu_{\e_{\sigma(m)}}^{s_0} (\chi))_{m \in \m{N}}$ is a Cauchy sequence in $\m{R}$.

Let $\alpha>0$. Since $\Span(\chi_i)$ is dense in $C_c^0(\m{R}^d)$, there exist a finite subset $J \subset I$, and real numbers ${(\lambda_j)}_{j \in J}$ such that 
\[ \left\| \sum_{j \in J} \lambda_j \chi_{j}-\chi \right\|_{L^\infty} \le \frac{\alpha}{3M_0C(T,R)}. \]

Let $m \ge k$ be two integers, we have
\begin{align*}
|\mu_{\e_{\sigma(m)}}^{s_0} (\chi)-\mu_{\e_{\sigma(k)}}^{s_0} (\chi)| 
&\le  \left| \mu_{\e_{\sigma(k)}}^{s_0} (\chi) - \sum_{j \in J} \lambda_j \mu_{\e_{\sigma(k)}}^{s_0} (\chi_j) \right|
+ \left| \sum_{j \in J} \lambda_j (\mu_{\e_{\sigma(m)}}^{s_0} (\chi_j)-\mu_{\e_{\sigma(k)}}^{s_0} (\chi_j)) \right| \\
 &\quad + \left| \sum_{j \in J} \lambda_j \mu_{\e_{\sigma(m)}}^{s_0} (\chi_j)-\mu_{\e_{\sigma(k)}}^{s_0} (\chi) \right|. 
 \end{align*}
 
Recall that \[ \mu_{\e_{\sigma(k)}}^{s_0} (\chi) - \sum_{j \in J} \lambda_j \mu_{\e_{\sigma(k)}}^{s_0} (\chi_j) = \int_{\m{R}^d}(\chi-\sum_{j \in J} \lambda_j \chi_{j}) d\mu_{\e_{\sigma(k)}}^{s_0}.\]

Let $R>0$ such that $\Supp(\chi)\subset B(0,R)$ and for all $j \in J$, $\Supp(\chi_j) \subset B(0,R)$.
We can bound
\begin{align*}
\left| \mu_{\e_{\sigma(k)}}^{s_0} (\chi) - \sum_{j \in J} \lambda_j \mu_{\e_{\sigma(k)}}^{s_0} (\chi_j) \right| &\le 
\left\| \chi-\sum_{j \in J} \lambda_j \chi_{j} \right\|_{L^ \infty} \int_{B(0,R)} d\mu_{\e_{\sigma(k)}}^{s_0} \\
&\le \left\| \chi-\sum_{j \in J} \lambda_j \chi_{j} \right\|_{L^\infty} M_0 C(T,R).
\end{align*}

Therefore we have
\[ \left| \mu_{\e_{\sigma(m)}}^{s_0} (\chi) - \sum_{j \in J} \lambda_j\ \mu_{\e_{\sigma(m)}}^{s_0} (\chi_j) \right| \le \frac{\alpha}{3}\]
and 
\[ \left| \mu_{\e_{\sigma(k)}}^{s_0} (\chi) - \sum_{j \in J} \lambda_j \mu_{\e_{\sigma(k)}}^{s_0} (\chi_j) \right| \le \frac{\alpha}{3}.\]
Since the sequence $\mu_{\e_{\sigma(m)}}^{s_0} (\chi_j)$ converges for all $j \in J$, we have, for $k$ large enough, 
\[ \left| \sum_{j\in J} (\mu_{\e_{\sigma(m)}}^{s_0} (\chi_j)-\mu_{\e_{\sigma(k)}}^{s_0} (\chi_j)) \right| \le \frac{\alpha}{3}.\]
Therefore, for $k$ large enough,
\[ |\mu_{\e_{\sigma(m)}}^{s_0} (\chi)-\mu_{\e_{\sigma(k)}}^{s_0} (\chi)| \le \alpha. \]
This proves that the family $(\mu_{\e_{\sigma(m)}}^{s_0} (\chi))_{n \in \m{N}}$ is a Cauchy sequence in $\m{R}$, thus it converges.
This determines the measure $\mu_{*}^{s_0}$ and establishes the convergence for $s=s_0$. Since $s_0$ was arbitrary, the conclusion of Theorem \ref{te1} follows.
\end{proof}

\subsection{The monotonicity formula on \texorpdfstring{$\mu_*^t$}{mu*t}}

The following result transfers the monotonicity formula on $\mu_\e^t$ to its singular limit $\mu_*^t$. In the case where the initial energy on the whole space is bounded by $M_0 |\ln \e|$, it is just an easy consequence of the monotonicity formula for $u_\e$ (\cite{Be1}). In our case, with our initial condition \eqref{in}, it requires some more elaborate computations.

\begin{prop}
\label{pro1}
For each $t>0,\ x \in \m{R}^d$ and $0<r \le \sqrt{t}$, we have
\[ \int_{\m{R}^d} \exp \left( -\frac{|x-y|^2}{4r^2} \right) d\mu_{\e}^{t-r^2}(y) \rightarrow \int_{\m{R}^d} \exp \left( -\frac{|x-y|^2}{4r^2} \right) d\mu_{*}^{t-r^2}(y) \quad \text{as } \e \rightarrow 0. \]
\end{prop}

\begin{proof}[Proof of Proposition \ref{pro1} :]

Let $t>0,\ x \in \m{R}^d$ and $0<r<\sqrt{t}$ fixed.

Since $\mu_{\e}^{t-r^2} \rightharpoonup \mu_{*}^{t-r^2}$ as measures, we have for all $\phi \in \q C_c^\infty(\m R^d)$
\[ \int_{\m{R}^d} \phi(y) \exp \left( -\frac{|y-x|^2}{8r^2} \right) d\mu_{\e}^{t-r^2}(y) \rightarrow \int_{\m{R}^d} \phi(y) \exp \left( -\frac{|y-x|^2}{8r^2} \right) d\mu_{*}^{t-r^2}(y).\]
For $x \in \m{R}^d$, $t>0$, and $0<r \le \sqrt{t}$ fixed, the measures $\ds \exp \left(-\frac{|y-x|^2}{8r^2} \right) d\mu_{\e}^{t-r^2}(y)$ are bounded independently of $\e$: indeed, we apply the monotonicity formula at the point $(x,t+r^2)$ between $\sqrt{2}r$ and $\sqrt{t+r^2}$.

\begin{align*}
\MoveEqLeft \int_{\m{R}^d} e_{\e}(u_{\e})(y,t-r^2)  \exp \left(-\frac{|y-x|^2}{8r^2} \right)dy \\
& \le  \left(\frac{2r^2}{t+r^2} \right)^{\frac{d-2}{2}}  \int_{\m{R}^d} e_{\e}(u_{\e}^0)(y) \exp \left(-\frac{|y-x|^2}{4(t+r^2)} \right) dy \le  C(t,r) M_0|\ln \e|,
\end{align*}
by Lemma \ref{H1:H0}. Therefore,
\begin{equation}
\label{eq3}
\int_{\m{R}^d} \exp \left( -\frac{|y-x|^2}{8r^2} \right) d\mu_{\e}^{t-r^2}(y) \le C(t,r).
\end{equation}
Hence, the measures $\ds \exp \left( -\frac{|y-x|^2}{8r^2} \right) d\mu_{\e}^{t-r^2}(y)$  are bounded independently of $\e$. On the other hand, they converge in measure to $\ds \exp \left( -\frac{|y-x|^2}{8r^2} \right) d\mu_{*}^{t-r^2}(y)$. 
As a consequence, the measure $\ds \exp \left( -\frac{|y-x|^2}{8r^2} \right) d\mu_{*}^{t-r^2}(y)$ is bounded on $\m{R}^d$ and for all $\phi \in \q C_b^\infty (\m R^d)$, there holds
\[ \int_{\m{R}^d} \phi(y) \exp \left( -\frac{|y-x|^2}{8r^2} \right) d\mu_{\e}^{t-r^2}(y) \rightarrow \int_{\m{R}^d} \phi(y) \exp \left( -\frac{|y-x|^2}{8r^2} \right) d\mu_{*}^{t-r^2}(y). \]
In particular, for $\ds \phi(y) = \exp \left( -\frac{|y-x|^2}{8r^2} \right)$, we have
\[ \int_{\m{R}^d}  \exp \left(-\frac{|y-x|^2}{4r^2} \right) d\mu_{\e}^{t-r^2}(y) \rightarrow \int_{\m{R}^d} \exp \left(-\frac{|y-x|^2}{4r^2} \right) d\mu_{*}^{t-r^2}(y). \qedhere \]
\end{proof}

Then, we can get an analogous of the monotonicity formula on $\mu_*^t$.

\begin{prop}%{(similar to Lemma 2, Part II, \cite{Be1}).}
\label{lem2}
For each $t>0$ and $x \in \m{R}^d$, the function $\mathcal{E}((x,t),.)$ defined on $(0,+\infty)$ by
\[ r \mapsto \mathcal{E}_\mu((x,t),r) = \frac{1}{r^{d-2}} \int_{\m{R}^d} \exp \left(-\frac{|x-y|^2}{4r^2} \right) d\mu_{*}^{t-r^2}(y) \]
is non-decreasing for $0<r \le \sqrt{t}$.
\end{prop}

\begin{proof}
It is a direct consequence of Proposition \ref{pro1} and of the monotonicity formula for $u_\e$.
\end{proof}

%\subsection{Bounds for the energy}
%The result of Proposition \ref{pro1} is also true for $t=r^2$, i.e. 
%\begin{equation}
%\label{eq4}
%\int_{\m{R}^d} \exp \left( -\frac{|x-y|^2}{4r^2} \right) d\mu_{\e}^{0}(y) \rightarrow \int_{\m{R}^d} \exp \left( -\frac{|x-y|^2}{4r^2} \right) d\mu_{*}^{0}(y) \quad \text{as }\e \rightarrow 0. 
%\end{equation}
%Indeed, the measures $\ds y \mapsto \exp \left( -\frac{|x-y|^2}{8r^2} \right) d\mu_{\e}^{0}(y)$ are bounded independently of $\e$, since we have, using \eqref{in},
%\[ \int_{\m{R}^d} \exp \left( -\frac{|x-y|^2}{8r^2} \right) d\mu_{\e}^{0}(y) \le C(t) M_0. \]
From there, we can get the following energy bound. 
\begin{prop}
For all $x \in \m{R}^d$ and $T >0$, we have
\begin{equation}
\label{eq5}
\forall t \in (0,T], \ s \in [0,T], \quad \int_{\m{R}^d} \exp \left( -\frac{|x-y|^2}{4t} \right)  d\mu^{s}_{*} \le C(T) M_0.
\end{equation}
\end{prop}

\begin{proof}
Due to the monotonicity formula \eqref{MF} and Lemma \ref{H1:H0}, 
\begin{align*}
 \int_{\m{R}^d}  \exp \left( -\frac{|y-x|^2}{4t} \right) e_{\e}(u_\e)(y,s) dy  & \le \left( \frac{t^2}{s +t^2} \right)^{\frac{d-2}{2}}  \int_{\m{R}^d} \exp \left( -\frac{|y-x|^2}{4(s+t^2)} \right)  e_{\e}(u_{\e}^0)(y) dy \\
 &\le  C(t^2+s) M_0|\ln \e|. 
 \end{align*}
Thus,
 \[\int_{\m{R}^d}\exp \left( -\frac{|x-y|^2}{4t} \right) d\mu_{\e}^{s} \le C(t^2+s)  M_0 \le C(T) M_0. \]
Letting $\e \rightarrow 0$, the conclusion follows by Proposition \ref{pro1}. 
\end{proof}

%[Proof of Theorem \ref*{thA}]{Proof of Theorem \ref{thA}}

%We can now gather all the information we obtained on $\mu_*$ and $\Sigma_\mu$ together, and t
The remainder of this section is devoted to prove Theorem \ref{thA}.

\subsection{Densities and concentration set}

In order to analyse geometric properties of the measures $\mu_{*}$ and $\mu_{*}^t$, a key point is the concept of densities. For a given Radon measure $\nu$ on $\m{R}^d$, we have the classical definition:

\begin{defi}
\label{def1}
For $m \in \m{N}$, the $m$-dimensional lower density of $\nu$ at the point $x$ is defined by 
\[ \Theta_{*,m}(\nu,x) = \liminf_{r \rightarrow 0} \frac{\nu(B(x,r))}{\omega_m r^m}, \]

where $\omega_m$ denotes the volume of the unit ball $B^m$. Similarly, the $m$- dimensional upper density $\Theta_m^{*}(\nu,x)$ is given by
\[ \Theta_m^{*}(\nu,x) = \limsup_{r \rightarrow 0} \frac{\nu(B(x,r))}{\omega_m r^m}. \]

When both quantities coincide, $\nu$ admits a $m$-dimensional density $\Theta_m(\nu,x)$ at the point $x$, defined as the common value.
\end{defi}

Since the energy measure is expected to concentrate on points, our main efforts will be devoted to the study of the density $\Theta_{*,0}(\mu_{*}^t,\cdot)$. Invoking the monotonicity formula once more, we have

\begin{lem}%{(similar to Lemma 3, Part II, \cite{Be1}).}
\label{lem3}
For all $x \in \m{R}^d$ and for all $t>0$,
\[ \Theta_{*,d-2}(\mu_{*}^t,x) \le \Theta_{d-2}^{*}(\mu_{*}^t,x) \le C(d) M_0 \frac{(1+t)^d}{t^{\frac{d-2}{2}}}<+\infty. \]
\end{lem}

\begin{proof}
Only the middle inequality is relevant. It is a consequence the bounds on $\mu_\e^t$ derived from Lemma \ref{H1:H0}, and Proposition \ref{prop1}. Indeed fix $t >0$ and $x \in \m R^d$; for $r >0$, there holds
\begin{align*}
\frac{\mu^t_*(B(x,r))}{r^{d-2}} & = \frac{1}{r^{d-2}} \int_{B(x,r)} d\mu^t_* \le \frac{e^{1/4}}{r^{d-2}} \int \exp \left( - \frac{|x-y|^2}{4r^2} \right) d\mu^t_* \\
& \le \lim_{\e \to 0} \frac{e^{1/4}}{r^{d-2}} \int \exp \left( - \frac{|x-y|^2}{4r^2} \right) d\mu^t_\e.
\end{align*}
Now, for any $\e >0$, the mononicity formula yields
\begin{align*}
\MoveEqLeft \frac{e^{1/4}}{r^{d-2}} \int \exp \left( - \frac{|x-y|^2}{4r^2} \right) d\mu^t_\e \le \frac{e^{1/4}}{(t+r^2)^{\frac{d-2}{2}}} \int \exp \left( - \frac{|x-y|^2}{4(t+r^2} \right) \frac{e_\e(u_\e^0)(y)}{\ln \e} dy \\
& \le \frac{e^{1/4}}{(t+r^2)^{\frac{d-2}{2}}} C(d) (1+ t +r^2)^d M_0.
\end{align*}
As this does not depend on $\e$, we infer that
\[ \frac{\mu^t_*(B(x,r))}{r^{d-2}} \le \frac{e^{1/4}}{(t+r^2)^{\frac{d-2}{2}}} C(d) (1+ t +r^2)^d M_0. \]
Taking the $\limsup$ in $r\to 0$, we get
\[ \Theta_{d-2}^{*}(\mu_{*}^t,x) \le C(d) e^{1/4} \frac{(1+t)^d}{t^{\frac{d-2}{2}}} M_0. \qedhere \]
\end{proof}

The previous lemma provides an upper bound. For regularity properties (of the concentration set) it is well known that lower bounds play a key role. However, it seems difficult to work with $\Theta_{*,d-2}(\mu_{*}^t, \cdot)$ directly (since the equation depends on time); instead, we will first consider parabolic densities (which involve space-time measures), whose definition is recalled below, and which is more natural in view of the monotonicity for $\mu_*$ (Proposition \ref{lem2}).

\begin{defi}
\label{def2}
Let $\nu$ be a Radon measure on $\m{R}^d \times [0,+\infty)$ such that $\nu=\nu^t dt$. For $t>0$ and $m \in \m{N}$, the parabolic $m$-dimensional lower density of $\nu$ at the point $(x,t)$ is defined by
\[ \Theta _{*,m}^P(\nu,(x,t)) = \liminf_{r \rightarrow 0} \frac{1}{r^m} \int_{\m{R}^d} \exp \left( -\frac{|x-y|^2}{4r^2} \right) d\nu^{t-r^2}(y). \]

The parabolic upper density and parabolic density are defined accordingly, and denoted respectively $\Theta_m^{P,*}$ and $\Theta_m^P$.
\end{defi}

It clearly follows from the monotonicity formula that for $\nu = \mu_*$ and $m=d-2$, the limit in definition \ref{def2} is decreasing, so that $\Theta_{d-2}^P(\mu^{*},(x,t))$ exists everywhere in $\m{R}^d \times (0,+\infty)$. Another consequence, is that the parabolic measure dominates the $d-2$-dimensional density for $\mu_*$ (see Proposition \ref{prop:Theta_P} for a precise statement). Motivated by this fact, we define
\begin{gather} \label{def:Sigma_mu}
\Sigma_{\mu} := \left\{ (x,t) \in \m{R}^d \times (0,+\infty) \text{ such that } \Theta_{d-2}^P(\mu_{*},(x,t))>0 \right\},
\end{gather}
and for $t>0$, 
\begin{gather} \label{def:Sigma_mut}
\Sigma_{\mu}^t:= \Sigma_{\mu} \cap (\m{R}^d \times \{t\}).
\end{gather}

We now state the properties we need on $\Sigma_\mu$ to conclude the proof of Theorem \ref{thA}, and postpone their proofs to the end of this subsection.

\subsection{Properties and regularity of $\Sigma_\mu$}

\subsubsection{Diffuse part of $\mu_*^t$ outside of the vorticity}

Let us first state an important consequence of the analysis carried out above in Section \ref{sec:cl-out}.

\begin{thm}%{(similar to Theorem 5, Part II, \cite{Be1}).}
\label{th5}
There exist an absolute constant $\eta_2>0$ and a positive continuous function $\lambda$ defined on $(0,+\infty)$ such that, if for $x \in \m{R}^d,\ t>0$ and $r>0$ we have
\begin{equation}
\label{equat2}
\mu_{*}^t(B(x,\lambda(t) r)) < \eta_2r^{d-2}, 
\end{equation}

then for every $s \in [t+\frac{15}{16}r^2,t+r^2]$, $\mu_{*}^t$ is absolutely continuous with respect to the Lebesgue measure on the ball $B(x,\frac{1}{4}r)$. More precisely
\[ \mu_{*}^s=|\nabla \Phi_{\dagger}|^2\ dx \quad \text{ on } B \left(x,\frac{1}{4}r \right) , \]
where $\Phi_{\dagger}$ satisfies the heat equation in $\Lambda_{\frac{1}{4}}=B(x,\frac{1}{4}r) \times \left[ t+\frac{15}{16}r^2,t+r^2 \right]$.
\end{thm}

Note that the constant $\eta_2$ and the function $\lambda$ are the same as in Proposition \ref{prop4} of the previous part. Notice also that $\mu_{*}=|\nabla \Phi_{\dagger}|^2dxdt$ on $\Lambda_{\frac{1}{4}}$, and that $|\nabla \Phi_{\dagger}|^2$ is a smooth function.

\begin{proof}
We briefly sketch the proof of Theorem \ref{th5}, which is a direct consequence of Theorems \ref{th1} and \ref{th2} of the previous Section. 

Let $x \in \m{R}^d,\ t>0\ ,r>0$ be fixed such that (\ref{equat2}) is verified, then for $\e$ small enough, we have 
\[ \int_{B(x,\lambda(t) r)} e_\e(u_\e) < \eta_2 r^{d-2} |\ln \e|,\]

so that we may invoke Proposition \ref{prop4}. This yields
\[ e_\e(u_\e) =|\nabla \Phi_\e|^2 + \kappa_\e \quad \text{ in } \Lambda_{\frac{1}{4}}, \]

where $\Phi_\e$ verifies the heat equation in $\Lambda_{\frac{3}{8}}$ and we have the bounds 
\[ |\nabla \Phi_\e|^2 \le C(\Lambda)|\ln \e|, \quad \text{and} \quad |\kappa_\e| \le C(\Lambda) \e^{\beta} \quad \text{ in } \Lambda_{\frac{1}{4}}. \]

Extracting possibly a further subsequence, we may assume that
\[ \frac{\Phi_\e}{\sqrt{|\ln \e|}} \rightarrow \Phi_{\dagger} \quad \text{ uniformly on } \Lambda_{\frac{5}{16}}. \]

Since $\Phi_\e$ verifies the heat equation, it follows that for every $k \in \m{N}$,
\[ \frac{\Phi_\e}{\sqrt{|\ln \e|}} \rightarrow \Phi_{\dagger} \quad \text{ in } \mathcal{C}^k(\Lambda_{\frac{1}{4}}), \]

and $\Phi_{*}$ verifies the heat equation on $\Lambda_{\frac{1}{4}}$. On the other hand, 
\[ \kappa_\e \rightarrow 0 \text{ uniformly on } \Lambda_{\frac{1}{4}}, \]

so that \[ \frac{e_\e(u_\e)}{|\ln \e|} \rightarrow |\nabla \Phi_{\dagger}|^2 \quad \text{ uniformly on } \Lambda_{\frac{5}{16}}. \qedhere\]
\end{proof}

\subsubsection{Clearing-out}

Following Brakke \cite{Br} and Ilmanen \cite{Il}, the main tool in the study of geometric properties of $\Sigma_\mu$ is the following result.

\begin{thm}[Clearing-out]%{(similar to Theorem 6, Part II, \cite{Be1}).}
\label{th6}

There exists a positive continuous function $\eta_3$ defined on $(0,+\infty)$ such that, for any $(x,t) \in \m{R}^d \times (0,+\infty)$ and any $0<r<\sqrt{t}$, if
\[ \mathcal{E}_\mu((x,t),r) = \frac{1}{r^{d-2}} \int_{\m{R}^d} \exp(-\frac{|x-y|^2}{4r^2})\ d\mu_{*}^{t-r^2}(y) \le \eta_3(t-r^2), \]
then
\[ (x,t) \notin \Sigma_\mu. \]
\end{thm}

%We stress out the fact that in our theorem, $\eta_3$ is only a constant and not a function as in Theorem 6 of Part II of \cite{Be1}, since our work is done on the $d$-vortex, a much more specific initial data.
An immediate corollary is

\begin{cor}%{(similar to Corollary 1, Part II, \cite{Be1}).}
\label{cor1}
For any $(x,t) \in \Sigma_{\mu}$, we have
\[ \Theta_{d-2}^P(\mu_{*},(x,t)) \ge \eta_3(t). \]
\end{cor}

The remainder of this paragraph is devoted to the proof of Theorem \ref{th6}, which is essentially a consequence of Theorem \ref{th5}. We first need two preliminary lemmas.

\begin{lem}%{(similar to Lemma 6.1, Part II, \cite{Be1}).}
\label{lem6.1}
Let $(x,t) \in \Sigma_\mu$ and $0<r<\sqrt{t}$. Then, we have
\[ r^{2-d} \mu_{*}^{t-r^2}(B(x,\lambda(t-r^2) r)) \ge \eta_2, \]
where $\eta_2$ is the constant in Theorem \ref{th5}.
\end{lem}

\begin{proof}

Indeed, assume by contradiction that
\[ r^{2-d} \mu_{*}^{t-r^2}(B(x,\lambda(t-r^2) r)) < \eta_2. \]

Then, by Theorem \ref{th5}, for every $\tau \in \left[ t-\frac{1}{16}r^2,t \right]$
\[ \mu_{*}^{\tau}= |\nabla \Phi_{\dagger}|^2\ dx \quad \text{on } B \left( x,\frac{r}{4} \right), \]

where $\Phi_{\dagger}$ is smooth. We are going to show that
\begin{equation}
\label{equa6.1}
s^{2-d} \int_{\m{R}^d}\exp \left( -\frac{|x-y|^2}{4s^2} \right) \ d\mu_{*}^{t-s^2} \rightarrow 0 \quad \text{as } s \rightarrow 0. 
\end{equation}
Indeed, we write and compute
\begin{align}
s^{2-d} \int_{B(x,\frac{r}{8})}\exp \left( -\frac{|x-y|^2}{4s^2} \right) d\mu_{*}^{t-s^2} &\le s^{2-d} \|\nabla \Phi_{\dagger} \|_{L^{\infty}(B(x,\frac{1}{8}r))} \int_{\m{R}^d} \exp \left( -\frac{|x-y|^2}{4s^2} \right) dx \nonumber \\
&\le K \|\nabla \Phi_{\dagger}\|_{L^{\infty}(B(x,\frac{1}{8}r))} s^2 \rightarrow 0 \quad \text{as } s \rightarrow 0. \label{equa6.2}
\end{align}
On the other hand,
\begin{equation}
\label{equa6.3}
s^{2-d} \int_{\m{R}^d \setminus B(x,\frac{r}{8})} \exp \left( -\frac{|x-y|^2}{4s^2} \right) d\mu_{*}^{t-s^2} \rightarrow 0 \quad \text{as } s \rightarrow 0.
\end{equation}

Indeed,
\begin{align*}
\MoveEqLeft s^{2-d} \int_{\m{R}^d \setminus B(x,\frac{r}{8})} \exp \left(-\frac{|x-y|^2}{4s^2} \right) d\mu_{*}^{t-s^2}  = s^{2-d} \int_{\m{R}^d \setminus B(x,\frac{r}{8})} \left( \exp \left(-\frac{|x-y|^2}{8s^2} \right) \right)^2 d\mu_{*}^{t-s^2}  \\
&\le  s^{2-d}\int_{\m{R}^d \setminus B(x,\frac{r}{8})} \exp \left( -\frac{r^2}{8 \times 64s^2} \right) \exp \left(-\frac{|x-y|^2}{8s^2} \right) d\mu_{*}^{t-s^2} \\
&\le  s^{2-d} \exp \left( -\frac{r^2}{512s^2} \right) \int_{\m{R}^d} \exp \left( -\frac{|x-y|^2}{8s^2} \right)  d\mu_{*}^{t-s^2} \\
&\le  \left( \frac{2}{t+s^2} \right)^{\frac{d-2}{2}} \exp \left( -\frac{r^2}{512s^2} \right) \int_{\m{R}^d} \exp \left( -\frac{|x-y|^2}{4(t+s^2)} \right) d\mu_{*}^{0} 
\end{align*} 
by monotonicity formula at $(x,t+s^2)$ between $\sqrt{2}s$ and $\sqrt{t+s^2}$ (we can assume $s<\sqrt{t}$). 

Since 
\[ \int_{\m{R}^d}\exp \left( -\frac{|x-y|^2}{4(t+s^2)} \right) d\mu_{*}^{0} \le  M_0 \]
by inequality (\ref{eq5}), we have
\[ s^{2-d} \int_{\m{R}^d \setminus B(x,\frac{r}{8})} \exp \left( -\frac{|x-y|^2}{4s^2} \right) d\mu_{*}^{t-s^2}  \le  \left( \frac{2}{t} \right)^{\frac{d-2}{2}} M_0 \exp \left( -\frac{r^2}{512s^2} \right) \rightarrow 0 \quad \text{as } s \rightarrow 0.\]

Combining (\ref{equa6.2}) and (\ref{equa6.3}), (\ref{equa6.1}) follows and hence $\Theta_{d-2}^P(\mu_{*},(x,t))=0$, i.e. $(x,t) \notin \Sigma_\mu$, which is a contradiction.
\end{proof}

\begin{lem}%{(similar to Lemma 6.2, Part II, \cite{Be1}).}
\label{lem6.2}
The function $(x,t) \rightarrow \Theta_{d-2}^P(\mu_{*},(x,t))$ is upper semi-continuous on the set
$\m{R}^d \times (0,+\infty)$.
\end{lem}

\begin{proof}
Let $(x,t) \in \m{R}^d \times (0,+\infty)$, and let $(x_m,t_m)_{m \in \m{N}}$ be a sequence such that $(x_m,t_m) \rightarrow (x,t)$. We are going to show that

\begin{equation}
\label{equa6.4}
\limsup_{m \rightarrow +\infty} \Theta_{d-2}^P(\mu_{*},(x_m,t_m)) \le \Theta_{d-2}^P(\mu_{*},(x,t)). 
\end{equation}
Let $0<r<\frac{1}{2}\sqrt{t}$ be fixed for the moment. For $m$ sufficiently large, let $r_m=\sqrt{r^2+t_m-t}$, so that $t-r^2=t_m-r_m^2$. By the monotonicity formula, we have
\begin{align*}
\Theta_{d-2}^P(\mu_{*},(x_m,t_m)) &\le \frac{1}{r_m^{d-2}} \int_{\m{R}^d} \exp(-\frac{|y-x_m|^2}{4r^2})\ d\mu_{*}^{t_m-r_m^2}(y) \\
& \le  \frac{1}{r_m^{d-2}} \int_{\m{R}^d} \exp \left( -\frac{|y-x_m|^2}{4r^2} \right)\ d\mu_{*}^{t-r^2}(y).
\end{align*}
We assert that, as $m \to +\infty$,
\[  \frac{1}{r_m^{d-2}} \int_{\m{R}^d} \exp \left(-\frac{|y-x_m|^2}{4r^2} \right) d\mu_{*}^{t-r^2}(y) \rightarrow  \frac{1}{r^{d-2}}\int_{\m{R}^m} \exp \left( -\frac{|y-x|^2}{4r^2} \right) d\mu_{*}^{t-r^2}(y). \]
Indeed, since the sequence $(x_m)_m$ is bounded by a constant $M>0$, we write 
\begin{align*}
\frac{1}{r_m^{d-2}} \int_{\m{R}^d} \exp \left( -\frac{|y-x_m|^2}{4r^2} \right) d\mu_{*}^{t-r^2}(y) & = \frac{1}{r_m^{d-2}} \int_{B(0,2M)} \exp \left( -\frac{|y-x_m|^2}{4r^2} \right) d\mu_{*}^{t-r^2}(y) \\
& \quad + \frac{1}{r_m^{d-2}} \int_{B(0,2M)^c} \exp \left( -\frac{|y-x_m|^2}{4r^2} \right) d\mu_{*}^{t-r^2}(y). 
\end{align*}
Passing to the limit in the integral on $B(0,2M)$ does not raise any difficulty. For the second integral, we have
\[ |y-x_m| \ge |y|-|x_m| \ge |y|-M \ge \frac{1}{2}|y|, \]
so, if $y \in \m{R}^d \setminus B(0,2M)$,
 \[ \exp \left( -\frac{|y-x_m|^2}{4r^2} \right) d\mu_{*}^{t-r^2}(y) \le \exp \left( -\frac{|y|^2}{16 r^2} \right) d\mu_{*}^{t-r^2}(y), \]
and by the monotonicity formula at point $(0,t+3r^2)$ between $r_1=2r$ and $r_2=\sqrt{t+3r^2}$,
\[ \int_{\m{R}^d} \exp \left( -\frac{|y|^2}{16 r^2} \right) d\mu_{*}^{t-r^2}(y) \le \left( \frac{2r}{\sqrt{t+3r^2}} \right)^{d-2} \int_{\m{R}^d} \exp \left( -\frac{|y|^2}{4(t+3r^2)} \right) d\mu_{*}^{0}(y) \le C M_0, \]
by \eqref{eq5}. We then get the result as a consequence of Lebesgue's dominated convergence Theorem.
So if we let $m \to +\infty$, we obtain
\[ \limsup_{m \rightarrow +\infty} \Theta_{d-2}^P(\mu_{*},(x_m,t_m)) \le \int_{\m{R}^d} \exp \left( -\frac{|y-x|^2}{4r^2} \right) d\mu_{*}^{t-r^2}(y). \]
Now, let $r \rightarrow 0$, and (\ref{equa6.4}) follows.
\end{proof}

\begin{prop} \label{prop:Theta_P}
There exists an explicit constant $K$ such that for all $x \in \m R^d$ and $t >0$, 
\begin{equation}
\label{Eq3}
\Theta_{d-2}^P(\mu_{*},(x,t)) \ge K \Theta_{*,d-2}(\mu_{*}^t,x).
\end{equation}
In particular,
\[ \Theta_{*,d-2}(\mu_{*}^t,x) \equiv 0 \quad \text{ on } \m{R}^d \setminus \Sigma_{\mu}^t. \]
\end{prop}

\begin{proof}
Let $(x,t) \in \m{R}^d \times (0,+\infty)$ be given. Let $0<r<t$ be fixed for the moment. We write, for every $0<s<\sqrt{t}$,
\begin{align*}
\frac{1}{r^{d-2}} \mu_{*}^t(B(x,r)) & \le  \exp \left( \frac{1}{4} \right)  \frac{1}{r^{d-2}}  \int_{\m{R}^d} \exp \left( -\frac{|y-x|^2}{4r^2} \right) d\mu_{*}^{t}(y) \\
&\le  \exp \left( \frac{1}{4} \right) \frac{1}{(r^2+s^2)^{\frac{d-2}{2}}} \int_{\m{R}^d} \exp \left( -\frac{|y-x|^2}{4(r^2+s^2)} \right) d\mu_{*}^{t-s^2}(y),
\end{align*}
where we have used the monotonicity formula at the point $(x,t+r^2)$ between $r$ and $\sqrt{r^2+s^2}$ for the last inequality. Next, we choose $s=\sqrt{r}$. This yields
\begin{equation}
\label{equa6.5}
 \frac{1}{r^{d-2}}  \mu_{*}^t(B(x,r)) \le \exp \left(\frac{1}{4} \right) \frac{1}{(r^2+r)^{\frac{d-2}{2}}} \int_{\m{R}^d} \exp \left(-\frac{|y-x|^2}{4(r^2+r)} \right) d\mu_{*}^{t-r}(y).
\end{equation}
In the last integral, we decompose
\[ \m{R}^d = B(x,1) \cup (\m{R}^d \setminus B(x,1)). \]
On $B(x,1)$, observe that
\[ \exp \left( -\frac{|y-x|^2}{4(r^2+r)} \right) \le K \exp \left(-\frac{|y-x|^2}{4r} \right), \]
for some absolute constant $K$. On the other hand, on $\m{R}^d \setminus B(x,1)$, we have
\begin{equation}
\label{zz51}
\frac{1}{(r^2+r)^{\frac{d-2}{2}}}  \int_{\m{R}^d \setminus B(x,1)} \exp \left(-\frac{|x-y|^2}{4(r^2+r)} \right) d\mu_{*}^{t-r} \le \left( \frac{2}{t} \right)^{\frac{d-2}{2}} M_0 \exp \left(-\frac{1}{8(r^2+r)} \right). 
\end{equation}
Indeed,
\begin{align}
\MoveEqLeft \frac{1}{(r^2+r)^{\frac{d-2}{2}}}  \int_{\m{R}^d \setminus B(x,1)} \exp \left(-\frac{|x-y|^2}{4(r^2+r)} \right) d\mu_{*}^{t-r} \\
& \le \exp \left(-\frac{1}{8(r^2+r)} \right) \frac{1}{(r^2+r)^{\frac{d-2}{2}}}  \int_{\m{R}^d} \exp \left(-\frac{|x-y|^2}{8(r^2+r)} \right) d\mu_{*}^{t-r^2}(y) \nonumber \\
\label{zz21}
& \le \exp \left(-\frac{1}{8(r^2+r)} \right) \left( \frac{2}{2r^2+r+t} \right)^{\frac{d-2}{2}}  \int_{\m{R}^d} \exp \left(-\frac{|x-y|^2}{4(2r^2+r+t)} \right) d\mu_{*}^{0}(y) \\
\label{ch3:zz22}
& \le \left( \frac{2}{t} \right)^{\frac{d-2}{2}} \exp \left(-\frac{1}{8(r^2+r)} \right) \int_{\m{R}^d} \exp \left(-\frac{|x-y|^2}{4(2r^2+r+t)} \right) d\mu_{*}^{0}(y) 
\end{align}
where we have used the monotonicity formula at point  $(x,2r^2+r+t)$ between $\sqrt{2(r^2+r)}$ and $\sqrt{2r^2+r+t}$ for inequality (\ref{zz21}).
Then, combining  (\ref{eq5}) and (\ref{ch3:zz22}) leads to inequality (\ref{zz51}). Going back to (\ref{equa6.5}), we infer
 \[ \frac{1}{r^{d-2}}  \mu_{*}^{t}(B(x,r)) \le \frac{K}{r^{\frac{d-2}{2}}} \int_{\m{R}^d} \exp \left( -\frac{|x-y|^2}{4r} \right)\ d\mu_{*}^{t-r} + \left( \frac{2}{t} \right)^{\frac{d-2}{2}} M_0 \exp \left( -\frac{1}{8(r^2+r)} \right). \]
Letting $r$ go to zero, the conclusion follows.
\end{proof}

We can now complete the proof of Theorem \ref{th6}.

\begin{proof}[Proof of Theorem \ref{th6}]
Let $(x,t) \in \m{R}^d \times (0,+\infty)$ and $0<r<\sqrt{t}$. We have

\begin{equation}
\label{equa6.6}
r^{2-d} \mu_{*}^{t-r^2}(B(x,\lambda(t-r^2) r)) \le \exp \left( \frac{\lambda^2(t-r^2)}{4} \right) \mathcal{E}_\mu((x,t),r).
\end{equation}
Consider therefore the function
\[ \eta_3(s)=\exp \left( -\frac{\lambda^2(s)}{4} \right) \eta_2, \]
and assume next that, for some $0<r<\sqrt{t}$,
\[ \mathcal{E}_\mu((x,t),r) \le \eta_3(t-r^2). \]
Then, by (\ref{equa6.6}), 
\[ r^{2-d} \mu_{*}^{t-r^2}(B(x,\lambda(t-r^2) r)) \le \eta_2 \]
and the conclusion follows by Lemma \ref{lem6.1}.
\end{proof}

\subsubsection{Consequences: regularity of $\Sigma_\mu$ and decomposition of $\mu_*$}

At this stage, we are in position to derive the following conclusions, without invoking any further property of the equation \eqref{pgl}.

\begin{prop}%{(similar to Proposition 6, Part II, \cite{Be1}).}
\label{prop6}

\begin{enumerate}
\item  The set $\Sigma_\mu$ is closed in $\m{R}^d \times (0,+\infty)$. 

\item For any $t>0$ and $x \in \m R^d$,

\[ \mathcal{H}^{d-2} (B(x,1) \cap \Sigma_\mu^t) \le K(d) M_0 < +\infty.\] 

\item For any $t>0$, the measure $\mu_{*}^t$ can be decomposed as 
\[ \mu_{*}^t = g(x,t) \mathcal{H}^d \ +\ \Theta_{*}(x,t) \mathcal{H}^{d-2} \llcorner \Sigma_\mu^t, \]
where $g$ is some smooth function defined on $\m{R}^d \times (0,+\infty) \setminus \Sigma_\mu$ and $\Theta_{*}$ verifies the bound
\[ \Theta_{*}(x,t) \le K M_0 (1+t)^d t^{\frac{2-d}{2}}. \]
\end{enumerate}
\end{prop}

The function $\Theta_{*}$ is the Radon-Nikodym derivative of $\mu_{*}^t \llcorner \Sigma_\mu^t$ with respect to $\mathcal{H}^{d-2}$; at this stage we may just assert that it lies between the lower and upper densities.

Concerning $g$, it can be locally defined as $|\nabla \Phi_{\dagger}|^2$ for some smooth $\Phi_{\dagger}$ verifying the heat equation. The function $\Phi_{\dagger}$ however is not yet globally defined.

\begin{proof}%[Proof of Proposition \ref{prop6}] {~}

(1) In view of Corollary \ref{cor1}, we have

\[ \Sigma_\mu= \left\{ (x,t) \in \m{R}^d \times (0,+\infty) \mid \Theta_{d-2}^P(\mu_{*},(x,t)) \ge \eta_3(t) \right\}. \]
Since $\eta_3(\cdot)$ is continuous and since $\Theta_{d-2}^P(\mu_{*},\cdot)$ is upper semi-continuous by Lemma \ref{lem6.2}, we deduce that $\Theta_{d-2}^P(\mu_{*},\cdot)-\eta_3(\cdot)$ is upper semi-continuous as well and the conclusion follows.

\bigskip

(2) Let us first observe that the proof of in \cite{Be1}, which uses a scaling argument, does not yield the result. Instead, we will crucially rely on the fact that the function $\lambda$ which appears in Lemma \ref{lem6.1} has actually slow variations. Let $\ds 0<\delta<\sqrt{\frac{t}{2}}$. Consider a standard covering of $\m{R}^d$ such that
\[ \m{R}^d \subset \bigcup_{j \in J} B(x_j, \lambda(t) \delta), \quad \text{ and } \quad B \left(x_i, \lambda(t)\frac{\delta}{2} \right) \cap B \left( x_j,\lambda(t) \frac{\delta}{2} \right)= \varnothing \quad \text{for } i \ne j. \]
Define
\[ I_\delta=\left\{ i : \  B(x_i, \lambda(t) \delta) \cap B(x,1) \cap \Sigma_\mu^1 \ne \varnothing \right\}. \]

For $i \in I_\delta$, there exists some $y_i \in \Sigma_\mu^t \cap B(x_i, \lambda(t) \delta)$. Hence, by Lemma \ref{lem6.1},
\[  \mu_{*}^{t-\delta^2}(B(y_i,\lambda(t-\delta^2) \delta))> \eta_2 \delta^{d-2}, \]
and in particular 
\[  \mu_{*}^{1-\delta^2}(B(x_i,(\lambda(t-\delta^2)+\lambda(t)) \delta))> \eta_2 \delta^{d-2}.\]
Recall that there exists an absolute constant $C$ (not depending on $t$ or $\delta$) such that
\[ \forall \tau \in \left( t/2,2t \right), \quad \lambda(\tau)  \le C \lambda(t). \]
(See Proposition \ref{prop2} where the function was first defined.) It follows that
\[ \lambda(t-\delta^2)+\lambda(t) \le (C+1) \lambda(t). \]
Since the balls $\ds B \left( x_i,  \lambda(t) \frac{\delta}{2} \right)$ are disjoint, the balls $B(x_i, (C+1)  \lambda(t) \delta)$ cover at most $K(d)$ times the ball $B(x,1)$ (where the constant $K(d)$ only depends on $d$).
Therefore, using inequality (\ref{Eq1}),
\[ \delta^{2-d} \sum_{i \in I_\delta } \mu_{*}^{t-\delta^2} \left( B(x_i,(C+1) \lambda(t) \delta) \right) \le K(d) \delta^{2-d} \mu_{*}^{t-\delta^2} \left( B(x,1) \right) \le C(d) K M_0 \delta^{2-d}. \]
This implies that
\[ \eta_2 \Card I_\delta <  K(d) M_0 \delta^{2-d}. \]
Now by definition, we have
\[ \mathcal{H}^{d-2}(B(x,1) \cap \Sigma_\mu^t) \le \limsup_{\delta \rightarrow 0} ( \Card I_\delta)\delta^{d-2} \le K(d) M_0 . \]

%To conclude for all times, we invoke a scaling argument. For $t_0\ge 1$ fixed, consider the function
%\[ v_\epsilon(x,t)=u_\e(\sqrt{t_0} x,t_0 t)\]
%where $\ds \epsilon=\frac{\e}{\sqrt{t_0}}$, so that
%\[ v_\epsilon(x,1)=u_\e(\sqrt{t_0} x,t_0),\]
%$v_\epsilon(x,t)$ verifies $(PGL)_\epsilon$ and $E_\epsilon(v_\epsilon^0)= t_0^{\frac{2-d}{2}} E_\e(u_\e^0).$
%Letting $\e_m \rightarrow 0$, so does $\ds \epsilon_m=\frac{\e_m}{\sqrt{t_0}}$, and
%\[  \Sigma_{\mu}^t(u)= t_0^{\frac{1}{2}}  \Sigma_{\mu}^{\frac{t}{t_0}}(v) . \]
%(with obvious notations), in particular for $t=t_0$,
%\[ B(x_0,1) \cap  \Sigma_{\mu}^{t_0}(u)= B(x_0,1) \cap t_0^{\frac{1}{2}} \Sigma_{\mu}^{1}(v) = \sqrt{t_0} \left( B(x_0/\sqrt{t_0}, 1/\sqrt{t_0}) \cap \Sigma_{\mu}^{1}(v) \right).\]
%By Step 1 applied to $v_\epsilon$ and the corresponding measure $\Sigma_{\mu}(v)$, we obtain
%\[ \mathcal{H}^{d-2}( B(x_0/\sqrt{t_0}, 1/\sqrt{t_0}) \cap \Sigma_{\mu}^{1}(v)) \le K \left( 1+ \frac{1}{\sqrt{t_0}} \right)^d t_0^{\frac{2-d}{2}}M_0.\] 
%so that 
%\[ \mathcal{H}^{d-2}(B(x_0,1) \cap  \Sigma_{\mu}^{t_0}(u))= t_0^{\frac{d-2}{2}} \mathcal{H}^{d-2}(\Sigma_{\mu}^{1}(v)) \le K \left( 1+ \frac{1}{\sqrt{t_0}} \right)^d M_0.\]
%If $t_0 \ge 1$, 
%and the conclusion follows.

\bigskip

(3) As $\Sigma_\mu^t$ is closed, we have the decomposition
\[ \mu_*^t =  \mu_*^t \llcorner (\m R^d \setminus \Sigma_\mu^t) + \mu_*^t \llcorner \Sigma_\mu^t. \]
On the one hand, on $\Sigma_\mu^t$, and due to the bound (2), we can use the Radon-Nikodym theorem, which provides a function $\Theta_* \in L^1_{\mr{loc}}(\q H^{d-2}(\m R^d))$ such that
\[ \mu_*^t \llcorner \Sigma_\mu^t = \Theta_* \q H^{d-2} \llcorner \Sigma_\mu^t. \]
Due to Lemma \ref{lem3}, we also have the bound
\[ \Theta_*(x,t) \le C(d) M_0 \frac{(1+t)^d}{t^{\frac{d-2}{2}}}. \]
On the other hand, for $x_0 \notin \Sigma_\mu^{t_0}$, $\Theta_{d-2}^P(\mu_*(x,t)) \to 0$ as $(x_0,t_0) \to (x,t)$ by Lemma \ref{lem6.2}. From there, and using the bound \eqref{Eq3}, there exists $\delta>0$ small enough such that
\[ \mu_*^{t-\delta^2} (B(x, \lambda(t-\delta^2) \delta)) \le \eta_2 \delta^{d-2}. \]
Then Theorem \ref{th5} applies and shows that 
\[ \mu_*^s = |\nabla \Phi_*(x,s)|^2 dx \quad \text{on } B(x_0,\delta_0/4) \times [t - \delta^2/16, t], \]
where $\Phi_*$ solves the heat equation and is smooth: this allows to define $g$ locally.
\end{proof}

%\subsubsection{Regularity of $\Sigma_\mu^t$}
\subsubsection{Lower bounds for $\Theta_{*,d-2}(\mu_*^t)$}

As already mentioned, lower bounds for $\Theta_{*,d-2}$ will play an important role for regularity issues: however, up to now we have only lower bounds for $\Theta_{d-2}^P$ (see Corollary \ref{cor1}). The next result provides the reverse inequality to (\ref{Eq3}).

\begin{prop}%{(similar to Proposition 7, Part II, \cite{Be1}).}
\label{prop7}
For almost every $t>0$, the following inequality holds
\begin{equation}
\label{eq6}
\Theta_{*,d-2}(\mu_{*}^{t},x) \ge K \Theta_{d-2}^{P}(\mu_{*},(x,t)) 
\end{equation}
for $\mathcal{H}^{d-2}$ almost every $x \in \m{R}^d$.
\end{prop}

Combining Corollary \ref{cor1} and Proposition \ref{prop7}, we are led to

\begin{cor}
\label{cor2}
For almost every $t\geq0$, 
\[ \Theta_{*,d-2}(\mu_{*}^{t},x) \ge K \eta_3(t) \quad \text{ for } \mathcal{H}^{d-2} \text{ a.e.}  \ x \in \m{R}^d. \]
\end{cor}

%\subsection{Proofs of the properties of \texorpdfstring{$\Sigma_\mu$}{Sigma mu}}
%
%The purpose of this section is to provide detailed proofs of some technical statements concerning $\Sigma_\mu$ in the introduction of Part III. More precisely, we will prove (\ref{Eq3}), Lemma \ref{lem3}, Theorem \ref{th6} and Propositions \ref{prop6}, \ref{prop7}, \ref{prop8}. We begin with a few elementary observations which we will use later in the proofs.
%

Our goal in this subsection is to prove Proposition \ref{prop7}, but we need some preliminary tools first. Our starting point is the estimate for the time derivative $\partial_t u_\e$ provided by Proposition \ref{x61}, namely
\[ \forall T>0,\forall R>0,\quad \frac{1}{|\ln \e|} \int_{B(0,R)\times [0,T]} |\partial_t u_{\e}|^2 \le C(T,R). \]
Therefore, by a diagonal extraction argument, we may assume that there exists some non negative Radon measure $\omega_{*}$ defined on $\m{R}^d \times [0,+\infty)$ such that
\[ \frac{1}{|\ln \e|} |\partial_t u_{\e}|^2 \rightharpoonup \omega_{*} \quad \text{as measures,} \]  
so that $\omega_{*}(B(0,R)\times [0,T]) \le C(T,R).$

In \cite{Be1}, it is known at this point that $\Sigma_{\mu} \subset \m{R}^d \times (0,T_f+1)$ for some constant $T_f$. In our case, we don't have the same result, and since we want to work on compact domains in $\m{R}^d \times [0,+\infty)$, we will fix $T>0$, and then intersect our sets of $\m{R}^d$ with balls $B(0,R)$.

\bigskip

\emph{Step 1.} We fix $T>0$.

We want to prove that for almost $t \in (0,T]$, inequality (\ref{eq6}) holds for every $x \in \m{R}^d$.

We define, for $\ell \in \m{N}_{*}$ and $q >0$ to be fixed later, the set
 \[A_\ell(\omega_{*})= \left\{ (x,t) \in \m{R}^d \times (0,T] : \ \limsup_{r \rightarrow 0} \ \frac{1}{r^q} \int_{B(x,\ell r) \times [t-r^2,t]} \ \omega_{*} \ge 1 \right\} \]
and its intersection with $B(0,R)$
\[ A_\ell(\omega_{*}^{R})= \left\{ (x,t) \in A_\ell(\omega_{*}):  |x| \le R \right\}. \]

The following shows that $A_\ell(\omega_{*})$ is small in some appropriate sense.

\begin{lem}
\label{lem6.3}
%{(similar to Lemma 6.3, Part II, \cite{Be1}).}
For each $\ell \in \m{N}^{*}$ and $R>0$,
\[ \mathcal{H}_{P}^{q}(A_\ell(\omega_{*}^{R})) < +\infty, \]
where $\mathcal{H}_{P}^{q}$ denotes the $q$-dimensional Hausdorff measure with respect to the parabolic distance $d_P((x,t),(x',t'))=\max (|x-x'|,|t-t'|^{\frac{1}{2}}).$
\end{lem}

\begin{proof}
Let $0<\delta \le 1$ be given, and fixed for the moment. For $(x,t) \in A_\ell(\omega_{*}^{R})$, there exists $r=r(x,t)<\delta$ such that 
\[ \int_{B(x,\ell r) \times [t-r^2,t]}  \omega_{*} \ge r^q. \]
If we denote
\[ \Gamma_{\ell}^{P}(x,t,r(x,t)) = B(x,\ell r(x,t)) \times [t-r(x,t)^2,t], \]
clearly, the union of parabolic balls $\bigcup_{(x,t)\in A_\ell(\omega_{*}^{R})} \Gamma_{\ell}^{P}(x,t,r(x,t))$ covers $A_\ell(\omega_{*}^{R})$. Notice that $\diam(\Gamma_{\ell}^{P}) \le 2 \ell r$.  Since $A_\ell(\omega_{*}^{R}) \subset B(0,R) \times [0,T]$, we may apply the Besicovitch covering theorem (see for example \cite{Fed}). There exists an integer $m(\ell,d)$ depending only on $d$ and $\ell$, and there exists a sub-covering of the form 
\[ A_\ell(\omega_{*}^{R}) \subset \bigcup_{i=1}^{m(\ell,d)}  \bigcup_{j \in J_{i}^{\delta}} \Gamma_{\ell}^{P}(x_j,t_j,r_j(x_j,t_j)), \]
where for fixed $i$, the sets $\Gamma_j= \Gamma_{\ell}^{P}(x_j,t_j,r_j(x_j,t_j))$ are disjoint. Consequently, it follows that for each $i \in \llbracket 1,m(\ell,n) \rrbracket$,
\[ \sum_{j \in J_i^{\delta}} r(x_j,t_j)^q \le \sum_{j \in J_i^{\delta}} \int_{\Gamma_j} \omega_{*} \le \int_{B(0,R+\ell) \times [0,T]} \omega_{*}
\le C(T,R,\ell). \]
Therefore,
\[\sum_{j=1}^{m(\ell,d)} \sum_{j \in J_i^{\delta}} \diam(\Gamma_j)^q \le m(\ell,d) \ell^q C(T,R,\ell). \]
Observe that the constant on the right hand side is independent of $\delta$. Hence, letting $\delta \rightarrow 0$, we obtain 
\[ \mathcal{H}_{P}^{q}(A_\ell(\omega_{*}^{R})) \le \limsup_{\delta \rightarrow 0} (\sum_{j=1}^{m(\ell,d)} \sum_{j \in J_i^{\delta}} \diam(\Gamma_j)^q) \le m(\ell) \ell^q C(T,R,\ell), \]
and the proof is complete.
\end{proof}

We fix $\ds q=d-\frac{3}{2}.$ This choice has no specific geometrical meaning, but is convenient as the following shows.

\begin{cor}%{(similar to Corollary 6.1, Part II, \cite{Be1}).}
We have
\[ \mathcal{H}^{d-1} \left(\bigcup_{\ell \in \m{N}^{*}} A_\ell(\omega_{*}) \right) = 0.\]
Hence, for almost every $t>0$,
\[ \mathcal{H}^{d-2} \left(\bigcup_{\ell \in \m{N}^{*}} A_\ell^t(\omega_{*}) \right) = 0, \]
where $A_\ell^{t}(\omega_{*})=A_\ell(\omega_{*})\cap (\m{R}^d \times \{t\})$.
\end{cor}

\begin{proof}
Since, by Lemma \ref{lem6.3}, $\mathcal{H}_{P}^{d-\frac{3}{2}}(A_\ell(\omega_{*}^{R})) <\infty$, it follows that 
\[\mathcal{H}_{P}^{d-1}(A_\ell(\omega_{*}^{R}))=0. \]
On the other hand, parabolic balls are smaller than euclidian balls of the same radius, so that the parabolic Hausdorff measure dominates the euclidian Hausdorff measure.
It follows that 
\[ \mathcal{H}^{d-1} \left( \bigcup_{\ell \in \m{N}^{*}} \bigcup_{R \in \m{N}^{*}} A_\ell^{R}(\omega_{*}) \right) = 0.\]
Since $\ds \bigcup_{R \in \m{N}^{*}} A_\ell^{R}(\omega_{*}))= A_l(\omega_{*})$, we obtain that
\[ \mathcal{H}^{d-1} \left( \bigcup_{\ell \in \m{N}^{*}} A_\ell(\omega_{*}) \right) = 0,\]
and the proof is complete.
\end{proof}

Next, we introduce the set
\[ \Omega_{\omega}= (\m{R}^d \times [0,T]) \setminus \bigcup_{\ell \in \m{N}^{*}} A_\ell(\omega_{*}). \]

\begin{lem}%{(similar to Lemma 6.4, Part II, \cite{Be1}).}
\label{lem6.4}
Let $\chi \in \q C_{c}^{\infty}(\m{R}^d)$. Then, for $(x,t) \in \Omega_{\omega}$,
\[ \frac{1}{r^{d-2}} \lim_{r \rightarrow 0} \left( \int_{\m{R}^d} \chi \left(\frac{y-x}{r} \right) d\mu_{*}^{t}(y) - \frac{1}{r^{d-2}} \int_{\m{R}^d} \chi \left(\frac{y-x}{r} \right) d\mu_{*}^{t-r^2}(y) \right) = 0. \] 
\end{lem}

\begin{proof}
We need to go back to the level of the function $u_\e$.
For $0<r<\sqrt{t}$, by Lemma \ref{lll2.1}, we have
\begin{align}
\MoveEqLeft \int_{\m{R}^d \times\{t\}} \frac{e_{\e}(u_\e)}{|\ln \e|} \chi \left(\frac{y-x}{r} \right)  dx - \int_{\m{R}^d \times\{t-r^2\}} \frac{e_{\e}(u_\e)}{|\ln \e|} \chi \left(\frac{y-x}{r} \right) dx \nonumber \\
& = - \int_{\m{R}^d \times[t-r^2,t]} \frac{|\partial_{t} u_\e|^2}{|\ln \e|} \chi \left( \frac{y-x}{r} \right) dx dt \\
& \qquad  -\frac{1}{r|\ln \e|} \int_{\m{R}^d \times[t-r^2,t]} \partial_{t} u_\e \nabla u_\e \cdot \nabla \chi \left( \frac{y-x}{r} \right)  dx dt. \label{r5}
\end{align}

Let $\ell \in \m{N}^{*}$ such that $\Supp(\chi) \subset B(\ell).$ We set $\Lambda=B(x,\ell r) \times [t-r^2,t]$, and estimate the last term in the previous identity by the Cauchy-Schwarz inequality ,
\[ \frac{1}{r|\ln \e|} \left |\int_{\Lambda} \partial_{t} u_\e \nabla u_\e.  \nabla \chi \left( \frac{y-x}{r} \right) \right| \le \left( \int_{\Lambda} \frac{|\partial_{t} u_\e|^2}{|\ln \e|} \right)^{\frac{1}{2}} \left( \int_{\Lambda} \frac{|\nabla u_\e|^2}{r^2|\ln \e|} \right)^{\frac{1}{2}} \ \|\nabla \chi\|_{\infty}. \]
We now let $\e \rightarrow 0$ in (\ref{r5}), therefore obtaining the inequality for measures
\begin{multline*} \frac{1}{r^{d-2}} \left| \int_{\m{R}^d}
\chi \left( \frac{y-x}{r} \right) (d\mu_{*}^{t}-d\mu_{*}^{t-r^2})(y) \right| \\
\le \left( \frac{1}{r^{d-2}} \int_{\Lambda}\omega_{*} +
\left( \frac{1}{r^{d-2}} \int_{\Lambda}\omega_{*} \right)^{\frac{1}{2}} \left( \frac{1}{r^{d}}
\int_{\Lambda}d\mu_{*}\right)^{\frac{1}{2}} \right) \  \|\chi\|_{\q C^1}.
\end{multline*}
Let $\psi \in \q C_{c}^{\infty}(\m{R}^d \times [0,+\infty))$ such that $\psi=1$ in $\Lambda$, and $\psi =0$ out of $B(x,2\ell r)\times [t-2r^2,t].$ 
We have 
\begin{align*}
 \frac{1}{r^{d}} \int \psi d\mu_\e &\le  \frac{1}{r^{d-2}} \int_{t-2r^2}^{t} \int_{B(x,2\ell r)} \frac{e_\e(u_\e)}{|\ln \e|} \\
  &\le  \frac{1}{r^{d}} \int_{t-2r^2}^{t} e^{\frac{1}{4}} \int_{\m{R}^d} \frac{e_\e(u_\e)(y,s)}{| \ln \e|}e^{-\frac{|x-y|^2}{16 \ell^2r^2}} dyds \\
 &\le  \frac{1}{r^{d}} \int_{t-2r^2}^{t} e^{\frac{1}{4}} \left(\frac{2\ell r}{\sqrt{s+4\ell^2r^2}}\right)^{d-2} \int_{\m{R}^d} \frac{e_\e(u_\e^0)}{|\ln \e|} e^{-\frac{|x-y|^2}{4(s+4 \ell^2r^2)}} dyds 
\end{align*}
where we used the monotonicity formula applied at point $(x,s+4\ell^2r^2)$ between $2\ell r$ and $\sqrt{s+4 \ell^2 r^2}$. From there, we infer
\[\frac{1}{r^{d}} \int \psi d\mu_\e \le e^{\frac{1}{4}} \frac{(2 \ell)^{d-2}}{(t-2r^2)^{\frac{d-2}{2}}} \left( M_0 + \frac{C_1(\sqrt{t+4 \ell^2 r^2})}{|\ln \e|} \right).\]
  %\le C(t) l^{d-2} \int_{t-2r^2}^{t} s^{\frac{2-n}{2}}\ ds. \]
%\[ \le 2r^2 C(t,|x|,l,r).\] 
Letting $\e \rightarrow 0$, we get
\begin{align*}
\frac{1}{r^{d}} \int \psi d\mu_{*} & \le M_0 e^{\frac{1}{4}} \frac{ (2\ell)^{d-2}}{(t-2r^2)^{\frac{d-2}{2}}}.
\end{align*}
%\le C(t) l^{d-2} \int_{t-2r^2}^{t} s^{\frac{2-n}{2}}\ ds
%\[ \frac{1}{r^{d-2}} \int_{\Lambda} d\mu_{*} \le
Recall that 
\begin{equation}
\label{eq7}
\text{if} \quad (x,t) \in \Omega_\omega, \quad \text{then} \quad \limsup_{r \rightarrow 0} \frac{1}{r^{d-\frac{3}{2}}} \int_{B(x,\ell r) \times [t-r^2,t]} \omega_{*} \le 1.
\end{equation}
Now we have
\begin{align*}
\MoveEqLeft \frac{1}{r^{d-2}} \int_{\Lambda}\omega_{*} + \left( \frac{1}{r^{d-2}} \int_{\Lambda}\omega_{*} \right)^{\frac{1}{2}} \left(\frac{1}{r^d}
\int_{\Lambda}d\mu_{*}\right)^{\frac{1}{2}}  \\
& =  \left(\frac{1}{r^{d-\frac{3}{2}}} \int_{\Lambda}\omega_{*}\right) \ r^{\frac{1}{2}} +  r^{\frac{1}{4}}
\left( \frac{1} {r^{d-\frac{3}{2}}}  \int_{\Lambda}\omega_{*}\right)^{\frac{1}{2}}\ \left(\frac{1}{r^d} \int_{\Lambda} d\mu_{*} \right)^{\frac{1}{2}} \\
&\le  \left(\frac{1}{r^{d-\frac{3}{2}}} \int_{\Lambda}\omega_{*}\right) \ r^{\frac{1}{2}} +  r^{\frac{1}{4}} C(\ell) M_0(t-2r^2)^{\frac{2-d}{2}}
\left( \frac{1} {r^{d-\frac{3}{2}}}  \int_{\Lambda}\omega_{*}\right)^{\frac{1}{2}}.
\end{align*}
According to (\ref{eq7}), and letting $r \rightarrow 0$, we deduce that 
\[ \int_{\Lambda}\omega_{*} + \left(\int_{\Lambda}\omega_{*}\right)^{\frac{1}{2}} \left(\frac{1}{r^2} \int_{\Lambda}d\mu_{*}\right)^{\frac{1}{2}} \rightarrow 0, \]
and the proof is complete.
\end{proof}

In Lemma \ref{lem6.4}, we assumed that $\chi$ has compact support. The following shows that the result still holds for $\ds \chi=\exp \left(-\frac{|x|^2}{4} \right)$, which is of special interest in view of the monotonicity formula.

\begin{cor}%{(similar to Corollary 6.4, Part II, \cite{Be1}).}
\label{cor6.2}
We have, for any $(x,t) \in \Omega_\omega$,
\[ \lim_{r \rightarrow 0} \left( \frac{1}{r^{d-2}} \int_{\m{R}^d} \exp \left(-\frac{|x-y|^2}{4r^2} \right) d\mu_{*}^{t}(y) - \frac{1}{r^{d-2}} \int_{\m{R}^d} \exp \left( -\frac{|x-y|^2}{4r^2} \right) d\mu_{*}^{t-r^2}(y) \right) =0. \]
In particular, for $(x,t) \in \Sigma_\mu \cup \Omega_\omega$, the following limit exists and verifies the inequality
\[ \lim_{r \rightarrow 0} \frac{1}{r^{d-2}} \int_{\m{R}^d} \exp \left(-\frac{|x-y|^2}{4r^2} \right) d\mu_{*}^{t}(y) \ge \eta_3. \]
\end{cor}

\begin{proof}

Let $\zeta$ be a smooth cut-off function such that $0 \le \zeta \le 1,\ \zeta=1$ on $B(1)$ and $\zeta=0$ outside $B(2).$ For $\ell \in \m{N}^{*}$, consider the function $\zeta_{\ell}$ defined by $\zeta_\ell (y)=\zeta(y/\ell)$, and denote $\chi_\ell$ the function defined on $\m R^d$ by
\[ \chi_\ell(y)=\exp \left( -\frac{|y|^2}{4} \right) \zeta_\ell(y). \]

We apply Lemma \ref{lem6.4} to $\chi_\ell$, so that
\begin{equation}
\label{equa6.18}
\lim_{r \rightarrow 0} \left( \frac{1}{r^{d-2}} \int_{\m{R}^d} \chi_\ell(\frac{y-x}{r}) d\mu_{*}^{t}(y) - \frac{1}{r^{d-2}} \int_{\m{R}^d} \chi_\ell(\frac{y-x}{r}) d\mu_{*}^{t-r^2}(y) \right)=0. 
\end{equation} 
On the other hand, we claim that, for every $s \in [t-r^2,t]$,
\begin{equation}
\label{equa6.19}
\frac{1}{r^{d-2}}\int_{\m{R}^d} \left( \exp \left(-\frac{|y-x|^2}{4r^2} \right)- \chi_\ell(\frac{y-x}{r}) \right)  d\mu_{*}^{s}(y) \le  K M_0 s^{\frac{2-d}{2}} \exp \left(-\frac{\ell^2}{8} \right). 
\end{equation}
Indeed, notice first that 
\begin{align*}
\exp \left( -\frac{|y-x|^2}{4r^2} \right)-\chi_\ell \left( \frac{y-x}{r} \right) &= \exp(-\frac{|y-x|^2}{4r^2}) \left(1-\zeta_\ell \left( \frac{y-x}{r} \right) \right) \\
& \le \exp \left( -\frac{|y-x|^2}{8r^2} \right) \exp \left( -\frac{\ell^2}{8} \right).
\end{align*}
Second, due to the monotonicity formula and to (\ref{eq3}), we have
\begin{align*}
\MoveEqLeft \frac{1}{(\sqrt{2}r)^{d-2}} \int_{\m{R}^d} \exp \left( -\frac{|y-x|^2}{8r^2} \right) d\mu_{*}^{s}(y)  \\
& \le \frac{1}{(s+2r^2)^{\frac{d-2}{2}}} \int_{\m{R}^d} \exp \left( -\frac{|y-x|^2}{4(s+2r^2)} \right) d\mu_{*}^{0} \le \frac{1}{(s+2r^2)^{\frac{d-2}{2}}} M_0,
\end{align*}
and the claim \eqref{equa6.19} follows. Observe that the right hand side of (\ref{equa6.19}) does not depend on $r$, for $r< \frac{1}{2}\sqrt{t}.$ Combining (\ref{equa6.18}) and (\ref{equa6.19}), we conclude that
\[ \limsup_{r \rightarrow 0} \left( \frac{1}{r^{d-2}} \int_{\m{R}^d} \exp \left(-\frac{|x-y|^2}{4r^2} \right) (d\mu_{*}^{t} - d\mu_{*}^{t-r^2})(y)\right) \le   K M_0  t^{\frac{2-d}{2}} \exp \left( -\frac{\ell^2}{8} \right). \]
Since $\ell$ was arbitrary, the conclusion follows. 
\end{proof}

We are now in position to prove Proposition \ref{prop7}.

\begin{proof}[Proof of Proposition \ref{prop7}]
For $(x,t) \in \Omega_\omega$, define
\[ \tilde{\Theta}_{d-2}(\mu_{*}^{t},x)= \lim_{r \rightarrow 0}  \frac{1}{r^{d-2}} \int_{\m{R}^d} \exp \left( -\frac{|x-y|^2}{4r^2} \right) d\mu_{*}^{t}(y)).\]
In view of Corollary \ref{cor6.2}, $\tilde{\Theta}_{d-2}(\mu_{*}^{t},x)$ exists on $\Omega_\omega$ and 
\begin{equation}
\label{equa6.20}
\tilde{\Theta}_{d-2}(\mu_{*}^{t},x)= \Theta_{d-2}^{P}(\mu_{*},(x,t)).
\end{equation}
If $(x,t) \notin \Sigma_\mu$, then $\Theta_{d-2}^{P}(\mu_{*},(x,t))=0$ so that (\ref{eq6}) is obviously verified. Therefore, we assume in the following that $(x,t) \in \Sigma_\mu \cup \Omega_\omega$. Arguing as for the claim in Corollary \ref{cor6.2}, we obtain
\[ \frac{1}{(\ell r)^{d-2}} \int_{B(x,\ell r)}d\mu_{*}^{t} \ge \frac{K}{\ell^{d-2}} \frac{1}{r^{d-2}} \int_{\m{R}^d} \exp \left(-\frac{|x-y|^2}{4r^2} \right) d\mu_{*}^{t}-K M_0 t^{\frac{2-d}{2}} \exp \left(-\frac{\ell^2}{8} \right). \]
Hence, letting $r \rightarrow 0$, and by (\ref{equa6.20}), 
\begin{equation}
\label{Eq11}
\Theta_{*,d-2}(\mu_{*}^{t},x) \ge \frac{K}{\ell^{d-2}} \left( \Theta_{d-2}^{P}(\mu_{*},(x,t))-K \ell^{d-2} \exp \left( -\frac{\ell^2}{8} \right) M_0 t^{\frac{2-d}{2}} \right).
\end{equation}
In order to obtain (\ref{eq6}), we invoke the fact that on $\Sigma_\mu$, $\Theta_{d-2}^P \ge \eta_3(t)$. We choose $\ell$ sufficiently large so that
\[ K \ell^{d-2} \exp \left( -\frac{\ell^2}{8} \right) M_0 t^{\frac{2-d}{2}} \le \frac{1}{2} \eta_3(t) \le \frac{1}{2} \Theta_{d-2}^{P}(\mu_{*},(x,t)). \]
Going back to (\ref{Eq11}), with this choice of $\ell$, we obtain
\[ \Theta_{*,d-2}(\mu_{*}^{t},x) \ge \frac{K}{2\ell^{d-2}} \Theta_{d-2}^{P}(\mu_{*},(x,t)) \]
so that inequality (\ref{eq6}) is true for almost $t \in (0,T]$.
%
%
%The result is a straightforward consequence of the fact that for all $T>0$, for almost every $t \in (0,T]$ and all $x \in \m R^d$,
%\[ \Theta_{*,d-2}(\mu_{*}^{t},x) \ge K \Theta_{d-2}^{P}(\mu_{*},(x,t)). \qedhere \]
\end{proof}

\subsubsection{$\Sigma_\mu^t$ is $(d-2)$-rectifiable}

Now, we have to deal with the diffuse part, and different kinds of arguments could then lead to regularity for $\Sigma_\mu^t$. One way is to prove the existence of the density $\Theta_{d-2}$ and then to invoke Preiss' regularity Theorem \cite{Pr87}. More precisely, we have:

\begin{prop}%{(similar to Proposition 8, Part II, \cite{Be1}).}
\label{prop8}
For almost every $t>0$,
\[ \Theta_{*,d-2}(\mu_{*}^{t},x) = \Theta_{d-2}^{*}(\mu_{*}^t,x) \ge K \eta_3(t) \quad \mathcal{H}^{d-2}\text{-a.e } x \in \Sigma_\mu^t.  \]
As a consequence, for almost every $t>0$, the set $\Sigma_\mu^t$ is $(d-2)$-rectifiable.
\end{prop}

\begin{proof}
The proof is as in \cite[Proof of Proposition 8, p.155]{Be1}. For the convenience of the reader, we provide a brief sketch.

Fix $(x,t) \in \Omega_\omega$. Denote $L^\infty_c([0,+\infty),\m R)$ the set of bounded, real valued functions with compact support; for $r >0$ and $g \in L^\infty_c([0,+\infty),\m R)$, we define
\[ I_r(g) = \frac{1}{r^{d-2}} \int g \left( \frac{|y-x|}{r} \right) d\mu_*^t(y). \]
Consider the set
\[ F = \left\{ g \in  L^\infty_c([0,+\infty),\m R) : I(g) : = \lim_{r \to 0} I_r(g) \text{ exists and is finite} \right\}. \]
Then $F$ is a vector space, and the equality of upper and lower $(d-2)$-dimensional densities is equivalent to $\m 1_{[0,1]} \in F$.

Observe that $F$ is stable by scaling $g \mapsto g_\alpha$ where $g_\alpha(x) = g(\alpha x)$ for $\alpha >0$ and that 
\begin{equation} \label{I:scaling} I(g_\alpha) = I(g) \alpha^{\frac{d-2}{2}}.
\end{equation}
The starting point is Corollary \ref{cor6.2}. It shows that $x \mapsto e^{-\alpha x^2}$ belongs to $F$. By repetitively differentiating \eqref{I:scaling} with respect to $\alpha$, one can then infer that for any $k \in \m N$, $\psi_k: x \mapsto x^{2k} e^{-x^2}$ belongs to $F$. Denote $V_m = \Span( \psi_k, k=0, \dots, m)$, and $P_m$ the $L^2$ projection on $V_m$.

Using the Hermite polynomials, one can prove that for $f \in \q C^2_c(\m R)$,
\[ \| f - P_m(f) \|_{H^1} \le \frac{C}{\sqrt{m}} (\| f'' \|_{L^2} + \| x^2 f \|_{L^2}). \]
From there, if $g \in \q C^2_c([0,+\infty))$ satisfies $g'(0) =0$, we can symmetrize it into $\tilde g \in  \q C^2_c(\m R)$, and an approximation argument shows that $g \in F$.

Then one approximates $\m 1_{[0,1]}$ by the sequence of functions $g_n \in \q C^2_c([0,+\infty))$ defined as follows. Let $\chi \in \q C^\infty(\m R)$ be non increasing, and such that $\chi(x) = 1$ for $x \le - 1/2$ and $\chi(x) =0$ for $x \ge 0$; we set for $x \ge 0$, $g_n(x) = \chi (n(x-1))$. This allows to prove that $\m 1_{[0,1]} \in F$, as desired.

Finally, we invoke Corollary \ref{cor2} for the lower bound.
\end{proof}

\subsection{Convergence in modulus outside $\Sigma_\mu$}
%{End of the proof of Theorem \texorpdfstring{\ref{thA}}{3.1}}

\begin{proof}[End of the proof of Theorem \ref{thA}]
We can now complete the proof of Theorem \ref{thA}. We gather the results obtained in Propositions \ref{prop6} and \ref{prop8}. Only two points are left to clarify. 

\medskip

The first point is that the function $\Phi_{*}$ which appears in Theorem \ref{thA} is global, whereas up to now the function $\Phi_{\dagger}$ constructed in Theorem \ref{th5} was only locally defined. Let us sketch how to overcome this problem, we refer to \cite[Part II, Section 4, p. 131]{Be1} for more details.

Fix $(x_0,t_0) \notin \Sigma_\mu$. Recall that $\Phi_\dagger$ is defined on a neighborhood $\ds \Lambda(x_0,t_0)_{\frac{1}{2}}$ of $(x_0,t_0)$ as the limit of $\ds \frac{\Phi_\e}{\sqrt{|\ln \e|}}$, where $\Phi_\e$ is the phase provided by Theorem \ref{th2}.

Now denote $\tilde \Phi_{\e,m}$ the phase provided be Theorem \ref{th3} on the compact $K_m = B(0,m) \times [1/m,m]$. Up to changing $\tilde \Phi_{\e,m}$ by a constant, we can also assume that
\begin{equation} \label{eq:Phi_x} 
\forall \e>0, \ \forall m, \quad \tilde \Phi_{\e,m} (x_0,t_0) = \Phi_\e(x_0,t_0).
\end{equation}
The bounds given in Theorem \ref{th3} (and parabolic regularization) show the existence of $\Phi_m$ such that
\[ \frac{\tilde \Phi_{\e,m}}{\sqrt{|\ln \e|}} \to \Phi_m \quad \text{in } K_{m-1}, \]
so that $\Phi_m$ also satisfies the heat equation on $K_{m-1}$. 

On the other hand,  using the bounds in Theorem \ref{th2} (and also in the remark that follows) \emph{on the gradient}, one can prove that
\[ \left\| \frac{\nabla \Phi_{\e,m}}{\sqrt{|\ln \e|}}  - \frac{\nabla \tilde \Phi_\e}{\sqrt{|\ln \e|}} \right\|_{L^p(\Lambda(x_0,t_0)_{\frac{1}{2}})} \to 0 \quad \text{as } \e \to 0. \]
Therefore, for $m$ so large that $\Lambda(x_0,t_0)_{\frac{1}{2}} \subset K_{m-1}$,
\[ \nabla \Phi_m = \nabla \Phi_\dagger \quad \text{on } \Lambda(x_0,t_0)_{\frac{1}{2}}. \]
It follows from there and \eqref{eq:Phi_x} that for $m$ large enough, all the $\Phi_m$ coincide on $\Lambda(x_0,t_0)_{\frac{1}{2}}$. By analyticity, for all $m_0$ large enough, the $(\Phi_m)_{m \ge m_0}$ coincide on $K_{m_0-1}$. Letting $m_0 \to +\infty$, we denote $\Phi^*$ their common value. By construction, it satisfies all the requirements of Theorem \ref{thA}.

\bigskip

The second and last remaining point is to show that $|u_\e(x,t)| \rightarrow 1$ uniformly on every compact subset $K \subset \m{R}^d \times (0,+\infty) \setminus \Sigma_\mu$. As $\Sigma_\mu$ is closed, up to doing finite reunions, it suffices to prove it for a family of cylinders (which cover $\m{R}^d \times (0,+\infty) \setminus \Sigma_\mu$). Therefore, we can furthermore assume that there exist $\delta >0$, $K_1\subset \m{R}^d$ convex and $K_2 \subset (0,+\infty)$ a compact interval such that
\[ \{ (x,t) : d((x,t), K) \le \delta \} \subset K_1 \times K_2 \subset  \m{R}^d \times (0,+\infty) \setminus \Sigma_\mu. \]
We have, for every $t \in K_2$, 
\[ \int_{K_1} \frac{e_\e(u_\e)(x,t)}{|\ln \e|}dx= \int_{K_1} |\nabla \Phi_{*}|^2 d\mathcal{H}^d \le C(K). \]
Hence
\[ \int_{K_1} (1-|u_\e|^2)^2(x,t)dx \le C(\q K) \e^2 |\ln \e|, \]
and 
\[\int_{K_1} |\nabla u_\e|^2(x,t)dx \le C(\q K)|\ln \e|. \] 
Let 
\[ A= \left\{(x,t) \in \q K: \forall \e >0, \ \ 1-|u_\e(x,t)|^2 \ge \kappa(\e) \right\}, \]
where $\kappa (\e) >0$ will be determined later.

We want to show that $A$ is empty, and argue by contradiction.
If not, let $(x,t) \in A$. We try to find $\gamma(\e)$ such that 
\[ \forall y \in B(x,\gamma(\e)),\quad |1-|u_\e(y)|^2| \ge \frac{\kappa(\e)}{2}. \]
We have 
\begin{align*}
|u_\e(x,t)-u_\e(y,t)| &\le \int_x^y |\nabla u_\e(t)| \le \|\nabla u_\e(t)\|_{L^2} \sqrt{|y-x|} \le C(\q K) \sqrt{|\ln \e|} \sqrt{\gamma(\e)},
\end{align*}
and
\begin{align*}
|u_\e(y,t)|^2-1 &= (|u_\e(y,t)|-1)(|u_\e(y,t)|+1) \le 2 (|u_\e(y,t)|-1) \\
&\le 2 \left( |u_\e(x,t)-u_\e(y,t))|+|u_\e(x,t)|-1\right) \\
& \le 2 \left(C(\q K) \sqrt{|\ln \e|} \sqrt{\gamma}+|u_\e(x,t)|-1 \right). 
\end{align*}
Since
\[ 1-|u_\e(x,t)| \ge \frac{\kappa(\e)}{1+|u_\e(x,t)|} \ge \frac{1}{2} \kappa (\e), \]
we deduce that
\[ |u_\e(y,t)|^2-1 \le 2C(\q K) \sqrt{|\ln \e|} \sqrt{\gamma}- \kappa (\e). \]
So we choose $\ds \gamma(\e)= \left(\frac{\kappa(\e)}{4C(K)}\right)^2 \frac{1}{|\ln \e|}$. Then, assuming that $B(x,\gamma) \subset K_1$, we have
\begin{align*}
\int_{B(x,\gamma)} (1-|u_\e(y,t)|^2)^2 dy & \ge \int_{B(x,\gamma)} \left(\frac{\kappa(\e)}{2}\right)^2 \ge \kappa(\e)^2 \frac{\omega_d}{4} \gamma(\e)^d \\
& \ge \kappa(\e)^{2d+2} \frac{1}{|\ln \e|^d} \frac{\omega_d}{4^{2d+2}\ C(\q K)^d}. 
\end{align*}
Thus,
\begin{equation} \label{eq:K_g}
C(K) \e^2 |\ln \e| \ge \kappa(\e)^{2d+2} \frac{1}{|\ln \e|^d} \frac{\omega_d}{4^{2d+2}\ C(\q K)^d}.
\end{equation}
This inequality leads us to define 
\[ \kappa(\e)=2 \left(\frac{4^{2d+2}\ C(\q K)^{d+1}}{\omega_d}  \e^2 |\ln \e|^{d+1} \right)^{\frac{1}{2d+2}}. \]
Observe that in this case, $\gamma(\e) \to 0$ as $\e \to 0$.

Let $\e_0>0$ be so small that for $\e \in (0,\e_0)$, $\gamma(\e) < \delta$. Then $B(x,\gamma(\e)) \subset K_1$, so that the computation \eqref{eq:K_g} above holds: this is a contradiction. This proves that for $\e<\e_0$, the set $A$ is empty. Arguing similarly, we obtain that the set
\[ B= \left\{(x,t) \in K: \forall \e >0, \ \ 1-|u_\e(x,t)|^2 \le - \kappa(\e) \right\} \]
is also empty. Therefore, for $\e < \e_0$,
\[ \|1-|u_\e|^2\|_{L^{\infty}(K)} \le \kappa(\e) \le C(K) \e \sqrt{|\ln e|}, \]
which means that $|u_\e(x,t)| \rightarrow 1$ uniformly on $K$.

\bigskip

The proof of Theorem \ref{thA} is complete.
\end{proof}

\section{Evolution of the limiting density}

Our goal in this section is to provide a proof for Theorem \ref{thB}.

\subsection{Mean curvature flows} \label{sec:mean_curvature}

Let us start by recalling some definitions on the mean curvature flow.

\subsubsection{The classical notion}

Let $\Sigma$ be a smooth compact manifold of dimension $k$, and $\gamma_0: \Sigma \rightarrow \m{R}^d$ ($n\ge k$) be a smooth embedding, so that $\Sigma^0=\gamma_0(\Sigma)$ is a smooth $k$-dimensional submanifold of $\m{R}^d$. The mean curvature vector at the point $x$ of $\Sigma^0$ is the vector of the orthogonal space $(T_x\Sigma^0)^\bot$ given by

\begin{equation}
\label{x20}
\vec{H}_{\Sigma^0}(x)=-\sum_{\alpha=1}^{d-k}\left(\sum_{j=1}^{k}(\tau_j \cdot \frac{\partial \nu^{\alpha}}{\partial \tau_j})\nu^{\alpha}\right)= -\sum_{\alpha=1}^{d-k}(\Div_{T_x\Sigma^0}\nu^{\alpha})\nu^{\alpha},
\end{equation}

where $(\tau_1,\dots,\tau_k)$ is an orthonormal moving frame on $T_x\Sigma^0$, $(\nu_1,\dots,\nu^{n-k})$ is an orthonormal moving frame on  $(T_x\Sigma^0)^\bot$, and $\Div_{T_x\Sigma^0}$ denotes the tangential divergence at point $x$. The integral formulation of (\ref{x20}) is given by
\begin{equation}
\label{x21}
\int_{\Sigma^0} \Div_{T_x\Sigma^0} \vec{X} d\mathcal{H}^k = - \int_{\Sigma^0} \vec{H}_{\Sigma^0} \cdot \vec{X} d\mathcal{H}^k, 
\end{equation}

for all $\vec{X} \in \q D(\m{R}^d,\m{R}^d)$.

Next, we introduce a time dependance, and consider a smooth family $\{\gamma_t\}_{t \in I}$ of smooth embeddings of $\Sigma$ in $\m{R}^d$, where $I$ denotes some open interval containing $0$. We set $\Sigma^t=\gamma_t(\Sigma)$. If $\chi$ is a smooth compactly supported function on $\m{R}^d$, a standard computation shows that

\begin{equation}
\label{x22}
\frac{d}{dt} \int_{\Sigma^t} \chi(x) d\mathcal{H}^k =  \int_{\Sigma^t} \left(-\chi(x)  \vec{H}_{\Sigma^t}(x) + P(\nabla \chi(x))\right) \cdot \vec{Y}(x) d\mathcal{H}^k, 
\end{equation}

where $\ds \vec{Y}(x)=\frac{d}{ds}\gamma_s(\gamma_t^{-1}(x))$ is the velocity vector at point $x$, and $P$ denotes the orthogonal projection on $(T_x\Sigma^t)^\bot$.

The family $(\Sigma^t)_{t \in I}$ is moved by mean curvature in the classical sense if and only if,  for all $m \in \Sigma$ and $t \in I$,
\begin{equation}
\label{x23}
\frac{d}{dt} \gamma_t(m) =  \vec{H}_{\Sigma^t}(\gamma_t(m)).
\end{equation}
In particular, if $(\Sigma^t)_{t \in I}$ is moved by mean curvature, (\ref{x22}) becomes
\begin{equation}
\label{x24}
\frac{d}{dt} \int_{\Sigma^t} \chi(x) d\mathcal{H}^k =  -\int_{\Sigma^t} \chi(x)  |\vec{H}_{\Sigma^t}(x)|^2 d\mathcal{H}^k + \int_{\Sigma^t} \nabla \chi(x) \cdot \vec{H}_{\Sigma^t}(x) d\mathcal{H}^k.
\end{equation}
Now, $\chi$ can be chosen arbitrarily, so that (\ref{x24}) is actually equivalent to (\ref{x23}).

\subsubsection{Brakke's flows}

In the attempt to extend (\ref{x23}) or (\ref{x24}) to a larger class of (less regular) objects, and in particular to extend the flow past singularities, Brakke \cite{Br} relaxed equality in (\ref{x24}), and considered instead sub-solutions, i.e. verifying the \emph{ine\-qua\-li\-ty}

\begin{equation}
\label{x25}
\frac{d}{dt} \int_{\Sigma^t} \chi(x) d\mathcal{H}^k \le  -\int_{\Sigma^t} \chi(x)  |\vec{H}_{\Sigma^t}(x)|^2 d\mathcal{H}^k + \int_{\Sigma^t} \nabla \chi(x) \cdot \vec{H}_{\Sigma^t}(x) d\mathcal{H}^k, 
\end{equation}

for all non-negative $\chi \in \q D(\m{R}^d)$. Following Brakke \cite{Br}, we are thus going to extend (\ref{x25}) to less regular objects than smooth embedded manifolds. Actually, we adopt the point of view of Ilmanen \cite{Il}, which is slightly different from Brakke's original one.

Recall that a Radon measure $\nu$ on $\m{R}^d$ is said to be $k$-rectifiable if there exists a $k$-rectifiable set $\Sigma$, and a density function $\Theta \in L^1_{\text{loc}}(\mathcal{H}^k \llcorner\Sigma)$ such that
\[ \nu=\Theta  \mathcal{H}^k \llcorner\Sigma.\]

Since $\Sigma$ is rectifiable, for $\mathcal{H}^k$-a.e. $x \in \Sigma$, there exist a unique tangent space $T_x\Sigma$ belonging to the Grassmanian $G_{n,k}$. The distributional first variation of $\nu$ is the vector-valued distribution $\delta\nu$ defined by

\begin{equation}
\label{x26}
\delta\nu(\vec{X})= \int_{\Sigma} \Div_{T_x\Sigma} \vec{X} d\nu \quad \text{ for all } \vec{X} \in \q \q D(\m{R}^d,\m{R}^d).
\end{equation}

In the case when the measure $|\delta \nu|$ is absolutely continuous with respect to $\nu$, we say that $\nu$ has a first variation and we may write
\[ \delta \nu=\vec{H}\nu,\]
where $\vec{H}$ is the Radon-Nikodym derivative of $\delta \nu$ with respect to $\nu$. In this case, formula (\ref{x26}) becomes 

\begin{equation}
\label{x27}
\int_{\Sigma^0} \Div_{T_x\Sigma} \vec{X} d\nu= - \int_{\Sigma} \vec{H} \cdot \vec{X} d\nu. 
\end{equation}

We are now in position to give the precise definition of a Brakke flow. Let $(\nu_t)_{t \ge 0}$ be a family of Radon measures on $\m{R}^d$. For $\chi \in \q C_c^2(\m{R}^d,[0,+\infty))$, we define
\begin{equation}
\label{x1}
\bar{D}_t \nu^{t_0}(\chi)=\limsup_{t \rightarrow t_0} \frac{\nu^t(\chi)-\nu^{t_0}(\chi)}{t-t_0}.
\end{equation}

If $\nu^t  \llcorner\{\chi>0\}$ is a $k$-rectifiable measure which has a first variation verifying $\chi |\vec{H}|^2 \in L^1(\nu^t)$, then we set
\[ \mathcal{B}(\nu^t,\chi) = - \int \chi |\vec{H}|^2 d\nu^t + \int \nabla \chi \cdot P(\vec{H}) d\nu^t. \]
(Here $P$ denotes $\mathcal{H}^k$-a.e. the orthogonal projection onto the tangent space to $\nu^t$.)

Otherwise we set
\[  \mathcal{B}(\nu^t,\chi)=-\infty.\]

\begin{defi}{(Brakke flow)}
Let $(\nu_t)_{t \ge 0}$ be a family of $k$-rectifiable Radon measures on $\m{R}^d$. We say that $(\nu_t)_{t \ge 0}$ is a $k$-dimensional Brakke flow if and only if
\begin{equation}
\label{x2}
\bar{D}_t \nu^t(\chi) \le \mathcal{B}(\nu^t,\chi),
\end{equation}

for every $\chi \in \q D(\m{R}^d,[0,+\infty))$ and for all $t \ge 0$.
\end{defi}

\subsection{Mean curvature flow in Brakke's formulation for $\nu_*^t$}

\subsubsection{Relating \eqref{pgl} to mean curvature flow}

The starting point of the analysis is the formal analogy of equality (\ref{x24}), namely
\[ \frac{d}{dt} \int_{\Sigma^t} \chi(x) d\mathcal{H}^k =  -\int_{\Sigma^t} \chi(x)  |\vec{H}_{\Sigma^t}(x)|^2 d\mathcal{H}^k + \int_{\Sigma^t} \nabla \chi(x) \cdot \vec{H}_{\Sigma^t}(x) d\mathcal{H}^k, \]

with the evolution of local energies for \eqref{pgl} (see equation \eqref{x30})

\begin{equation}
\label{x28}
\frac{d}{dt} \int_{\m{R}^d} \chi(x) d\mu_\e^t = - \int_{\m{R}^d \times \{t\}} \chi(x) \frac{|\partial_t u_\e|^2}{|\ln \e|} \ dx + \int_{\m{R}^d \times \{t\}} \nabla \chi(x) \frac{-\partial_t u_\e \cdot \nabla u_\e}{|\ln \e|} \ dx. 
\end{equation} 

We already know that as $\e \to 0$, $\mu_\e^t \to \mu_*^t$. Therefore, the comparison of the two formulas suggests, at least formally, that in the limit
 
\begin{equation}
\label{x3}
\omega_\e^t := \frac{|\partial_t u_\e|^2}{|\ln \e|}(x) dx \to |\vec{H}|^2 d\mu_*^t,  
\end{equation}

and

\begin{equation}
\label{x4}
\sigma_\e^t := \frac{-\partial_t u_\e.\nabla u_\e}{|\ln \e|}(x) dx \to \vec{H} d\mu_*^t.  
\end{equation}

Actually, this is a little over optimistic for two reasons. First, we have to deal with the diffuse part of the energy (this will be handled thanks to Theorem \ref{tteo1}). Second, since (\ref{x3}) involves the quadratic term $|\vec{H}|^2$, only lower semi-continuity can be expected at first sight.

\subsubsection{Convergence of \texorpdfstring{$\sigma_\e^t$}{sigma epsilon t} and decomposition of the limit}

In this section, we use the following estimate of the time derivative $\partial_t u_{\e}$, which requires some calculations, whereas it was straightforward in \cite{Be1}.

\begin{prop}
\label{x61}
For all $T>0$ and $R>0$, there holds
\[  \frac{1}{|\ln \e|} \int_{B(0,R)\times [0,T]} |\partial_t u_{\e}|^2 \le C(T,R). \]
\end{prop}

\begin{proof}
By Lemma \ref{ll2.3}, we have for all $\chi \in C_{c}^{\infty}(\m{R}^d)$, 
\[ \frac{\partial}{\partial t} \int_{\m{R}^d} \chi(x) d\mu_{\e}^{t} = -\int_{\m{R}^d \times \{t\}} \chi(x) \frac{|\partial_t u_{\e}|^2}{|\ln \e|}\ dx +  \int_{\m{R}^d \times \{t\}} \frac{D^2 \chi \nabla u_\e \cdot \nabla u_\e - \Delta \chi e_{\e}(u_\e)}{|\ln \e|} \ dx. \] 

Integrating the last equality between $0$ and $T$, we get
\begin{multline}
\label{ee17}
\int_{\m{R}^d} \chi(x) d\mu_{\e}^{T}-\int_{\m{R}^d} \chi(x) d\mu_{\e}^{0} \\
= -\int_{\m{R}^d \times [0,T]} \chi(x) \frac{|\partial_t u_{\e}|^2}{|\ln \e|}\ dx+\int_{\m{R}^d \times [0,T]} \frac{D^2 \chi \nabla u_\e \cdot \nabla u_\e - \Delta \chi e_{\e}(u_\e)}{|\ln \e|} \ dx.
\end{multline}
We choose $\chi$ such that $\chi(x)=1$ if $|x| \le R$, $\chi \ge 0$, $\Supp(\chi) \subset B(0,R+1)$ and $\|D^2 \chi\|_{\infty} \le 1$.

Then 
\begin{align}
\int_{\m{R}^d \times [0,T]} \frac{D^2 \chi \nabla u_\e \cdot \nabla u_\e - \Delta \chi e_{\e}(u_\e)}{|\ln \e|} \ dxdt
& \le C \int_{B(0,R+1)\times [0,T]} \frac{e_{\e}(u_{\e})}{|\ln \e|} \nonumber \\
& \le C(T,R), \label{ee18}
\end{align}
due to inequality \eqref{equaH0}. Plugging \eqref{ee18} into \eqref{ee17}, we get
\[ \int_{\m{R}^d \times [0,T]} \chi(x) \frac{|\partial_t u_{\e}|^2}{|\ln \e|}\ dx \le C(T,R) + \int_{\m{R}^d} \chi(x) d\mu_{\e}^{0}-\int_{\m{R}^d} \chi(x) d\mu_{\e}^{T}. \]

Let's deal with the two last terms.
Due to $H_1(M_0)$, we have 
 \[ \int_{\m{R}^d} \chi(x) d\mu_{\e}^{0} \le \int_{B(0,R+1)} d\mu_{\e}^{0} \le C(R).\]
Finally, 
\[\int_{\m{R}^d} \chi(x) d\mu_{\e}^{T} \ge 0.\]
So 
\[ \frac{1}{|\ln \e|} \int_{B(0,R)\times [0,T]} |\partial_t u_{\e}|^2 \le \int_{\m{R}^d \times [0,T]} \chi(x) \frac{|\partial_t u_{\e}|^2}{|\ln \e|}\ dx \le C(T,R). \qedhere \]
\end{proof}

\bigskip

Consider the measure $\sigma_\e=\sigma_\e^t dt$ defined on $\m{R}^d\times [0,+\infty)$. By Cauchy-Schwarz inequality,  for every $T>0$, $\sigma_\e$ is locally bounded on $\m{R}^d\times [0,T]$ uniformly in $\e>0$. Hence, passing to a further subsequence, we may assume that $\sigma_\e \rightharpoonup \sigma_*$ as measures. The Radon-Nikodym derivative of $\sigma_\e$ with respect to $\mu_\e$ verifies
\[ \frac{d|\sigma_\e|}{d\mu_\e} \le \frac{|\partial_t u_\e|}{\sqrt{e_\e(u_\e)}}.\]

Now, let $T>0$ and $R>0$ be fixed. We have
\begin{equation}
\label{x5}
\left\| \frac{|\partial_t u_\e|}{\sqrt{e_\e(u_\e)}} \right\|_{L^2(B(0,R) \times [0,T],d\mu_\e)} = \int_{B(0,R)\times[0,T]} \frac{|\partial_t u_\e|^2}{|\ln \e|} \le C(T,R),
\end{equation}
so that $\ds \frac{d|\sigma_\e|}{d\mu_\e}$ is bounded in $L^2(B(0,R)\times[0,T],d\mu_\e)$ uniformly in $\e>0$. Since it is true for all $T>0$ and $R>0$, it follows that $\sigma_*$ is absolutely continuous with respect to $\mu_*$ (see \cite[Remark 2.2]{Am} for further details). Therefore, we may write
\[ \sigma_*=\vec{h} \mu_*^t dt,\]

where $\vec{h} \in L^2_{\loc}(\m{R}^d\times[0,T], \mu_*^t dt)$ and satisfies the bound 
\[ \|\vec{h}\|_{L^2(B(0,R) \times [0,T],\mu_*^t dt)} \le C(T,R). \]
Arguing as in Theorem \ref{tteo1} and its proof, we infer

\begin{lem}
\label{x29}
The measure $\sigma_*$ can be decomposed as $\sigma_*=\sigma_*^tdt$, where for a.e. $t \ge 0$,
\[ \sigma_*^t =-\partial _t \Phi_* \cdot \nabla \Phi_* dx + \vec{h} \nu_*^t.\]
\end{lem}

\subsubsection{Mean curvature of $\nu_*^t$}

The next step will be to identify the restriction of $\vec{h}$ on $\Sigma_\mu^t$ with the mean curvature defined by (\ref{x27}), that is:

\begin{prop}
\label{x41}
For $t\ge 0$ a.e., $\nu_*^t$ has a first variation and
\[\delta \nu_*^t=\vec{h} \nu_*^t,\]
i.e. $\vec{h}$ is the mean curvature of $\nu_*^t$.
\end{prop}

\begin{proof}
Notice that we already know by Theorem \ref{tteo1} that $\nu_*^t$ is $(d-2)$-rectifiable for a.e. $t \ge 0$.

The starting point is formula (\ref{x30}). Indeed, let $\vec{X} \in \q D(\m{R}^d,\m{R}^d)$. Then for all $t \ge 0$,
\begin{equation}
\label{x31}
\frac{1}{|\ln \e|} \int_{\m{R}^d \times \{t\}} \left( e_\e(u_\e)\delta_{ij} -\frac{\partial u_\e}{\partial x_i}\frac{\partial u_\e}{\partial x_j}\right) \frac{\partial X_i}{\partial x_j} \ dx  = -\int_{\m{R}^d \times \{t\}}\vec{X} \cdot \sigma_\e^t. 
\end{equation}

Formula (\ref{x31}) is already very close to (\ref{x27}), in particular the right hand side. In order to handle the diffuse energy, as well as to interpret the left-hand side as a tangential divergence, we need to analyse the weak-limit of the stress-energy density tensor
\[ \alpha_\e^t := A_\e dx =  \left( \operatorname{Id}-\frac{\nabla u_\e \otimes \nabla u_\e}{e_\e(u_\e)}\right) d\mu_\e^t.\]
(Recall that the stress-energy matrix $A_\e$ was defined in \eqref{def:stress-energy}.)
Clearly, $|\alpha_\e^t| \le Kd\mu_\e^t$, and we may assume that
\[ \alpha_\e^t \rightharpoonup \alpha_*^t =: A  d\mu_*^t,\]

where $A$ is a $d \times d$ symmetric matrix. Since the symmetric matrix $\nabla u_\e \otimes \nabla u_\e$ is non-negative, we have
\[ A \le \Id.\]
On the other hand,
\[ \Tr(e_\e(u_\e) \Id-\nabla u_\e \otimes \nabla u_\e)=(d-2)e_\e(u_\e)+ dV_\e(u_\e).\]
Therefore, as the trace is a linear operation, we may take the limit $\e \to 0$ and obtain
\begin{equation}
\label{x32}
\Tr(A)=(d-2)+d\frac{dV_*}{d\mu_*}, 
\end{equation}
where the non-negative measure $V_*$ is the limit (up to a possibly further subsequence) of $V_\e(u_\e)/|\ln \e|$.

Going to the limit in (\ref{x31}), and using the decomposition in Theorem \ref{tteo1}, we obtain for $t\ge 0$ a.e.,
\begin{align}
\MoveEqLeft \int_{\m{R}^d}A^{ij} \frac{\partial X_i}{\partial x_j} d\nu_*^t + \int_{\m{R}^d} \left( \frac{|\nabla \Phi_*|^2}{2} \delta_{ij} -\frac{\partial\Phi_*}{\partial x_i}\frac{\partial\Phi_*}{\partial x_j}\right) \frac{\partial X_i}{\partial x_j} dx \nonumber \\
& = -\int_{\m{R}^d }\vec{X} \cdot  \vec{h}d\nu_*^t-\int_{\m{R}^d }\vec{X} \cdot \nabla \Phi_* \partial_t \Phi_* dx. \label{x33}
\end{align}

On the other hand, $\Phi_*$ verifies the heat equation

\begin{equation}
\label{x34}
\frac{\partial \Phi_*}{\partial t}-\Delta \Phi_*=0.
\end{equation}

Multiplying (\ref{x34}) by $\vec{X}.\nabla \Phi_*$, we obtain

\begin{equation}
\label{x35}
\int_{\m{R}^d} \left( \frac{|\nabla \Phi_*|^2}{2} \delta_{ij} -\frac{\partial\Phi_*}{\partial x_i}\frac{\partial\Phi_*}{\partial x_j}\right) \frac{\partial X_i}{\partial x_j} \ dx  = -\int_{\m{R}^d }\vec{X} \cdot \nabla \Phi_* \partial_t \Phi_* \ dx. 
\end{equation}

Combining (\ref{x33}) and (\ref{x35}) we have therefore proved

\begin{lem}
\label{x36}
For $t\ge 0$ a.e., and for every $\vec{X} \in \q D(\m{R}^d,\m{R}^d)$,
\begin{equation}
\label{x37}
\int_{\m{R}^d}A^{ij} \frac{\partial X_i}{\partial x_j} d\nu_*^t  = -\int_{\m{R}^d }\vec{X} \cdot \vec{h}d\nu_*^t. 
\end{equation}
\end{lem}

Recall that we already know that $\Sigma_\mu^t$ is rectifiable for a.e. $t\geq0$. Comparing (\ref{x37}) with (\ref{x27}) in order to identify $\vec{h}$ with the mean curvature of $\nu^t$, we merely have to prove that the matrix $A$ corresponds to the orthogonal projection $P$ onto the tangent space $T_x\Sigma_\mu^t$. We first have

\begin{lem}[{\cite[Lemma 6]{Be1}}]
\label{x38}
For $t\ge 0$ a.e.,
\begin{equation}
\label{x39}
A(x) \left( \int_{T_x\Sigma_\mu^t} \nabla \chi (y) d\mathcal{H}^{d-2}(y)\right)=0 \quad \text{ for } \mathcal{H}^{d-2}\text{-a.e. }  x \in \Sigma_\mu^t \ , 
\end{equation}
and for all $\vec{X} \in \q D(\m{R}^d,\m{R})$.
\end{lem}

A straightforward consequence is 
\begin{cor}
\label{x42}
For $t$ and $x$ as in Lemma \ref{x38},
\[ \left(T_x\Sigma_\mu^t\right)^\bot \subseteq \ker A(x).\]
\end{cor}

With a little more elementary linear algebra, we further deduce

\begin{cor}[{\cite[Corollary 4]{Be1}}]
\label{x40}
For $t$ and $x$ as in Lemma \ref{x38}, $A=P$ is the orthogonal projection onto the tangent space $T_x\Sigma_\mu^t$.
\end{cor}

Gathering \eqref{x37} and Corollary \eqref{x40} proves the proposition.
\end{proof}

\begin{nb}
\label{x50}
Corollaries \ref{x42} and \ref{x40} have many important consequences.
\begin{enumerate}
\item Using (\ref{x32}), we deduce that $\ds \frac{d V_*}{d\nu_*}=0$, i.e. there is only kinetic energy in the limit.
\item Let $(\tau_1,\dots,\tau_n)$ be an orthonormal frame such that $T_x\Sigma_\mu^t =\Span (\tau_3,\dots,\tau_n)$. In view of the expression of the stress-energy tensor in these coordinates, we infer that the energy concentrates in the $(\tau_1,\tau_2)$ plane (i.e. $(T_x\Sigma_\mu^t)^{\bot}$) and uniformly with respect to the direction. In particular, since $\sigma_\e^t$ is colinear to $\nabla u_\e$, this suggests strongly that $\vec{h}$ is perpendicular to $T_x\Sigma_\mu^t$. Such an argument is made rigorous in \cite[Proposition 6.2]{Am}. 
\end{enumerate}
\end{nb}

\subsubsection{Semi-continuity of $\omega_\e^t$.}

It solely remains to show \eqref{x2} to complete the proof of Theorem \ref{thB}.
We now prove that for a.e. $t\ge 0$,
\[ \liminf_{\e \to 0} \int_{\m{R}^d \times \{t\}} \chi \frac{|\partial_t u_\e|^2}{|\ln \e|} \ge \int_{\m{R}^d \times \{t\}} \chi |\vec{h}|^2 d\nu_*^t + \int_{\m{R}^d \times \{t\}} \chi |\partial_t \Phi_*|^2\ dx.\]
We recast the problem in the framework of Young measures, which turns out to be an appropriate concept to analyse the energies of the oscillations. In this direction, denote $\ds p_\e=-\frac{\nabla u_\e}{|\nabla u_\e|}$, and consider the measure (defined on $\m{R}^d \times \m{R}^{2d}$)
\[ \tilde{\omega}_\e^t=\delta_{p_\e(x)}\frac{|\partial_t u_\e.p_\e|^2}{|\ln \e|} \ dx.\]

Extracting possibly a further subsequence, we may assume that $\tilde{\omega}_\e^t dt\to \tilde{\omega}_*$ as measures. Arguing as in Theorem \ref{thA} and its proof once more,

\begin{lem}
\label{x43}
The measure $\tilde{\omega}_*$ decomposes as $\tilde{\omega}_*=\tilde{\omega}_*^tdt$, and for $t\ge 0$ a.e.
\[ \tilde{\omega}_*^t=\Pi_{*,x}^t(p)|\partial_t \Phi_*|^2\ dx + \mathcal{M}_*^t,\]
where $ \Pi_{*,x}^t(p)$ is a measure on $\m{R}^{2d}$ (with support on the unit ball) and $\mathcal{M}_*^t=\tilde{\omega}_*^t\llcorner \Sigma_\mu^t$. Moreover, $\Pi_{*,x}^t(p)(\m{R}^{2d})=1$.
\end{lem}

The main ingredient that we will borrow directly from the analysis by Ambrosio and Soner \cite{Am} can be formulated as follows.

\begin{prop}[{\cite[Section 6]{Am}}]
\label{x44}
For $t\ge 0$ a.e., and every $\chi \in \q D(\m{R}^d)$,
\[ \int_{\m{R}^d\times\m{R}^{2n}} \chi(x) \mathcal{M}_*^t(x,p) \ge \int_{\m{R}^d} \chi |\vec{h}|^2 d\nu_*^t.\]
\end{prop}

At this stage, we are in position to complete the proof of Theorem \ref{thB}.

\begin{proof}[Proof of Theorem \ref{thB}.]
In view of Theorem 4.4 in \cite{Am}, it suffices to establish the integral version of (\ref{x2}). Let $0<T_0<T_1$. We integrate (\ref{x28}) on $[T_0,T_1]$, and let $\e$ go to zero. Combining the results of Lemma \ref{x29}, Proposition \ref{x41}, Lemma \ref{x43}, Proposition  \ref{x44}, Theorem \ref{tteo1} and the Remark, we obtain
\begin{align}
\MoveEqLeft \nu_*^{T_1}(\chi)-\nu_*^{T_0}(\chi) + \int_{\m{R}^d\times\{T_1\}} \chi |\nabla \Phi_*|^2\ dx-\int_{\m{R}^d\times\{T_0\}} \chi |\nabla \Phi_*|^2\ dx \nonumber \\
& \le -\int_{\m{R}^d\times[T_0,T_1]} \chi |\vec{h}|^2 d\nu_* + \int_{\m{R}^d\times[T_0,T_1]} \nabla \chi P(\vec{h}) d\nu_* \nonumber \\
& \quad - \int_{\m{R}^d\times[T_0,T_1]} \chi |\partial_t \Phi_*|^2\ dxdt + \int_{\m{R}^d\times[T_0,T_1]} \nabla \chi \nabla \Phi_* \partial_t \Phi_*. \label{x45}
\end{align}

Since $\Phi_*$ verifies the heat equation, we have the identity
\begin{multline}
\int_{\m{R}^d\times\{T_1\}} \chi |\nabla \Phi_*|^2\ dx - \int_{\m{R}^d\times\{T_0\}} \chi |\nabla \Phi_*|^2\ dx \\
= \int_{\m{R}^d\times[T_0,T_1]} \chi |\partial_t \Phi_*|^2\ dxdt + \int_{\m{R}^d\times[T_0,T_1]} \nabla \chi \nabla \Phi_* \partial_t \Phi_*. \label{x51}
\end{multline}

Combining (\ref{x45}) and (\ref{x51}) we obtain
\[ \nu_*^{T_1}(\chi)-\nu_*^{T_0}(\chi) \le -\int_{\m{R}^d\times[T_0,T_1]} \chi |\vec{h}|^2 d\nu_* + \int_{\m{R}^d\times[T_0,T_1]} \nabla \chi P(\vec{h}) d\nu_*.\]

As mentioned above, this integral formulation implies (\ref{x2}), under suitable assumptions which are fulfilled here, namely rectifiability of $\Sigma_\mu^t$, lower bounds on the density $\Theta_*$, and orthogonality of the mean curvature $\vec{h}$ with $(T_x\Sigma_\mu^t)^{\bot}$.
The proof of Theorem \ref{thB} is complete.
\end{proof}

\section{Point vortices in two space dimensions}
%[Proof of Theorem \ref*{theo3.1}]{Proof of Theorem \ref{theo3.1}}

Our goal in this last section is to prove Theorem \ref{theo3.1}. We start with the derivation of some pointwise estimates, which actually hold in any dimension $d \ge 2$, and refine the estimates of Theorem \ref{th2}. Indeed, we need a careful analysis on the set where $|u_\e|$ is far from zero. For this purpose, we consider, for $T>0,\ \Delta T>0,\ R>0$ given, the cylinder

 \[ \Lambda=B(x_0,R) \times [T,T+\Delta T] \subset \m{R}^2 \times [0,+\infty), \]
 
and we assume that for some constant $0<\sigma<\frac{1}{2}$,
\begin{equation}
\label{equat2.22}
|u_\e| \ge 1-\sigma \quad \text{on } \Lambda. 
\end{equation}
In particular, we may write 
\[ u_\e=\rho_\e \exp(i \phi_\e) \quad \text{on } \Lambda, \]
where $\rho_\e=|u_\e|$ and where $\phi_\e$ is a smooth real-valued map on $\Lambda$. Set
\[ \Lambda_{\alpha}= B(x_0,\alpha R) \times [T+(1-\alpha^2)\Delta T,T+\Delta T].\]
The following higher-order regularity for $u_\e$ holds.

\begin{thm}%{(similar to Theorem 2.1, \cite{Be2}).}
\label{theo2.1}

Assume that (\ref{equat2.22}) holds. There exist constants $\ds 0<\sigma_0 \le \frac{1}{2}$ and $0<\alpha,\beta<1$ depending only on the dimension $d$, such that if $\sigma<\sigma_0$, then
\begin{gather}
\label{equat2.23}
\|\nabla \phi_\e\|_{L^{\infty}(\Lambda_{\frac{3}{4}})} \le C(\Lambda) \sqrt{M_0|\ln \e|}, \\
\label{equat2.24}
\|1-\rho_\e\|_{L^{\infty}(\Lambda_{\frac{1}{2}})} \le C(\Lambda) \e^2 \left( 1+\|\nabla \phi_\e\|_{L^{\infty}(\Lambda_{\frac{3}{4}})}^2 \right), \\
\label{equat2.25}
\| \partial_t \rho_\e\|_{\q C^{0,\alpha}(\Lambda_{\frac{1}{2}})} + \| \nabla \rho_\e\|_{\q C^{0,\alpha}(\Lambda_{\frac{1}{2}})} \le C(\Lambda)M_0 \e^{\beta}.
\end{gather}
In addition, there exists a real-valued function $\Phi_\e$ defined on $\Lambda_{\frac{1}{2}}$, and satisfying the heat equation, such that
\begin{equation}
\label{equat2.26}
\| \partial_t \phi_\e - \partial_t \Phi_\e\|_{\q C^{0,\alpha}(\Lambda_{\frac{1}{2}})} + \| \nabla \phi_\e - \nabla \Phi_\e \|_{\q C^{0,\alpha}(\Lambda_{\frac{1}{2}})} \le C(\Lambda) M_0 \e^{\beta}.
\end{equation}

\end{thm}

\begin{proof}
The proof is similar to the one given in \cite[Theorem 2.1]{Be2} since we have on $\Lambda$ the same bounds on the energy  $\ds \int_{\Lambda} e_\e(u_\e) \le  C(\Lambda) M_0 |\ln \e|$. It relies in particular (and improves) on Theorem \ref{th2}.
\end{proof}

From now on and throughout the rest of this section, we work in dimension $d=2$. 
In order to prove Theorem \ref{theo3.1}, let us first recall Theorem \ref{thA}. In dimension $2$, it asserts that $\Sigma_\mu^t$ has locally finite $\q H^0$ measure, i.e. $\Sigma_\mu^t$ is a discrete set. In particular, we can write
\[ \Sigma_\mu^t =  \{ b_i(t) : i \in I \}, \quad \text{where } I \text{ is finite or } I = \m N. \]
The bound on $\q H^0(B(x,1) \Sigma_\mu^t)$ translates to $\Card(B(x,1) \cap \{ b_i(t) :  i \in I\} \le K M_0$.

Also,  $\Sigma_\mu = \bigcup_{t>0} \Sigma_\mu^t$ is a closed set in $\m{R}^2\times (0,+\infty)$ and $|u_\e| \rightarrow 1$ locally uniformly on $\m{R}^2\times (0,+\infty) \setminus \Sigma_\mu$. Moreover, a.e. $t \ge 0$,
\[ \mu_*^t=\frac{|\nabla \Phi_*|^2}{2}(\cdot,t)dx +\nu_*^t,\quad \text{where} \quad  \nu_*^t =\sum_{i \in I} \sigma_i(t) \delta_{b_i(t)}, \]
where the function $\Phi_*$ satisfies the heat equation on $\m{R}^2\times (0,+\infty)$ and we have the dichotomy:
\begin{equation}
\label{zz27}
\forall i \in I, \forall t \ge 0, \quad \text{either} \quad \sigma_i(t) \ge \eta_1 \quad \text{or} \quad \sigma_i(t)=0.
\end{equation}

\bigskip

Off the singular set $\Sigma_\mu^t$, the main contribution to the time derivative $\partial_t u_\e$ stems from the phase $\Phi_\e$. In this direction, the following proposition is motivated by Lemma \ref{lll2.1}.

\begin{prop}%{(similar to Proposition 3.1, \cite{Be2}).}
\label{pp3.1}
We have, as $\e \rightarrow 0$,
\[ \frac{|\partial_t u_\e|^2}{|\ln \e|} \rightarrow |\partial_t \Phi_{*}|^2 \quad \text{in } \q C_{\loc}^0(\m{R}^2 \times (0,+\infty) \setminus \Sigma_\mu), \]

\[ \frac{\partial_t u_\e.\nabla u_\e}{|\ln \e|} \rightarrow \partial_t \Phi_{*}.\nabla \Phi_{*} \quad \text{in } \q C_{\loc}^0(\m{R}^2 \times (0,+\infty) \setminus \Sigma_\mu). \]
\end{prop}

\begin{proof}
As $\Sigma_\mu$ is open, it suffices to show uniform convergence on small cylinders (which cover $\m{R}^2 \times (0,+\infty) \setminus \Sigma_\mu$). This is then an immediate consequence of Theorem \ref{theo2.1}, which shows that
\[ \partial_t u_\e = i \exp(i \phi_\e) \partial_t \Phi_\e + O(\e^\beta), \quad \text{and} \quad \nabla u_\e = i \exp(i \phi_\e) \nabla \Phi_\e + O(\e^\beta). \qedhere \]

\end{proof}

We now need to establish some asymptotics for the measures
\[ \frac{|\partial_t u_\e|^2}{|\ln \e|}\ dxdt  \quad \text{and} \quad \frac{\partial_t u_\e.\nabla u_\e}{|\ln \e|}\ dxdt .\]

For the first one, it suffices to have the inequality
\begin{equation}
\label{eee3.6}
\liminf_{\e \rightarrow 0} \int_{\m{R}^2 \times [0,+\infty)} \frac{|\partial_t u_\e|^2}{|\ln \e|} \chi(x)\ dxdt \ge \int_{\m{R}^2 \times [0,+\infty)} |\partial_t \Phi_{*}|^2 \chi(x)\ dxdt, 
\end{equation}

which is a straightforward consequence of Proposition \ref{pp3.1}. For the second one, we need to care a little bit more. We have

\begin{lem}%{(similar to Lemma 3.1, \cite{Be2}).}
\label{lemm3.1}
Extracting possibly a further subsequence, 

\begin{equation}
\label{equat3.7}
\sigma_\e := - \frac{\partial_t u_\e.\nabla u_\e}{|\ln \e|}\ dxdt \rightharpoonup \sigma_{*} =: - \partial_t \Phi_{*}.\nabla \Phi_{*}dxdt +h\nu_{*},
\end{equation}

weakly as measures on $\m{R}^2 \times [0,+\infty)$, where $\nu_{*}=\nu_{*}^t dt= \mu_{*} \llcorner \Sigma_\mu$ and  $h \in L^2(\nu_{*})$.
\end{lem}

\begin{proof}
Let $R>0$, and  $T>0$. We have
\begin{equation}
\label{Equat1}
\int_{B(0,R) \times [0,T]} d\sigma_\e \le \left( \int_{B(0,R) \times [0,T]} \frac{|\partial_t u_\e|^2}{|\ln \e|}\right)^{\frac{1}{2}}
\left( \int_{B(0,R) \times [0,T]} \frac{|\nabla u_\e|^2}{|\ln \e|}\right)^{\frac{1}{2}} \le C(T,R). 
\end{equation}

In view of (\ref{Equat1}), and by an argument of diagonal extraction, we see that there exists a Radon measure, defined on $\m{R}^2 \times [0,+\infty)$, bounded on compacts, such that, up to a subsequence,
\[ \sigma_\e \rightharpoonup \sigma_{*} \text{ as measures in } \m{R}^2 \times [0,+\infty). \]
We claim that $\sigma_{*}$ is absolutely continuous with respect to $\mu_{*}$. In order to prove this, we compute the Radon-Nikodym derivative of $\sigma_\e$ with respect to $\mu_\e$, obtaining

\begin{equation}
\label{equat3.8}
\left|\frac{d\sigma_\e}{d\mu_\e}\right| \le \frac{|\partial_t u_\e|.|\nabla u_\e|}{e_\e(u_\e)},
\end{equation}
and therefore for all $T>0$ and $R>0$,
\begin{equation}
\label{equat3.9}
\int_{B(0,R) \times [0,T]}
\left|\frac{d\sigma_\e}{d\mu_\e}\right|^2 \le
\int_{B(0,R) \times [0,T]} \frac{|\partial_t u_\e|^2}{|\ln
  \e|}dxdt \le C(T,R).
\end{equation}
Invoking a result of Reshetnyak \cite{Re} as in \cite{Be2}, the claim is proved.

It follows from Proposition \ref{pp3.1} that on $\m{R}^2 \times [0,+\infty) \setminus \Sigma_\mu$, $\sigma_{*}=- \partial_t \Phi_{*}.\nabla \Phi_{*} dxdt$ and the conclusion follows.
\end{proof}

In the same spirit, we have

\begin{lem} %{(similar to Lemma 3.2, \cite{Be2}).}
\label{lemm3.2}
Extracting possibly a further subsequence, 

\begin{equation}
\label{equat3.10}
\frac{A_\e}{|\ln \e|}\ dxdt \rightharpoonup A_{*} =: T(\Phi_{*})dxdt +B\nu_{*},
\end{equation}

weakly as measures on $\m{R}^2 \times [0,+\infty)$, where $T$ is defined in \eqref{equa2.4} and $B \in L^{\infty}(\nu_{*})$.
\end{lem}

The proof is identical to the proof of Lemma \ref{lemm3.1}. The next result expresses the fact that the points have "zero mean curvature".

\begin{prop}%{(similar to Proposition 3.2, \cite{Be2}).}
\label{pp3.2}
The vector $h$ and the matrix $B$ given above are identically equal to zero.
\end{prop}

\begin{proof}
We rely again on equation \eqref{x30}: given a smooth (time independent) vector field $\vec X \in \q D(\m R^2, \m R^2)$, integrate \eqref{x30} on the time interval $[T_1,T_2]$ for $0 < T_1 < T_2$. Taking the limit $\e \to 0$, we get
\[ \int_{\m R^2 \times [T_1,T_2]} (A_*)_{ij} \partial_{j} X_i  = \int_{\m R^2 \times [T_1,T_2]} \vec X \cdot \sigma_*. \]
From \eqref{equat3.10}, and as $\Phi_*$ verifies the heat equation, we infer
\[ \int_{\m R^2 \times [T_1,T_2]} B_{ij} \partial_{j} X_i  d\nu_* = \int_{\m R^2 \times [T_1,T_2]} \vec X \cdot h d\nu_*. \]
Therefore letting $T_1, T_2 \to t$, we infer that a.e $t \ge 0$,
\[  \int_{\m R^2} B_{ij}(t) \partial_{j} X_i dx = \int_{\m R^2} \vec X \cdot h d\nu_*^t. \]
Now the support of $\nu_*^t$ is a discrete set: by using cut-off vector fields, we see that $h=0$, and from there, $B=0$.
\end{proof}

\begin{proof}[Proof of Theorem \ref{theo3.1}]

We claim that for any function $\chi \ge 0$ compactly supported on $\m{R}^2$, we have for a.e. $t>0$,
\begin{equation}
\label{equat3.15}
\frac{d}{dt}\int_{\m{R}^2 \times\{t\}} \chi d\nu_{*}^t \le 0.
\end{equation}

Indeed, passing to the limit in (\ref{equa2.1}) and using (\ref{eee3.6}), Lemma \ref{lemm3.1}, Lemma \ref{lemm3.2} and Proposition \ref{pp3.2}, we obtain
\begin{equation}
\label{equat3.16}
\frac{d}{dt}\int_{\m{R}^2 \times\{t\}} \frac{|\nabla \Phi_{*}|^2}{2} \chi(x) dx + \frac{d}{dt}\int_{\m{R}^2 \times\{t\}} \chi d\nu_{*}^t \le - \int_{\m{R}^2 \times\{t\}}|\partial_t \Phi_{*}|^2 \chi - \partial_t \Phi_{*}.\nabla \Phi_{*}.\nabla \chi dx. 
\end{equation}
On the other hand, since $\Phi_{*}$ solves the heat equation, we have
\begin{align*}
\MoveEqLeft \frac{d}{dt}\int_{\m{R}^2 \times\{t\}} \frac{|\nabla \Phi_{*}|^2}{2} \chi(x) dx = \int_{\m{R}^2 \times\{t\}} \nabla(\partial_t \Phi_{*}).\nabla \Phi_{*} \chi \\
& =- \int_{\m{R}^2 \times\{t\}} (\partial_t \Phi_{*}.\Delta \Phi_{*} \chi-\partial_t \Phi_{*}.\nabla \Phi_{*}.\nabla \chi) = - \int_{\m{R}^2 \times\{t\}} (|\partial_t \Phi_{*}|^2 \chi-\partial_t \Phi_{*}.\nabla \Phi_{*}.\nabla \chi),
\end{align*}
so that (\ref{equat3.15}) follows. We deduce that
\begin{gather} \label{equat3.17}
\nu_{*}^{t_1} \le \nu_{*}^{t_0} \text{ for any } 0<t_0\le t_1.
\end{gather}
%conclusion of Theorem \ref{theo3.1} 
It then follows as an easy consequence that: first, the $b_i$ do not move; and second, the $\sigma_i(t)$ are non increasing.
% in view of (\ref{equat3.15}) and the uniform bound $l \le KM_0$.
\end{proof}

\bigskip

\small

\textsc{Delphine Côte} \\
Laboratoire Jacques-Louis Lions \\
Université Pierre et Marie Curie \\
4 place Jussieu, 75005 Paris, France

\bigskip

\textsc{Raphaël Côte} \\
Centre de Mathématiques Laurent Schwartz\\
CNRS et École polytechnique \\
91128 Palaiseau cedex, France


\begin{thebibliography}{99}
\bibitem{Am} \textsc{L. Ambrosio and M. Soner}. A measure theoretic approach to higher codimension mean curvature flow, \emph{Ann. Sc. Norm. Sup. Pisa, Cl. Sci.} {\bfseries 25} (1997), 27--49. 

\bibitem{BaOrWe09} \textsc{S. Baldo, G. Orlandi, and S. Weitkamp}. Convergence of minimizers with local energy bounds for the Ginzburg-Landau functionals. \emph{Indiana Math. J.} {\bfseries 58} (2009), no. 5, 2369--2408.

\bibitem{BBBO01} \textsc{F. Bethuel, J. Bourgain, H. Brezis, and G. Orlandi}. $W^{1,p}$ estimates for solutions to the Ginzburg-Landau functional with boundary data in $H^{1/2}$,  \emph{C. R. Acad. Sci. Paris I} {\bfseries 333} (2001), 1--8.

\bibitem{BeBrHe94} \textsc{F. Bethuel, H. Brezis, and F. Hélein}. Ginzburg-Landau Vortices. Progr. Nonlinear Differential Equations Appl. {\bfseries 13}, Birkhäuser, Boston, 1994.


\bibitem{BeBrOr01} \textsc{F. Bethuel, H. Brezis, and G. Orlandi}. Asymptotics for the Ginzburg-Landau equation in arbitrary dimensions. \emph{J. Funct. Anal.} {\bfseries 186} (2001), 432--520. Erratum {\bfseries 188} (2002), 548--549.

\bibitem{BeBrOr05} \textsc{F. Bethuel, H. Brezis, and G. Orlandi}. Improved estimates for the {G}inzburg-{L}andau equation: the elliptic case. \emph{Ann. Sc. Norm. Super. Pisa Cl. Sci. (5)} {\bfseries 4} (2005), no. 2, 319--355.


\bibitem{BeOr02} \textsc{F. Bethuel and G. Orlandi}. Uniform estimates for the parabolic Ginzburg-Landau equation, \emph{ESAIM, Control Optim. Calc. Var.} {\bfseries 8} (2002), 219--238.

\bibitem{Be1} \textsc{F. Bethuel, G. Orlandi, and D. Smets}. Convergence of the parabolic Ginzburg-Landau equation to motion by mean curvature. \emph{Ann. Math. (2)} {\bfseries 163} (2006), no. 1, 37--163. 

\bibitem{Be2} \textsc{F. Bethuel, G. Orlandi, and D. Smets}. Collisions and phase-vortex interactions in dissipative Ginzburg-Landau dynamics. \emph{Duke Math. J.} {\bfseries 130} (2005), no. 3, 523--614.

\bibitem{Be4} \textsc{F. Bethuel, G. Orlandi, and D. Smets}. Quantization and motion law for Ginzburg-Landau vortices. \emph{Arch. Ration. Mech. Anal.} {\bfseries 183} (2007), no. 2, 315--370.

\bibitem{Be5} \textsc{F. Bethuel, G. Orlandi, and D. Smets}. Dynamics of multiple degree Ginzburg-Landau vortices.  \emph{Comm. Math. Phys.} {\bfseries 272} (2007), no. 1, 229--261.

\bibitem{BeOrSm07} \textsc{F. Bethuel, G. Orlandi, and D. Smets}. Dynamique des tourbillons de vorticit\'e pour l'\'equation de {G}inzburg-{L}andau parabolique. \emph{S\'eminaire: \'{E}quations aux {D}\'eriv\'ees {P}artielles.
              2006--2007}. S\'emin. \'Equ. D\'eriv. Partielles. Exp. no. 18. École Polytechnique, Palaiseau, 2007.

\bibitem{Br} \textsc{K. Brakke}. The motion of a surface by its mean curvature. Princeton University Press, 1978.

\bibitem{Che} \textsc{Y. Chen and M. Struwe}. Existence and partial regularity results for the heat flow for harmonic maps. \emph{Math. Z} {\bfseries 201} (1989), 83--103.

\bibitem{DCot15} \textsc{Delphine Côte}. Unbounded solutions to defocusing parabolic systems. \emph{Diff. Int. Eq.} {\bfseries 28}  (2015), no. 9-10, 899--940.

\bibitem{Fed} \textsc{H. Federer}. Geometric Measure Theory. Springer, Berlin, 1969.

\bibitem{GiVe97} \textsc{J. Ginibre and G. Velo}. The Cauchy Problem in Local Spaces for the Complex Ginzburg–Landau Equation II. Contraction Methods. \emph{Comm. Math. Phys.} {\bfseries 187} (1997), 45--79.

\bibitem{Il} \textsc{T. Ilmanen}. Convergence of the Allen-Cahn equation to Brakke's motion by mean curvature. \emph{J. Differential Geom.} {\bfseries 38} (1993), 417--461.

\bibitem{Ilm94} \textsc{T. Ilmanen}. Elliptic Regularization and Partial Regularity for Motion by Mean Curvature, \emph{Mem. Amer. Math. Soc.} {\bfseries 108} (1994), no. 520.


\bibitem{Je} \textsc{R. L. Jerrard and H. M. Soner}. Dynamics of Ginzburg-Landau vortices. \emph{Arch. Rational Mech. Anal.} {\bfseries 142} (1998), 99--125.

\bibitem{JeSo99}  \textsc{R. L. Jerrard and H. M. Soner}. Scaling limits and regularity results for a class of Ginzburg-Landau systems, \emph{Ann. Inst. H. Poincaré Anal. Non Linéaire} {\bfseries 16} (1999), 423--466.

\bibitem{Je2} \textsc{R. L. Jerrard and H. M. Soner}. The Jacobian and the Ginzburg-Landau energy. \emph{Calc. Var. PDE} {\bfseries 14} (2002), 151--191.

\bibitem{Lin} \textsc{F. H. Lin}. Some dynamical properties of Ginzburg-Landau vortices. \emph{Comm. Pure Appl. Math.} {\bfseries 49} (1996), 323--359.

\bibitem{Lin98} \textsc{F. H. Lin}. Complex Ginzburg-Landau equations and dynamics of vortices, filaments, and codimension-2 submanifolds. \emph{Comm. Pure Appl. Math.} {\bfseries 51} (1998), 385--441.

\bibitem{LiRi99} \emph{F. H. Lin and T. Rivière}. Complex Ginzburg-Landau equation in high dimension and codimension two area minimizing currents, \emph{J. Eur. Math. Soc.} 1 (1999), 237--311; Erratum, Ibid. 2 (2000), 87--91.

\bibitem{Lin2} \textsc{F. H. Lin and T. Rivi\`ere}. A quantization property for moving line vortices. \emph{Comm. Pure Appl. Math.} {\bfseries 54} (2001), 825--850.

\bibitem{Pr87} \textsc{D. Preiss}. Geometry of measures in $\m{R}^n$: distribution, rectifiability, and densities. \emph{Ann. of Math} {\bfseries 125} (1987), 537--643.

\bibitem{Re} \textsc{Y. Reshetnyak}. Weak convergence of completely additive functions on a set. \emph{Siberian Math. J.} {\bfseries 9} (1968), 487--498.

\bibitem{San} \textsc{E. Sandier and S. Serfaty}. Gamma-convergence of gradient flow with applications to Ginzburg-Landau. \emph{Comm. Pure Appl. Math.} {\bfseries 57} (2004), 1627--1672.

\bibitem{Ser07a} \textsc{S. Serfaty}. Vortex collision and energy dissipation rates in the Ginzburg-Landau heat flow,  part I: Study of the perturbed Ginzburg-Landau equation. \emph{J. Eur. Math. Soc.} {\bfseries 9} (2007), no. 2, 177--217.

\bibitem{Ser07b} \textsc{S. Serfaty}. Vortex collision and energy dissipation rates in the Ginzburg-Landau heat flow,  part II: The dynamics. \emph{J. Eur. Math. Soc.} {\bfseries 9} (2007), no. 3, 383--426.

\bibitem{Son97} \textsc{H. M. Soner}. Ginzburg-Landau equation and motion by mean curvature. I. Convergence, and II. Development of the initial interface, \emph{J. Geom. Anal. 7} (1997), 437--475 and 477--491.


\bibitem{Spirn} \textsc{D. Spirn}. Vortex Dynamics of the Full Time-Dependent Ginzburg-Landau Equations. \emph{Comm. Pure Appl. Math.} {\bfseries 55} (2002), 537--581.


\bibitem{Spirn03} \textsc{D. Spirn}. Vortex motion law for the Schrödinger-Ginzburg-Landau equations.
\emph{SIAM J. Math. Anal.} {\bfseries 34 } (2003), no. 6, 1435--1476.


\bibitem{Str} \textsc{M. Struwe}. On the evolution of harmonic maps in higher dimensions. \emph{J. Diff Geom.} {\bfseries 28} (1988), 485--502.

\bibitem{Str94} \textsc{M. Struwe}. On the asymptotic behavior of the Ginzburg-Landau model in 2 dimensions, \emph{Differential Integral Equations} {\bfseries 7} (1994), 1613--1624; Erratum 8 (1995), 224.

\bibitem{Wang04} \textsc{C. Wang}. On moving Ginzburg-Landau filament vortices, \emph{Comm. Anal. Geom.} {\bfseries 12} (2004), 1185--1199.
\end{thebibliography}
\end{document}